%% file: conv2level.tex
\documentclass[a4paper,12pt,reqno]{amsart}
\usepackage[T1]{fontenc}
\usepackage[utf8]{inputenc}
\usepackage[english]{babel}
\usepackage[left=25mm,right=25mm,top=35mm,bottom=35mm]{geometry}
\usepackage{amsmath,amssymb,amsthm}
\usepackage{pgfplots}
\usepackage{pgfplotstable}
\usepackage[margin=0pt,font={small}]{caption}
\usepackage{booktabs}
\usepackage{slashbox} 
\usepackage{color,colortbl}
\usepackage{hyperref}
\newcommand{\est}{\eta}
\newcommand\0{\boldsymbol{0}}
\newcommand{\eps}{\varepsilon}
\newcommand{\y}{\mathbf{y}}
\newcommand{\G}{\Gamma}
\newcommand{\N}{\mathbb{N}}
\renewcommand{\P}{\mathbb{P}}
\newcommand{\R}{\mathbb{R}}
\newcommand{\V}{\mathbb{V}}
\newcommand{\X}{\mathbb{X}}
\renewcommand{\AA}{\mathcal{A}}
\newcommand{\GG}{\mathcal{G}}
\newcommand{\EE}{\mathcal{E}}
\newcommand{\MM}{\mathcal{M}}
\newcommand{\NN}{\mathcal{N}}
\newcommand{\PP}{\mathcal{P}}
\newcommand{\RR}{\mathcal{R}}
\renewcommand{\SS}{\mathcal{S}}
\newcommand{\TT}{\mathcal{T}}
\newcommand\III{\mathfrak{I}}
\newcommand\MMM{\mathfrak{M}}
\newcommand\PPP{\mathfrak{P}}
\newcommand\QQQ{\mathfrak{Q}}
\newcommand{\Cmark}{C_{\vartheta}}
\newcommand{\Cthm}{C_{\mathrm{thm}}}
\newcommand{\qlin}{q_{\mathrm{lin}}}
\newcommand{\qsat}{q_{\mathrm{sat}}}
\DeclareMathOperator*{\hull}{span}
\DeclareMathOperator*{\refine}{refine}
\DeclareMathOperator*{\supp}{supp}
\newcommand{\enorm}[3][]{#1|\!#1|\!#1|\,#2\,#1|\!#1|\!#1|_{#3}}
\newcommand{\norm}[3][]{#1\|#2#1\|_{#3}}
\newcommand{\dpi}{\mathrm{d}\pi}
\newcommand{\dx}{\mathrm{d}x}
\newcommand{\cost}{\mathsf{cost}}
\newcommand{\tol}{\mathsf{tol}}
\newcommand{\thetaX}{\theta_{\mathbb{X}}}
\newcommand{\thetaP}{\theta_{\mathbb{P}}}
\newcommand{\uref}{u_{\mathrm{ref}}}
\newcommand{\Rtildel}{\widetilde\RR_\ell}
\newcommand{\Mtildel}{\widetilde\MMM_\ell}
\pgfplotsset{
every axis/.append style={
font={\fontsize{8pt}{12pt}\selectfont},  
},
tick label style={font=\tiny},
title style={font=\tiny,yshift=-1.5ex},
xlabel style={font=\tiny,yshift=+1.0ex},
ylabel style={font=\tiny,yshift=-1.2ex},
}
\newcommand{\cellvsp}{\\[-4pt]}

\newcommand{\fontsizetwo}{0.80em}
\newcommand{\fontsizethree}{0.7em}

\newcommand{\smallfontthree}[1]{ {\fontsize{\fontsizethree}{\fontsizetwo}\selectfont{#1}} }
\definecolor{myBrown}{rgb}{0.6 0.4 0.2}
\definecolor{myOrange}{rgb}{1.0 0.6 0.2}
\definecolor{myLightGray}{RGB}{235,235,235}
\definecolor{myViolet}{RGB}{153,50,204}
\newtheorem{theorem}{Theorem}
\newtheorem{proposition}[theorem]{Proposition}
\newtheorem{lemma}[theorem]{Lemma}
\newtheorem{corollary}[theorem]{Corollary}
\newtheorem{algorithm}[theorem]{Algorithm}
\newtheorem{remark}[theorem]{Remark}
\newtheorem{marking}{Marking criterion}

\makeatletter
\def\@seccntformat#1{%
  \protect\textup{\protect\@secnumfont
    \ifnum\pdfstrcmp{subsection}{#1}=0 \bfseries\fi
    \csname the#1\endcsname
    \protect\@secnumpunct
  }%
}
\makeatother
\usepackage{fancyhdr}
\advance\footskip0.5cm
\title{Convergence of adaptive stochastic Galerkin FEM}
\author{Alex Bespalov}
\address{School of Mathematics, University of Birmingham, Edgbaston, Birmingham B15 2TT, UK}
\email{a.bespalov@bham.ac.uk}
\author{Dirk Praetorius}
\address{Institute for Analysis and Scientific Computing, TU Wien, Wiedner Hauptstra\ss{}e~8--10, 1040 Vienna, Austria}
\email{dirk.praetorius@asc.tuwien.ac.at}
\author{Leonardo Rocchi}
\address{School of Mathematics, University of Birmingham, Edgbaston, Birmingham B15 2TT, UK}
\email{lxr507@bham.ac.uk}
\author{Michele Ruggeri}
\address{Faculty of Mathematics, University of Vienna, Oskar-Morgenstern-Platz~1, 1090 Vienna, Austria}
\email{michele.ruggeri@univie.ac.at}
\subjclass[2010]{35R60, 65C20, 65N12, 65N15, 65N30}
\keywords{adaptive methods, a~posteriori error analysis, convergence, two-level error estimate,
stochastic Galerkin methods, finite element methods, parametric PDEs}
\thanks{{\em Acknowledgements.} This work was initiated and part of it was undertaken
when AB visited the Institute for Analysis and Scientific Computing at TU Wien in 2018.
This author wishes to thank
the colleagues in that Institute for hospitality and stimulating research atmosphere.
The work of AB and LR was supported by the EPSRC under grant EP/P013791/1.
The work of DP and MR was supported by the Austrian Science Fund (FWF) under grants W1245 and F65.
}
\date{\today}

\begin{document}

\begin{abstract}
We propose and analyze novel adaptive algorithms for the numerical solution
of elliptic partial differential equations with parametric uncertainty.
Four different marking strategies are employed for refinement of stochastic
Galerkin finite element approximations.
The algorithms are driven by the energy error reduction estimates derived from
two-level \textsl{a~posteriori} error indicators for spatial approximations and
hierarchical \textsl{a~posteriori} error indicators for parametric approximations.
The focus of this work is on the mathematical foundation of the
adaptive algorithms in the sense of rigorous convergence analysis.
In particular, we prove that the proposed algorithms
drive the underlying energy error estimates to zero.
\end{abstract}

\maketitle
\thispagestyle{fancy}

\section{Introduction}

The design and analysis of adaptive algorithms for the numerical solution of
partial differential equations (PDEs) with parametric or uncertain inputs
have been active research themes in the last decade.
Adaptive algorithms are indispensable when solving a particularly challenging class of parametric problems
represented by PDEs whose inputs depend (e.g., in an affine way) on infinitely many uncertain parameters.
For this class of problems, adaptive algorithms have been shown, on the one hand,
to yield approximations that are immune to the curse of dimensionality and, on the other hand,
to outperform standard sampling methods (see \cite{CohenDeVore15,CohenDeVoreSchwab10}).

It is well known in the finite element community that 
adaptive strategies based on rigorous \textsl{a~posteriori} error analysis of computed solutions
provide an effective mechanism for building approximation spaces and accelerating convergence.
Several adaptive strategies of this type have been proposed
in the context of stochastic Galerkin finite element method (sGFEM) for PDE problems
with parametric or uncertain inputs.
Typically, they are developed by extending the \textsl{a~posteriori} error estimation techniques
commonly used for deterministic problems to parametric settings.
For example, dual-based \textsl{a~posteriori} error estimates are employed in~\cite{mathelinLeMaitre2007};
implicit error estimators (in the spirit of~\cite{AinsworthOden00}) are used in~\cite{WanKarniadakis2009}
for the sGFEM based on multi-element generalized polynomial chaos expansions;
explicit residual-based \textsl{a~posteriori} error estimators provide spatial and stochastic error indicators for
adaptive refinement in~\cite{gittelson13,egsz14, egsz15};
local equilibration error estimators are utilized in~\cite{em16};
and hierarchical error estimators and the associated estimates of error reduction drive adaptive algorithms
proposed in~\cite{bs16,br18,bprr18,CrowderPowellBespalov,KhanBespalovPowellSilvester}.

In contrast to the design of algorithms, convergence analysis of adaptive sGFEM is much less developed.
In~\cite{egsz15}, convergence of the adaptive algorithm driven by residual-based error estimators
is proved in the spirit of the convergence analysis for deterministic FEM in~\cite{CKNS08};
moreover, the quasi-optimality of the generated sequence of meshes, in a suitable sense, is established.
The analysis in~\cite{egsz15}, however, requires that the adaptive algorithm enforces
additional spatial refinements during the iterations where parametric enrichment
is performed (see \cite[Section~6]{egsz15}).
This is caused by a purely theoretical artifact associated with using inverse estimates for
the residual-based error estimators (see~\cite[\S6.1]{egsz15}).

In this paper, we study convergence of adaptive algorithms
which are driven by the energy error reduction estimates derived from
two-level \textsl{a~posteriori} error indicators for spatial approximations and
hierarchical \textsl{a~posteriori} error indicators for parametric approximations.
The underlying \textsl{a~posteriori} error estimate that combines these two types of indicators has been
recently introduced and analyzed in~\cite{bprr18}.
We employ four practical marking criteria which are combinations of 
D\"orfler~\cite{doerfler} and maximum~\cite{bv84} marking strategies.
At each step, the algorithm performs either solely mesh refinement or solely polynomial enrichment.
Our central result in Theorem~\ref{thm:plain_convergence} shows that each proposed adaptive algorithm generates
a sequence of Galerkin approximations such that the corresponding sequence of energy error estimates converges to zero.
Therefore, this result provides a theoretical guarantee that, for any given positive tolerance,
the algorithms stop after a finite number of iterations.
We note in Remark~\ref{remark:main_thm} that the proof of Theorem~\ref{thm:plain_convergence} is given
for more general marking strategies, which are inspired by~\cite[\S2.2]{msv08}.
As an immediate consequence of Theorem~\ref{thm:plain_convergence}, we show that,
under the saturation assumption, the Galerkin approximations generated by the algorithms
converge to the true parametric solution (Corollary~\ref{cor:plain_convergence}).
Further to that, in the case of D\"orfler marking, we prove linear convergence of the energy error in
Theorem~\ref{thm:linear_convergence}.

We note that, although the results in this paper are presented for a simple model problem---steady-state diffusion
equation whose coefficient has affine dependence on infinitely many parameters---our analysis will apply to
more general elliptic linear problems with affine-parametric inputs
(e.g., to linear elasticity models, see~\cite{KhanBespalovPowellSilvester})
as well as in the context of goal-oriented adaptivity (see~\cite{bprr18}).

The paper is organized as follows.
Section~\ref{sec:problem} introduces the parametric model problem and its weak formulation.
In Section~\ref{sec:discretization}, we introduce the approximation spaces, define sGFEM formulations,
and recall the \textsl{a~posteriori} error estimates derived in~\cite{bprr18}.
In Section~\ref{sec:main:results}, we present adaptive algorithms with four different marking criteria
and formulate the main results of this work.
The results of numerical experiments are reported in Section~\ref{sec:numer:results}, where, in particular,
we compare the computational cost associated with employing different marking criteria.
Technical details and the proofs of theorems are given in Sections~\ref{section:plain_convergence}--\ref{section:linear_convergence}.
The results of a more extensive experimental study of the computational cost associated with
different marking criteria for a range of marking parameters are presented in Appendix~\ref{sec:appendix}.

\section{Parametric model problem} \label{sec:problem}

Let $D \subset \R^d$ ($d = 2, 3$) be a bounded Lipschitz domain with polytopal boundary $\partial D$
and let $\G := \prod_{m=1}^\infty [-1,1]$ denote the infinitely-dimensional hypercube.
We consider the elliptic boundary value problem
\begin{equation} \label{eq:strongform}
\begin{aligned}
 -\nabla \cdot (a \nabla u) &= f \quad &&\text{in } D \times \G,\\
 u & = 0 \quad &&\text{on } \partial D \times \G,
\end{aligned}
\end{equation}
where the scalar coefficient $a$ and the right-hand side function $f$ (and, hence, the solution~$u$)
depend on a countably infinite number of scalar parameters, i.e.,
$a = a(x, \y)$, $f = f(x, \y)$, and $u = u(x, \y)$ with $x \in D$ and $\y \in \G$.
For the coefficient $a$, we assume linear dependence on the parameters, i.e.,
\begin{equation} \label{eq1:a}
a(x,\y) = a_0(x) + \sum_{m=1}^\infty y_m a_m(x)
\quad \text{for } x \in D \text{ and } \y = (y_m)_{m\in\N} \in \G,
\end{equation}
whereas for the right-hand side of \eqref{eq:strongform} we assume that $f \,{\in}\, {L^2_\pi(\G;H^{-1}(D))}$.
Here, $\pi \,{=}\, \pi(\y)$ is a probability measure on $(\G,\mathcal{B}(\G))$ with
$\mathcal{B}(\G)$ being the Borel $\sigma$-algebra on $\G$, and
we assume that $\pi(\y)$ is the product of symmetric Borel probability measures $\pi_m$ on~$[-1,1]$,
i.e., $\pi(\y) = \prod_{m=1}^\infty \pi_m(y_m)$.

The scalar functions $a_m \in W^{1,\infty}(D)$ ($m \in \N_0$) in~\eqref{eq1:a} are required to satisfy the
following inequalities
\begin{equation}\label{eq2:a}
 0 < a_0^{\rm min} \le a_0(x) \le a_0^{\rm max} < \infty
 \quad \text{for almost all } x \in D
\end{equation}
and
\begin{equation}\label{eq3:a}
 \tau := \frac{1}{a_0^{\rm min}} \, \sum_{m=1}^\infty \norm{a_m}{L^\infty(D)} < 1.
\end{equation}
With the Sobolev space $\X := H^1_0(D)$, consider the Bochner space $\V := L^2_\pi(\G;\X)$.  
On $\V$, define the bilinear forms
\begin{align*}
B_0(u,v) &:= \int_\G \int_D a_0(x) \nabla u(x,\y)\cdot\nabla v(x,\y) \, \dx \, \dpi(\y),
\\
B(u,v) &:= B_0(u,v) + \sum_{m=1}^\infty \int_\G \int_D y_ma_m(x) \nabla u(x,\y)\cdot\nabla v(x,\y) \, \dx \, \dpi(\y).
\end{align*}
An elementary computation shows that assumptions~\eqref{eq1:a}--\eqref{eq3:a} ensure that the bilinear 
forms $B_0(\cdot,\cdot)$ and $B(\cdot,\cdot)$ are symmetric, continuous, and elliptic on $\V$.
Let $\enorm{\cdot}{}$ (resp., $\enorm{\cdot}{0}$) denote the norm induced by $B(\cdot,\cdot)$ (resp., $B_0(\cdot,\cdot)$).
Then, there holds
\begin{equation}\label{eq:lambda}
 \lambda \, \enorm{v}{}^2 \le \enorm{v}{0}^2 \le \Lambda \, \enorm{v}{}^2
\quad \text{for all } v \in \V,
\end{equation}
where
$0 < \lambda := \frac{a_0^{\rm min}}{a_0^{\rm max}\,(1+\tau)} < 1 <
  \Lambda := \frac{a_0^{\rm max}}{a_0^{\rm min}\,(1-\tau)} < \infty$.

The parametric problem~\eqref{eq:strongform} is understood in the weak sense:
Given $f \in L^2_\pi(\G;H^{-1}(D))$, find $u \in \V$ such that
\begin{equation} \label{eq:weakform}
B(u,v) = F(v) := \int_\G \int_D f(x,\y) v(x,\y) \, \dx \, \dpi(\y)
\quad \text{for all } v \in \V.
\end{equation}
The existence and uniqueness of the solution $u \in \V$ to~\eqref{eq:weakform} follow by the Riesz theorem.

\section{Finite element discretization and a~posteriori error analysis} \label{sec:discretization}

\subsection{Approximation spaces}

Let $\TT_\bullet$ be a mesh, i.e., a conforming triangulation of $D$ into
compact non-degenerate simplices $T \in \TT_\bullet$ (e.g., triangles for $d = 2$).
Let $\EE_\bullet$ be the corresponding set of facets (e.g., edges for $d = 2$).
Let $\EE_\bullet^{\rm int} \subset \EE_\bullet$ be the set of interior facets, i.e.,
for each $E \in \EE_\bullet^{\rm int}$, there exist unique $T, T' \in \TT_\bullet$ such that $E = T \cap T'$.
Let $\NN_\bullet$ be the set of vertices of $\TT_\bullet$. For $z \in \NN_\bullet$, let $\varphi_{\bullet,z}$
be the associated hat function, i.e., $\varphi_{\bullet,z}$ is piecewise affine, globally continuous,
and satisfies the Kronecker property $\varphi_{\bullet,z}(z') = \delta_{zz'}$ for all $z' \in \NN_\bullet$.
We consider the space of continuous piecewise linear finite elements
\begin{equation*}
 \X_\bullet := \SS^1_0(\TT_\bullet) := \{v_\bullet \in \X : v_\bullet \vert_T \text{ is affine for all } T \in \TT_\bullet \} \subset \X = H^1_0(D).
\end{equation*}
Recall that $\{ \varphi_{\bullet,z} : z \in \NN_\bullet \setminus \partial D \}$ is the standard basis of $\X_\bullet$.

Let us now introduce the polynomial spaces on $\G$.
For each $m \in \N$, let $(P_n^m)_{n\in\N_0}$ denote the sequence of univariate polynomials
which are orthogonal with respect to $\pi_m$ such that $P_n^m$ is a polynomial of degree
$n \in \N_0$ with $\norm{P_n^m}{L^2_{\pi_m}(-1,1)}=1$ and $P_0^m \equiv 1$.
It is well known that $\{P_n^m : n\in\N_0 \}$ is an orthonormal basis of $L^2_{\pi _m}(-1,1)$.
With $\N_0^\N := \{\nu = (\nu_m)_{m\in\N} : \nu_m\in\N_0 \text{ for all } m\in\N \}$
and $\supp(\nu):= \{m\in\N : \nu_m\neq 0 \}$, let $ \III := \{\nu \in \N_0^{\N} : \#\supp(\nu) < \infty \}$ be the set of finitely supported multi-indices.
Note that $\III$ is countable. With
\begin{equation*}
P_\nu(\y)
:= \prod_{m\in\N} P_{\nu_m}^m(y_m)
= \prod_{m\in\supp(\nu)} P_{\nu_m}^m(y_m)
\quad \text{for all } \nu \in \III \text{ and all } \y \in \G,
\end{equation*}
the set $\{ P_\nu : \nu\in\III \}$ is an orthonormal basis of $\P := L^2_\pi(\G)$; see~\cite[Theorem~2.12]{sg11}.

The Bochner space $\V = L^2_\pi(\G;\X)$ is isometrically isomorphic to $\X \otimes \P$ and
each function $v \in \V$ can be represented in the form
\begin{equation}\label{eq:representation}
 v(x,\y) = \sum_{\nu \in \III} v_\nu(x) P_\nu(\y)
 \quad\text{with unique coefficients }
 v_\nu \in \X.
\end{equation}
Moreover, there holds (see~\cite[Lemma~2.1]{bprr18})
\begin{equation} \label{eq1:lemma:orthogonal}
B_0(v,w) = \sum_{\nu \in \III} \int_D a_0(x) \, \nabla v_\nu(x) \cdot \nabla w_\nu(x) \, \dx
\quad \text{for all } v,w \in \V
\end{equation}
and, in particular,
\begin{equation}
\label{eq2:lemma:orthogonal}
\enorm{v}{0}^2
= \sum_{\nu \in \III} \norm{a_0^{1/2}\nabla v_\nu}{L^2(D)}^2
= \sum_{\nu \in \III} \enorm{v_\nu P_\nu}{0}^2
\quad \text{for all } v \in \V.
\end{equation}

Let $\0 = (0,0,\dots)$ denote the zero index,
and let $\PPP_\bullet \subset \III$ be a finite index set such that $\0 \in \PPP_\bullet$.
We denote by $\supp(\PPP_\bullet) := \bigcup_{\nu \in \PPP_\bullet} \supp(\nu)$ the set of active parameters in~$\PPP_\bullet$.

Our discretization of~\eqref{eq:weakform} is based on the finite-dimensional tensor-product space
\begin{equation*}
\V_\bullet
:= \X_\bullet \otimes \P_\bullet
\subset \X \otimes \P
= \V
\quad \text{with} \quad
\P_\bullet := \hull\{P_\nu : \nu \in \PPP_\bullet \} \subset \P = L^2_\pi(\G).
\end{equation*}
The Galerkin discretization of~\eqref{eq:weakform} reads as follows:
Find $u_\bullet \in \V_\bullet$ such that
\begin{equation}\label{eq:discrete_formulation}
 B(u_\bullet, v_\bullet) = F(v_\bullet) 
 \quad \text{for all } v_\bullet \in \V_\bullet. 
\end{equation}
Again, the Riesz theorem proves the existence and uniqueness of the solution $u_\bullet \in \V_\bullet$.

\subsection{Mesh refinement and parametric enrichment}

\begin{figure}[b]
\centering
\includegraphics[width=.28\textwidth]{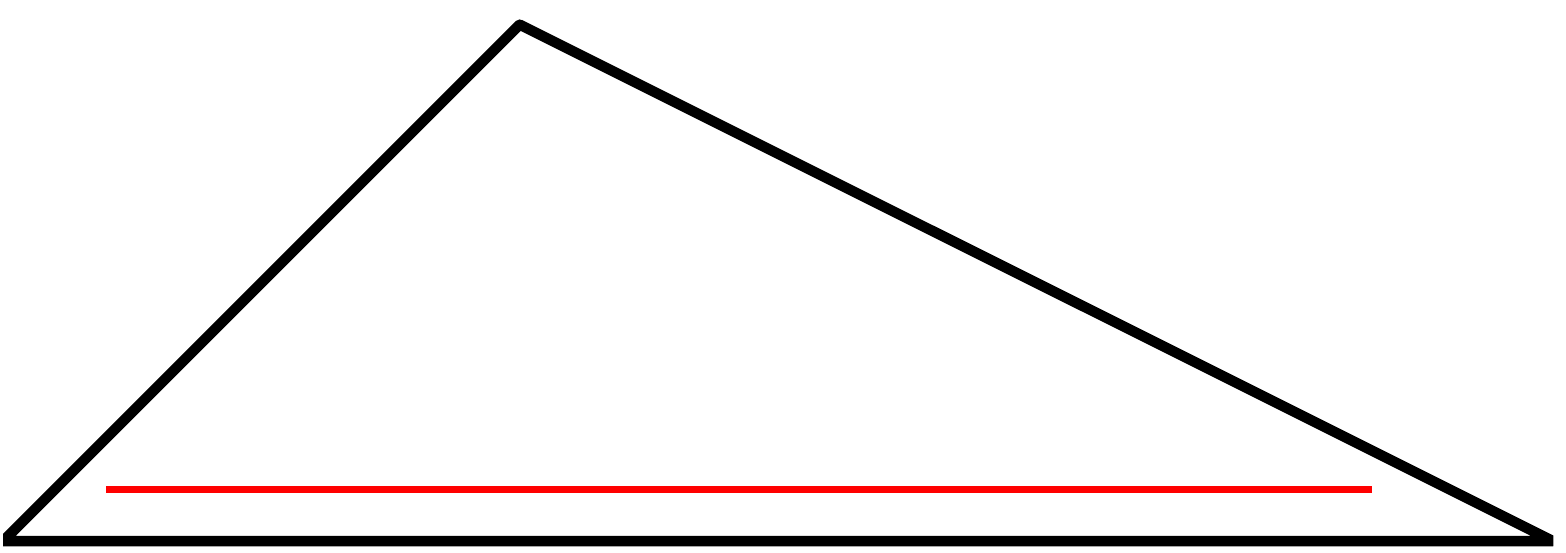} \quad
\includegraphics[width=.28\textwidth]{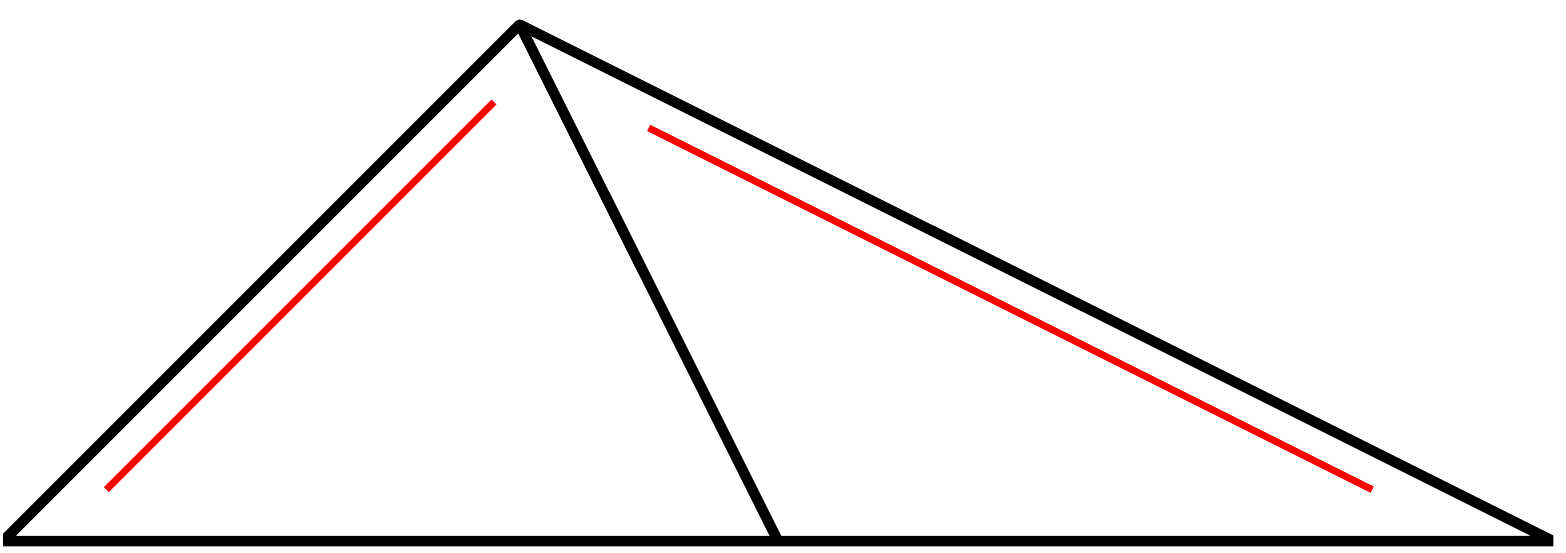} \quad
\includegraphics[width=.28\textwidth]{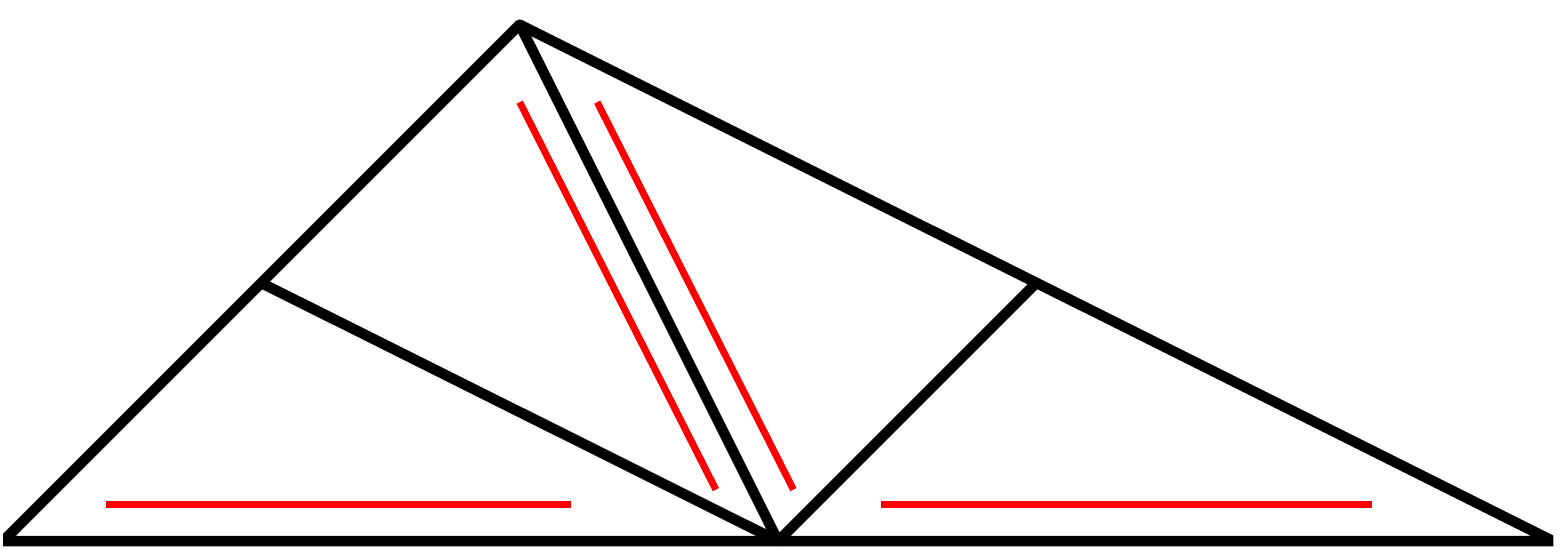} \quad
\caption{For NVB in 2D, each triangle $T\in\mathcal{T}_\bullet$ has one \emph{reference edge},
indicated by the double line (left). 
Bisection of $T$ is achieved by halving the reference edge.
The reference edges of the sons are always opposite to the new vertex (middle).
Recursive application of this rule leads to conforming meshes.
After three bisections per element all edges of a triangle are halved (right).
If all elements $T\in\mathcal{T}_\bullet$ are refined by three bisections, the resulting uniform refinement is conforming.}
\label{fig:nvb2d}
\end{figure}

For mesh refinement, we employ newest vertex bisection (NVB);
see Figure~\ref{fig:nvb2d} for $d = 2$ and, e.g., \cite[Figure~2]{egp18+} for $d=3$ as well as~\cite{stevenson,kpp}.
We assume that any mesh $\TT_\bullet$ employed for the spatial discretization
can be obtained by applying NVB refinement(s) to a given initial mesh $\TT_0$.

For a given mesh $\TT_\bullet$, let $\widehat\TT_\bullet$ be the coarsest mesh obtained from $\TT_\bullet$ such that:
(i) for $d=2$, all edges of $\TT_\bullet$ have been bisected once
(which corresponds to uniform refinement of all elements by three bisections; see Figure~\ref{fig:nvb2d});
(ii) for $d = 3$, all faces contain an interior vertex
(see~\cite[Figure~3]{egp18+} and the associated discussion therein).
Then $\widehat\NN_\bullet$ denotes the set of vertices of $\widehat\TT_\bullet$ and
$\{ \widehat\varphi_{\bullet,z} : z \in \widehat\NN_\bullet \}$ is the corresponding set of hat functions.
The finite element space associated with $\widehat\TT_\bullet$ is denoted by
$\widehat\X_\bullet := \SS^1_0(\widehat\TT_\bullet)$.
With $\NN_\bullet^+ := (\widehat\NN_\bullet \setminus \NN_\bullet) \setminus \partial D$
being the set of new interior vertices created by uniform refinement of $\TT_\bullet$,
one has $\widehat\X_\bullet = \X_\bullet \oplus \hull\{ \widehat\varphi_{\bullet,z} : z \in \NN_\bullet^+ \}$.
For a later use, we note that there exists a constant $K \ge 1$ depending only on the initial mesh $\TT_0$
such that
\begin{equation} \label{eq:const:K}
\# \{ z \in \NN_\bullet^+ : |T \cap \supp(\widehat\varphi_{\bullet,z})|>0 \} \le K < \infty\quad
\hbox{for all $T \in \TT_\bullet$}.
\end{equation}

For a set of marked vertices $\MM_\bullet \subseteq \NN_\bullet^+$,
let $\TT_\circ := \refine(\TT_\bullet,\MM_\bullet)$ be the coarsest mesh such that
$\MM_\bullet \subseteq \NN_\circ$, i.e., all marked vertices are vertices of $\TT_\circ$.
Since NVB is a binary refinement rule, this implies that $\NN_\circ \subseteq \widehat\NN_\bullet$ and $(\NN_\circ \setminus \NN_\bullet) \setminus \partial D = \NN_\bullet^+ \cap \NN_\circ$.
In particular, the choices $\MM_\bullet = \emptyset$ and $\MM_\bullet = \NN_\bullet^+$ lead to the meshes
$\TT_\bullet = \refine(\TT_\bullet,\emptyset)$ and $\widehat\TT_\bullet = \refine(\TT_\bullet,\NN_\bullet^+)$, respectively. 

Turning now to the parametric enrichment,
we follow~\cite{bs16,br18,bprr18} and consider the \emph{detail index set}
\begin{equation} \label{def:Q}
   \QQQ_\bullet := \{ \mu \in \III \setminus \PPP_\bullet : \mu = \nu \pm \eps_m
                                 \text{ for all } \nu \in \PPP_\bullet \text{ and all } m = 1,\dots, M_{\PPP_\bullet} + 1
                             \},
\end{equation}
where $\eps_m \in \III$ denotes the $m$-th unit sequence,
i.e., $(\eps_m)_i = \delta_{mi}$ for all $i \in \N$, and $M_{\PPP_\bullet} \in \N$ is given by
\begin{equation*}
M_{\PPP_\bullet} :=
\begin{cases}
0                                                                             &  \text{if $\PPP_\bullet = \{ \0 \} $}, \\
\max \{ \max (\supp (\nu) ) : \nu \in \PPP_\bullet \setminus \{ \0 \} \} & \text{otherwise}.
\end{cases}
\end{equation*}
Then an enriched polynomial space $\P_\circ$ with $\P_\bullet \subset \P_\circ \subset \P$
can be obtained by adding some marked indices $\MMM_\bullet \subseteq \QQQ_\bullet$
to the current index set $\PPP_\bullet$, i.e.,
$\P_\circ := \hull\{P_\nu : \nu \in \PPP_\circ \}$ with $\PPP_\circ := \PPP_\bullet \cup \MMM_\bullet$.
We denote by $\widehat\P_\bullet \subset \P$ the polynomial space
obtained by adding to $\PPP_\bullet$ all indices of $\QQQ_\bullet$, i.e.,
$\widehat\P_\bullet := \hull\{P_\nu : \nu \in \widehat\PPP_\bullet \}$
with $\widehat\PPP_\bullet := \PPP_\bullet \cup \QQQ_\bullet$.

The analysis of the forthcoming adaptive algorithm will also rely on the enriched spaces
\begin{equation} \label{eq:enriched:spaces}
 \widehat\V_\bullet := (\widehat\X_\bullet \otimes \P_\bullet) + (\X_\bullet \otimes \widehat\P_\bullet)
 \quad \text{and} \quad
 \widehat\V_\bullet' := \X_\bullet \otimes \widehat\P_\bullet.
\end{equation}

\subsection{A~posteriori error estimation}

In order to estimate the error due to spatial discretization,
we employ the two-level error estimation strategy from~\cite{bprr18}.
Specifically, our spatial error estimate is given by
\begin{equation} \label{eq:spatial-error-estimate}
 \est_\bullet(\NN_\bullet^+)^2
 := \sum_{z \in \NN_\bullet^+} \est_\bullet(z)^2
 \quad \text{with} \quad
 \est_\bullet(z)^2 
 := \sum_{\nu \in \PPP_\bullet} \frac{|F(\widehat\varphi_{\bullet,z} P_\nu) - B(u_\bullet,\widehat\varphi_{\bullet,z} P_\nu)|^2}{\norm{a_0^{1/2}\nabla\widehat\varphi_{\bullet,z}}{L^2(D)}^2}.
\end{equation}

\begin{remark}
For $d=2$, we have $\#\NN_\bullet^+ = \#\EE_\bullet^\mathrm{int}$,
and the new degrees of freedom correspond to the midpoints of interior edges.
Then, the spatial error estimate can be indexed by $E \in \EE_\bullet^\mathrm{int}$
rather than by $z \in \NN_\bullet^+$; see~\cite{bprr18}.
Furthermore, in this case, one has $K = 3$ in~\eqref{eq:const:K}.
\end{remark}

In order to estimate the error due to polynomial approximation on the parameter domain $\G$,
we employ the hierarchical error estimator from~\cite{bps14,bs16}.
First, for each $\nu \in \QQQ_\bullet$, we define the estimator $e_\bullet^\nu \in \X_\bullet$ satisfying
\begin{equation}\label{eq:def:hat-e-ell-nu}
 B_0(e_\bullet^\nu P_\nu, v_\bullet P_\nu) 
 = F(v_\bullet P_\nu) - B(u_\bullet, v_\bullet P_\nu) 
 \quad \text{for all } v_\bullet \in \X_\bullet.
\end{equation}
Then, the parametric error estimate is defined as follows:
\begin{equation} \label{eq:parametric-error-estimate}
 \est_\bullet(\QQQ_\bullet)^2 
 := \sum_{\nu \in \QQQ_\bullet} \est_\bullet(\nu)^2 
 \quad \text{with} \quad
 \est_\bullet(\nu) := \norm{a_0^{1/2}\nabla e_\bullet^\nu}{L^2(D)}.
\end{equation}
From now on, for any $\MM_\bullet \subseteq \NN_\bullet^+$ and $\MMM_\bullet \subseteq \QQQ_\bullet$,
we use the following notation
\begin{equation*}
\est_\bullet(\MM_\bullet)^2
:= \sum_{z \in \MM_\bullet} {\est_\bullet}(z)^2,
\ \
\est_\bullet(\MMM_\bullet)^2
 := \sum_{\nu \in \MMM_\bullet} {\est_\bullet}(\nu)^2,
\ \
\est_\bullet(\MM_\bullet,\, \MMM_\bullet)^2
 := \est_\bullet(\MM_\bullet)^2 + \est_\bullet(\MMM_\bullet)^2.
\end{equation*}
We define the overall error estimate as follows:
\begin{equation} \label{eq:overall:err:estimate}
 \est_\bullet^2 := \est_\bullet(\NN_\bullet^+,\, \QQQ_\bullet)^2 = \est_\bullet(\NN_\bullet^+)^2 + \est_\bullet(\QQQ_\bullet)^2.
\end{equation}
Let us now consider the enriched space $\widehat\V_\bullet$ defined in~\eqref{eq:enriched:spaces}.
According to the Riesz theorem, there exists a unique $\widehat u_\bullet \in \widehat\V_\bullet$ such that
\begin{equation} \label{eq:discrete_formulation:hat}
 B(\widehat u_\bullet, \widehat v_\bullet) = F(\widehat v_\bullet) 
 \quad \text{for all } \widehat v_\bullet \in \widehat \V_\bullet. 
\end{equation}
Since $\V_\bullet \subset \widehat\V_\bullet$, the Galerkin orthogonality implies that
\begin{equation}\label{eq:pythagoras}
\enorm{u - \widehat u_\bullet}{}^2 + \enorm{\widehat u_\bullet - u_\bullet}{}^2 
= \enorm{u - u_\bullet}{}^2.
\end{equation}
In~\cite[Theorem~3.1]{bprr18}, we prove the following theorem for the overall error estimate~$\est_\bullet$.
The main result is the estimate~\eqref{eq1:thm:estimator},
while efficiency~\eqref{eq:efficiency} and reliability~\eqref{eq:reliability} then follow easily from~\eqref{eq:pythagoras}.

\begin{theorem}\label{thm:estimator}
There exists a constant $\Cthm \geq 1$, which depends only on the initial mesh $\TT_0$ and the mean field $a_0$, such that
\begin{equation}\label{eq1:thm:estimator}
 \frac{\lambda}{K} \, \est_\bullet^2
 \le \enorm{\widehat u_\bullet - u_\bullet}{}^2 
 \le \Lambda \Cthm \, \est_\bullet^2,
\end{equation}
where $\lambda$, $\Lambda$ are the constants in~\eqref{eq:lambda}
and $K$ is the constant in~\eqref{eq:const:K}.
In particular, there holds efficiency
\begin{equation}\label{eq:efficiency}
 \frac{\lambda}{K} \, \est_\bullet^2
 \le \enorm{\widehat u_\bullet - u_\bullet}{}^2 
 \stackrel{\eqref{eq:pythagoras}}{\leq}
 \enorm{u - u_\bullet}{}^2.
\end{equation}
Moreover, under the saturation assumption
\begin{equation}\label{eq:saturation}
 \enorm{u - \widehat u_\bullet}{} \le \qsat \, \enorm{u -  u_\bullet}{}
 \quad \text{with some constant } 0 < \qsat < 1,
\end{equation}
there holds reliability
\begin{equation}\label{eq:reliability}
 \enorm{u - u_\bullet}{}^2
 \stackrel{\eqref{eq:pythagoras}}{\leq} \frac{1}{1 - \qsat^2} \, \enorm{\widehat u_\bullet - u_\bullet}{}^2 
 \le \frac{\Lambda \Cthm}{1 - \qsat^2} \, \est_\bullet^2.
\end{equation}
\end{theorem}

The proof of Theorem~\ref{thm:estimator} given in~\cite{bprr18} essentially relies on the stable subspace
decompositions
\begin{equation*}
 \widehat\X_\bullet = \X_\bullet \oplus \bigoplus_{z \in \NN_\bullet^+} \hull{\{\widehat\varphi_{\bullet,z}\}}
 \quad\hbox{and}\quad
 \widehat\P_\bullet = \P_\bullet \oplus \bigoplus_{\nu \in \QQQ_\bullet} \hull{\{P_\nu\}}.
\end{equation*}
For $d = 2$, the analysis in~\cite{bprr18}, in fact, proves a more general result
than estimate~\eqref{eq1:thm:estimator}.
Let $\TT_\circ = \refine(\TT_\bullet,\MM_\bullet)$ and consider
$z \in (\NN_{\circ} \setminus \NN_\bullet) \setminus \partial D = \NN_\bullet^+ \cap \NN_\circ \subseteq \NN_\bullet^+$.
Let $\varphi_{\circ,z} \in \X_\circ$ and $\widehat\varphi_{\bullet,z} \in \widehat\X_\bullet$ be the corresponding hat functions.
Then, 2D NVB refinement ensures that $\varphi_{\circ,z} = \widehat\varphi_{\bullet,z}$, which yields the stable decomposition
\begin{equation*}
 \X_{\circ} 
 = \X_\bullet \oplus \bigoplus_{z \in \NN_\bullet^+ \cap \NN_\circ} \hull{\{\varphi_{\circ,z}\}}
 = \X_\bullet \oplus \bigoplus_{z \in \NN_\bullet^+ \cap \NN_\circ} \hull{\{\widehat\varphi_{\bullet,z}\}}.
\end{equation*}
As a consequence, the analysis from~\cite{bprr18} also proves the following result
that allows to control the error reduction due to \emph{adaptive} enrichment of
both components of the approximation space $\V_\bullet = \X_\bullet \otimes \P_\bullet$.

\begin{corollary}\label{cor:estimator}
Let $d = 2$.
Let $\Cthm \geq 1$ be the constant from Theorem~\ref{thm:estimator}.
Suppose that $\TT_\circ = \refine(\TT_\bullet,\MM_\bullet)$
and $\widehat\TT_\bullet = \refine(\TT_\bullet,\NN_\bullet^+)$ are obtained by 2D NVB refinement
and $\PPP_\circ = \PPP_\bullet \cup \MMM_\bullet$ for an index set $\MMM_\bullet \subseteq \QQQ_\bullet$.
If $u_\bullet \in \V_\bullet$ and $u_\circ \in \V_\circ$ are two Galerkin approximations, then there holds
\begin{equation}\label{eq1:cor:estimator}
 \frac{\lambda}{K} \, \est_\bullet\big( \NN_\bullet^+ \cap \NN_\circ,\, \MMM_\bullet \big)^2 \le
 \enorm{u_\circ - u_\bullet}{}^2 \le
 \Lambda \Cthm \, \est_\bullet\big( \NN_\bullet^+ \cap \NN_\circ,\, \MMM_\bullet \big)^2.
\end{equation}
\end{corollary}

\section{Main results} \label{sec:main:results}

\subsection{Adaptive algorithms} \label{subsec:algorithms}
Let $\TT_0$ be the initial mesh and let the initial index set $\PPP_0$ contain only the zero index, i.e., $\PPP_0 := \{ \0 \}$.
The adaptive algorithm below generates a sequence $(\TT_\ell)_{\ell \in \N_0}$ of adaptively refined meshes and
a sequence $(\PPP_\ell)_{\ell \in \N_0}$ of adaptively enriched index sets such that, for all $\ell \in \N_0$, there holds
\begin{equation*}
   \TT_{\ell + 1} = \refine(\TT_\ell, \MM_\ell)\ \hbox{for some $\MM_\ell \subseteq \NN_\ell^+$}\ \ \hbox{and}\ \
   \PPP_\ell \subseteq \PPP_{\ell+1} \subseteq \widehat\PPP_\ell = \PPP_\ell \cup \QQQ_\ell.
\end{equation*}
In particular, by the definition of the detail index set~\eqref{def:Q}, one has
$\QQQ_\ell \setminus \PPP_{\ell+1} \subseteq \QQQ_{\ell+1}$ and
$\widehat\PPP_\ell \subseteq \widehat\PPP_{\ell+1}$.
Thus, the following inclusions hold
\begin{equation*}
   \X_\ell \subseteq \X_{\ell + 1} \subseteq \widehat \X_\ell \subset \X
   \quad\hbox{and}\quad
   \P_\ell \subseteq \P_{\ell + 1} \subseteq \widehat \P_\ell \subseteq \widehat \P_{\ell+1} \subset \P.
\end{equation*}
Furthermore, since the adaptive algorithm presented below performs either mesh refinement or parametric enrichment
at each iteration $\ell \in \N_0$, one of the inclusions $\X_\ell \subseteq \X_{\ell+1}$ or $\P_\ell \subseteq \P_{\ell+1}$ is strict.
Therefore, recalling the definition of the enriched spaces $\widehat\V_\ell$ and $\widehat\V_\ell'$ (see~\eqref{eq:enriched:spaces}),
we conclude that
\begin{equation*}
 \V_\ell \subset \widehat\V_\ell' \subset \widehat\V_\ell \subset \V, 
 \quad
 \V_\ell \subset \V_{\ell+1},
 \quad
 \widehat\V_\ell' \subset \widehat\V_{\ell+1}',
 \quad
 \text{and } 
 \widehat\V_\ell \subset \widehat\V_{\ell+1}
 \quad \text{for all } \ell \in \N_0.
\end{equation*}
We consider the following basic loop of an adaptive algorithm,
where the precise marking strategy is still left open, but will be specified subsequently.

\begin{algorithm}\label{algorithm}
{\bfseries Input:} $\TT_0$, $\PPP_0 = \{ \0 \}$, marking criterion.
Set $\ell=0$.
\begin{itemize}
\item[\rm(i)] Compute discrete solution $u_\ell \in \V_\ell$.
\item[\rm(ii)] Compute error indicators $\est_\ell(z)$ and $\est_\ell(\nu)$ for all $z \in \NN_\ell^+$ 
and all $\nu \in \QQQ_\ell$.
\item[\rm(iii)] Use marking criterion to obtain $\MM_\ell \subseteq \NN_\ell^+$ and $\MMM_\ell \subseteq \QQQ_\ell$.
\item[\rm(iv)] Set $\PPP_{\ell+1} = \PPP_\ell \cup \MMM_\ell$ and $\TT_{\ell+1} = \refine(\TT_\ell, \MM_\ell)$.
\item[\rm(v)] Increase the counter $\ell \mapsto \ell + 1$ and continue with {\rm(i)}.
\end{itemize}
{\bfseries Output:}
$(\TT_\ell,\PPP_\ell,u_\ell,\est_\ell)_{\ell \in \N_0}$.
\end{algorithm}

The criteria below specify four different marking strategies
for Step~(iii) of Algorithm~\ref{algorithm}
and, at the same time, determine the type of enrichment for the next iteration of the algorithm.
Each strategy comes with three parameters:
$\vartheta>0$ is a weight modulating the choice between mesh refinement and
parametric enrichment (with parametric enrichment being favored for $\vartheta>1$),
$0 < \thetaX \leq 1$ controls the marking of nodes in $\NN_\ell^+$
(always based on the D\"orfler criterion),
whereas $0 < \thetaP \leq 1$ controls the marking of indices in $\QQQ_\ell$
(based on either the D\"orfler criterion
or the maximum criterion).

The first criterion enforces spatial refinement if the spatial error estimate is comparably large;
otherwise, parametric enrichment is chosen for the next iteration.
The marked facets (resp.,\ marked indices) are obtained via D\"orfler marking.

\begin{marking}[{\cite{egsz14, br18}}]\label{marking:A}
{\bfseries Input:}
error indicators $\{\est_\ell(z) : z \in \NN_\ell^+\}$, $\{\est_\ell(\nu) : \nu \in \QQQ_\ell\}$;
marking parameters $0 < \thetaX, \thetaP \le 1$ and $\vartheta > 0$.\\
Case~{\rm (a)}: $\vartheta \, \est_\ell(\QQQ_\ell) \le \est_\ell(\NN_\ell^+)$.
\begin{itemize}
\item[$-$] Set $\MMM_\ell = \emptyset$;
\item[$-$] Find $\MM_\ell \subseteq \NN_\ell^+$ with minimal cardinality such that
$\thetaX \, \est_\ell(\NN_\ell^+) \le \est_\ell(\MM_\ell)$.
\end{itemize}
Case~{\rm (b)}: $\vartheta \, \est_\ell(\QQQ_\ell) > \est_\ell(\NN_\ell^+)$.
\begin{itemize}
\item[$-$] Find $\MMM_\ell \subseteq \QQQ_\ell$ with minimal cardinality such that
$\thetaP \, \est_\ell(\QQQ_\ell) \le \est_\ell(\MMM_\ell)$;
\item[$-$] Set $\MM_\ell = \emptyset$.
\end{itemize}
{\bfseries Output:} $\MM_\ell \subseteq \NN_\ell^+$ and $\MMM_\ell \subseteq \QQQ_\ell$, where one of the subsets is empty.
\end{marking}

Criterion~\ref{marking:B} is based on the idea that the error estimate $\est_\ell$ on the refined elements
(resp., added indices) provides information about the associated error reduction (see Corollary~\ref{cor:estimator}).
This criterion enforces either spatial refinement (if the error reduction for spatial mesh refinement is comparably large)
or parametric enrichment (otherwise).

\begin{marking}[{\cite{br18}}]\label{marking:B}
{\bfseries Input:}
error indicators $\{\est_\ell(z) : z \in \NN_\ell^+\}$, $\{\est_\ell(\nu) : \nu \in \QQQ_\ell\}$;
marking parameters $0 < \thetaX, \thetaP \le 1$ and $\vartheta > 0$.
\begin{itemize}
\item[$-$] Find $\widetilde\MMM_\ell \subseteq \QQQ_\ell$ with minimal cardinality such that
$\thetaP \, \est_\ell(\QQQ_\ell) \le \est_\ell(\widetilde\MMM_\ell)$.
\item[$-$] Find $\widetilde\MM_\ell \subseteq \NN_\ell^+$ with minimal cardinality such that
$\thetaX \, \est_\ell(\NN_\ell^+) \le \est_\ell(\widetilde\MM_\ell)$.
\item[$-$] Define $\widetilde\RR_\ell := \NN_\ell^+ \cap \widetilde\NN_\ell$, where
$\widetilde\NN_\ell$ is associated with 
$\widetilde\TT_\ell = \refine(\TT_\ell,\widetilde\MM_\ell)$.
\end{itemize}
Case~{\rm (a)}: $\vartheta \, \est_\ell(\widetilde\MMM_\ell) \le \est_\ell(\widetilde\RR_\ell)$.
Set $\MMM_\ell = \emptyset$ and $\MM_\ell = \widetilde\MM_\ell$.\\
Case~{\rm (b)}: $\vartheta \, \est_\ell(\widetilde\MMM_\ell) > \est_\ell(\widetilde\RR_\ell)$.
Set $\MMM_\ell = \widetilde\MMM_\ell$ and $\MM_\ell = \emptyset$.\\
{\bfseries Output:} $\MM_\ell \subseteq \NN_\ell^+$ and $\MMM_\ell \subseteq \QQQ_\ell$, where one of the subsets is empty.
\end{marking}

Criterion~\ref{marking:C} is a modification of Criterion~\ref{marking:A}.
It employs a maximum criterion in the parameter domain, while using D\"orfler marking in the physical domain.
As in Criterion~\ref{marking:A}, the enrichment type is determined by the dominant contributing error estimate.

\begin{marking}\label{marking:C}
{\bfseries Input:}
error indicators $\{\est_\ell(z) : z \in \NN_\ell^+\}$, $\{\est_\ell(\nu) : \nu \in \QQQ_\ell\}$;
marking parameters $0 < \thetaX \le 1$, $0 \le \thetaP \le 1$ and $\vartheta > 0$.\\
Case~{\rm (a)}: $\vartheta \, \est_\ell(\QQQ_\ell) \le \est_\ell(\NN_\ell^+)$.
\begin{itemize}
\item[$-$] Set $\MMM_\ell = \emptyset$;
\item[$-$] Find $\MM_\ell \subseteq \NN_\ell^+$ with minimal cardinality such that
$\thetaX \, \est_\ell(\NN_\ell^+) \le \est_\ell(\MM_\ell)$.
\end{itemize}
Case~{\rm (b)}: $\vartheta \, \est_\ell(\QQQ_\ell) > \est_\ell(\NN_\ell^+)$.
\begin{itemize}
\item[$-$] Define
$\MMM_\ell := \{\mu \in \QQQ_\ell :
                           \est_\ell(\mu) \ge (1-\thetaP) \, \max_{\nu \in \QQQ_\ell} \est_\ell(\nu)
                        \}$;
\item[$-$] Set $\MM_\ell = \emptyset$.
\end{itemize}
{\bfseries Output:} $\MM_\ell \subseteq \NN_\ell^+$ and $\MMM_\ell \subseteq \QQQ_\ell$, where one of the subsets is empty.
\end{marking}

Finally, Criterion~\ref{marking:D} is a modification of Criterion~\ref{marking:B} 
in the same way as Criterion~\ref{marking:C} is a modification of Criterion~\ref{marking:A}.
Namely, we employ D\"orfler marking in the physical domain and
use a maximum criterion in the parameter domain, while
the refinement type for the next iteration is determined by the dominant error reduction.

\begin{marking}\label{marking:D}
{\bfseries Input:}
error indicators $\{\est_\ell(z) : z \in \NN_\ell^+\}$, $\{\est_\ell(\nu) : \nu \in \QQQ_\ell\}$;
marking parameters $0 < \thetaX \le 1$, $0 \le \thetaP \le 1$ and $\vartheta > 0$.
\begin{itemize}
\item[$-$] Define
$\widetilde\MMM_\ell := \{ \mu \in \QQQ_\ell :
                                           \est_\ell(\mu) \ge (1-\thetaP) \, \max_{\nu \in \QQQ_\ell} \est_\ell(\nu)
                                        \}$.
\item[$-$] Find $\widetilde\MM_\ell \subseteq \NN_\ell^+$ with minimal cardinality such that
$\thetaX \, \est_\ell(\NN_\ell^+) \le \est_\ell(\widetilde\MM_\ell)$.
\item[$-$] Define $\widetilde\RR_\ell := \NN_\ell^+ \cap \widetilde\NN_\ell$, where
$\widetilde\NN_\ell$ is associated with 
$\widetilde\TT_\ell = \refine(\TT_\ell,\widetilde\MM_\ell)$.
\end{itemize}
Case~{\rm (a)}: $\vartheta \, \est_\ell(\widetilde\MMM_\ell) \le \est_\ell(\widetilde\RR_\ell)$.
Set $\MMM_\ell = \emptyset$ and $\MM_\ell = \widetilde\MM_\ell$.\\
Case~{\rm (b)}: $\vartheta \, \est_\ell(\widetilde\MMM_\ell) \,{>}\, \est_\ell(\widetilde\RR_\ell)$.
Set\, $\MMM_\ell = \widetilde\MMM_\ell$ and~$\MM_\ell = \emptyset$.\\
{\bfseries Output:} $\MM_\ell \subseteq \NN_\ell^+$ and $\MMM_\ell \subseteq \QQQ_\ell$, where one of the subsets is empty.
\end{marking}

In what follows we will write, e.g., Algorithm~\ref{algorithm}.\ref{marking:A} to refer to the algorithm obtained
by employing Criterion~\ref{marking:A} in Step~(iii) of Algorithm~\ref{algorithm}.
When we refer to Algorithm~\ref{algorithm} without specifying the marking criterion,
this will mean that the statement holds for any of the four proposed marking strategies.

\subsection{Convergence results}

The following theorem is the first main result of the present work.
It shows that Algorithm~\ref{algorithm} ensures convergence of the
underlying error estimates to zero.
We emphasize that it is valid independently of the saturation assumption~\eqref{eq:saturation}. 

\begin{theorem}\label{thm:plain_convergence}
For any choice of the marking parameters $\thetaX,\, \thetaP$ and $\vartheta$,
Algorithm~\ref{algorithm} yields a convergent sequence of error estimates, i.e., $\est_\ell \to 0$ as $\ell \to \infty$.
\end{theorem}

The proof of Theorem~\ref{thm:plain_convergence} is postponed to Section~\ref{section:plain_convergence}.

\begin{remark} \label{remark:main_thm}
The proof of Theorem~\ref{thm:plain_convergence} allows for more general marking strategies
than those proposed in Section~\ref{subsec:algorithms} above
(see Propositions~\ref{prop:conv:parametric}--\ref{prop:conv:spatial} in Section~\ref{section:plain_convergence}).
However, we believe that the marking strategies proposed in Criteria~\ref{marking:A}--\ref{marking:D}
are natural candidates for the present setting.
\end{remark}

The following result is an immediate consequence of Theorem~\ref{thm:plain_convergence} and
the reliability~\eqref{eq:reliability} from Theorem~\ref{thm:estimator}.

\begin{corollary}\label{cor:plain_convergence}
Let $(u_\ell)_{\ell \in \N_0}$ be the sequence of Galerkin solutions generated by Algorithm~\ref{algorithm}.
Denote by $(\widehat u_\ell)_{\ell \in \N_0}$ the associated sequence of Galerkin solutions
satisfying~\eqref{eq:discrete_formulation:hat} and suppose that the saturation assumption~\eqref{eq:saturation}
holds for each pair $u_\ell,\, \widehat u_\ell$ ($\ell \in \N_0$).
Then, for any choice of marking parameters $\thetaX,\, \thetaP$ and $\vartheta$,
Algorithm~\ref{algorithm} yields convergence, i.e., $\enorm{u-u_\ell}{} \to 0$ as $\ell \to \infty$.
\end{corollary}

In 2D and under the saturation assumption~\eqref{eq:saturation},
Algorithm~\ref{algorithm}.\ref{marking:A} and Algorithm~\ref{algorithm}.\ref{marking:B}
allow for a stronger convergence result than Corollary~\ref{cor:plain_convergence}.
The following theorem states linear convergence of the energy error.
The proof is given in Section~\ref{section:linear_convergence}.

\begin{theorem}\label{thm:linear_convergence}
Let $d = 2$ and
let $(u_\ell)_{\ell \in \N_0}$ be the sequence of Galerkin solutions generated by
either Algorithm~\ref{algorithm}.\ref{marking:A} or Algorithm~\ref{algorithm}.\ref{marking:B}
with arbitrary $0 < \thetaX,\, \thetaP \le 1$ and $\vartheta>0$.
Denote by $(\widehat u_\ell)_{\ell \in \N_0}$ the associated sequence of Galerkin solutions satisfying~\eqref{eq:discrete_formulation:hat}
and suppose that the saturation assumption~\eqref{eq:saturation}
holds for each pair $u_\ell,\, \widehat u_\ell$ ($\ell \in \N_0$).
Then, there exists a constant $0 < \qlin < 1$ such that
\begin{equation*}
 \enorm{u-u_{\ell+1}}{} \le \qlin \, \enorm{u-u_{\ell}}{}
 \quad \text{for all }\, \ell \in \N_0.
\end{equation*}
The constant $\qlin$ depends only on the mean field $a_0$, the constant $\tau$ in~\eqref{eq3:a},
the saturation constant $\qsat$ in~\eqref{eq:saturation},
the coarse mesh $\TT_0$, and the marking parameters $\thetaX$, $\thetaP$, $\vartheta$.
\end{theorem}

\section{Numerical results} \label{sec:numer:results}

In this section, we report the results of numerical experiments aiming to underpin our theoretical findings
and compare the performance of Algorithms~\ref{algorithm}.\ref{marking:A}--\ref{algorithm}.\ref{marking:D}
for a range of marking parameters.
The experiments were performed using the open source 
MATLAB toolbox Stochastic T-IFISS \cite{BespalovR_stoch_tifiss}.

We consider the parametric model problem~\eqref{eq:strongform} posed on the L-shaped domain 
$D = (-1,1)^2\setminus (-1,0]^2 \subset \mathbb{R}^2$ and set $f \equiv 1$.
Following~\cite[Section~11.1]{egsz14}, we choose the expansion coefficients $a_m$ ($m \in \N_0$)
in~\eqref{eq1:a} to represent planar Fourier modes of increasing total order, i.e., 
\begin{equation*}
a_0(x) := 1, 
\quad
a_m(x) := \alpha_m \cos(2\pi\beta_1(m) \, x_1) \cos(2\pi\beta_2(m)\, x_2), 
\quad x = (x_1,x_2)\in D. 
\end{equation*}
Here, for all $m \in \N$, $\alpha_m := A m^{-\sigma}$ is the amplitude of the coefficient,
where $\sigma > 1$ and $0 < A < 1/\zeta(\sigma)$, with $\zeta$ denoting the Riemann zeta function, 
while $\beta_1$ and $\beta_2$ are defined as
\begin{equation*}
\beta_1(m) := m - k(m)(k(m) + 1)/2 
\quad \text{and} \quad
\beta_2(m) := k(m) - \beta_1(m),
\end{equation*}
with $k(m) := \lfloor -1/2  + \sqrt{1/4+2m}\rfloor$. 
Note that under these assumptions, both conditions \eqref{eq2:a} and \eqref{eq3:a} are satisfied
with $a_0^{\text{min}} = a_0^{\text{max}} = 1$ and $\tau = A \zeta(\sigma)$, respectively. 
We consider the case of $\sigma=2$, which corresponds to a slow decay of the coefficients;
fixing $\tau=A\zeta(\sigma)=0.9$, this results in $A\approx 0.547$.
Furthermore, we assume that the parameters $y_m$ ($m \in \N$) in \eqref{eq1:a} are the images of 
uniformly distributed independent mean-zero random variables on $[-1,1]$. 
In this case, $\dpi_m(y_m) = \mathrm{d}y_m/2$ and the orthonormal polynomial basis of 
$L^2_{\pi_m}(-1,1)$ consists of scaled Legendre polynomials.
Note that the same model problem was used in numerical experiments 
in, e.g., \cite{egsz14, egsz15, em16, br18, bprr18}.

\begin{table}[t!]
\begin{center}
\setlength\tabcolsep{2.9pt} 
\smallfontthree{
\renewcommand{\arraystretch}{1.85}
\begin{tabular}{c |*{9}{c} c } 
\noalign{\hrule height 1.0pt}
& \multicolumn{9}{c}{ $\thetaP$ } \\
\cline{2-10}
& 0.1 & 0.2 & 0.3 & 0.4 & 0.5 & 0.6 & 0.7 & 0.8 & 0.9 \\
\hline\\[-13pt]
\begin{tabular}{@{}c@{}}{\bf Algorithm~\ref{algorithm}.\ref{marking:A}} \\[-5pt] ($\thetaX=0.8$)\end{tabular} 
& \begin{tabular}{@{}c@{}}2,454,929 \end{tabular}
& \begin{tabular}{@{}c@{}}2,454,929 \end{tabular}
& \begin{tabular}{@{}c@{}}2,454,929 \end{tabular}
& \begin{tabular}{@{}c@{}}2,454,929 \end{tabular}
& \begin{tabular}{@{}c@{}}2,454,929 \end{tabular}
& \begin{tabular}{@{}c@{}}2,403,912 \end{tabular}
& \begin{tabular}{@{}c@{}}1,628,563 \end{tabular}
& \begin{tabular}{@{}c@{}}{\bf 1,560,286} \end{tabular}
& \begin{tabular}{@{}c@{}}1,731,044 \end{tabular} \\  
\rowcolor{myLightGray}
\begin{tabular}{@{}c@{}}{\bf Algorithm~\ref{algorithm}.\ref{marking:B}} \\[-4pt] ($\thetaX=0.7$)\end{tabular} 
& \begin{tabular}{@{}c@{}}3,157,697 \end{tabular}
& \begin{tabular}{@{}c@{}}3,157,697 \end{tabular}
& \begin{tabular}{@{}c@{}}3,157,697 \end{tabular}
& \begin{tabular}{@{}c@{}}3,157,697 \end{tabular}
& \begin{tabular}{@{}c@{}}3,157,697 \end{tabular}
& \begin{tabular}{@{}c@{}}2,146,095 \end{tabular}
& \begin{tabular}{@{}c@{}}1,973,460 \end{tabular}
& \begin{tabular}{@{}c@{}}1,966,801 \end{tabular}
& \begin{tabular}{@{}c@{}}{\bf 1,488,993} \end{tabular} \\
\begin{tabular}{@{}c@{}}{\bf Algorithm~\ref{algorithm}.\ref{marking:C}} \\[-5pt] ($\thetaX=0.7$)\end{tabular} 
& \begin{tabular}{@{}c@{}}2,094,382 \end{tabular}
& \begin{tabular}{@{}c@{}}1,891,752 \end{tabular}
& \begin{tabular}{@{}c@{}}1,970,087 \end{tabular}
& \begin{tabular}{@{}c@{}}2,014,430 \end{tabular}
& \begin{tabular}{@{}c@{}}{\bf 1,496,851} \end{tabular}
& \begin{tabular}{@{}c@{}}1,710,029 \end{tabular}
& \begin{tabular}{@{}c@{}}1,793,937 \end{tabular}
& \begin{tabular}{@{}c@{}}2,185,402 \end{tabular}
& \begin{tabular}{@{}c@{}}1,837,025 \end{tabular}\\  
\rowcolor{myLightGray}
\begin{tabular}{@{}c@{}}{\bf Algorithm~\ref{algorithm}.\ref{marking:D}} \\[-4pt] ($\thetaX=0.7$) \end{tabular} 
& \begin{tabular}{@{}c@{}}2,146,095 \end{tabular}
& \begin{tabular}{@{}c@{}}1,952,007 \end{tabular}
& \begin{tabular}{@{}c@{}}2,000,424 \end{tabular}
& \begin{tabular}{@{}c@{}}1,966,801 \end{tabular}
& \begin{tabular}{@{}c@{}}$\mathbf{1,460,210^\star}$ \end{tabular}
& \begin{tabular}{@{}c@{}}1,604,638 \end{tabular}
& \begin{tabular}{@{}c@{}}1,740,662 \end{tabular}
& \begin{tabular}{@{}c@{}}2,050,900 \end{tabular}
& \begin{tabular}{@{}c@{}}1,855,200 \end{tabular}\\
\noalign{\hrule height 1.0pt}
\end{tabular}
\vspace{8pt}
\caption{
Computational cost~\eqref{eq:cost} of Algorithms~\ref{algorithm}.\ref{marking:A}--\ref{algorithm}.\ref{marking:D}.
For each algorithm, we choose the spatial marking parameter $\thetaX \in \Theta$ for which
the smallest cost is incurred (see Tables~\ref{tab:fulldata:A}--\ref{tab:fulldata:D} in Appendix~\ref{sec:appendix})
and show the computational cost for all $\thetaP \in \Theta$.
The smallest cost for each algorithm is highlighted in boldface in the corresponding row.
The boldface starred value shows the overall smallest cost, i.e., 
the smallest cost among all computations with $81$ pairs $(\thetaX,\thetaP) \in \Theta\times\Theta$
for all four algorithms.
}
\label{tab:best:cost}
}
\end{center} 
\end{table}

\begin{figure}[t!]
\begin{tikzpicture}
\pgfplotstableread{data/lshaped/criterionA/varth1_thX0.8_thP0.1.dat}{\one}
\pgfplotstableread{data/lshaped/criterionA/varth1_thX0.8_thP0.3.dat}{\three}
\pgfplotstableread{data/lshaped/criterionA/varth1_thX0.8_thP0.5.dat}{\five}
\pgfplotstableread{data/lshaped/criterionA/varth1_thX0.8_thP0.7.dat}{\seven}
\pgfplotstableread{data/lshaped/criterionA/varth1_thX0.8_thP0.9.dat}{\nine}
\pgfplotstableread{data/lshaped/criterionA/varth1_thX0.8_thP1.dat}{\ten}
\begin{loglogaxis}
[
width = 8.1cm, height=6.4cm,								
title={{\bf Algorithm~\ref{algorithm}.\ref{marking:A}} ($\thetaX = 0.8$)},	
xlabel={degree of freedom, $N_\ell$}, 						
ylabel={overall error estimate, $\est_\ell$},				
ymajorgrids=true, xmajorgrids=true, grid style=dashed,		
xmin = (2.5)*10^(1),
xmax = (2.5)*10^(6),
ymin = (4)*10^(-3),
ymax = (1.6)*10^(-1),
legend style={legend pos=south west, legend cell align=left, fill=none, draw=none, font={\fontsize{9pt}{12pt}\selectfont}}
]
\addplot[myViolet,line width=1.5pt,densely dashed]		table[x=dofs, y=error]{\one};
\addplot[red,line width=1.5pt,densely dashed]			table[x=dofs, y=error]{\three};
\addplot[blue,line width=1.5pt,densely dashed]			table[x=dofs, y=error]{\five};
\addplot[teal,line width=1.5pt,densely dashed]			table[x=dofs, y=error]{\seven};
\addplot[myOrange,line width=1.5pt,densely dashed]		table[x=dofs, y=error]{\nine};
\addplot[myBrown,line width=1.5pt,densely dashed]		table[x=dofs, y=error]{\ten};
\addplot[black,solid,domain=10^(1.5):10^(7.8)] { 1.0*x^(-0.34) };
\node at (axis cs:3e4,3e-2) [anchor=south west] {$\mathcal{O}(N_\ell^{-0.34})$};
\legend{
{$\thetaP = 0.1$},
{$\thetaP = 0.3$},
{$\thetaP = 0.5$},
{$\thetaP = 0.7$},
{$\thetaP = 0.9$},
{$\thetaP = 1$},
}
\end{loglogaxis}
\end{tikzpicture}
\hfill
\begin{tikzpicture}
\pgfplotstableread{data/lshaped/criterionB/varth1_thX0.7_thP0.1.dat}{\one}
\pgfplotstableread{data/lshaped/criterionB/varth1_thX0.7_thP0.3.dat}{\three}
\pgfplotstableread{data/lshaped/criterionB/varth1_thX0.7_thP0.5.dat}{\five}
\pgfplotstableread{data/lshaped/criterionB/varth1_thX0.7_thP0.7.dat}{\seven}
\pgfplotstableread{data/lshaped/criterionB/varth1_thX0.7_thP0.9.dat}{\nine}
\pgfplotstableread{data/lshaped/criterionB/varth1_thX0.7_thP1.dat}{\ten}
\begin{loglogaxis}
[
width = 8.1cm, height=6.4cm,								
title={{\bf Algorithm~\ref{algorithm}.\ref{marking:B}} ($\thetaX = 0.7$)},	
xlabel={degree of freedom, $N_\ell$}, 						
ylabel={overall error estimate, $\est_\ell$},				
ymajorgrids=true, xmajorgrids=true, grid style=dashed,		
xmin = (2.5)*10^(1),
xmax = (2.5)*10^(6),
ymin = (4)*10^(-3),
ymax = (1.6)*10^(-1),
legend style={legend pos=south west, legend cell align=left, fill=none, draw=none, font={\fontsize{9pt}{12pt}\selectfont}}
]
\addplot[myViolet,line width=1.5pt,densely dashed]		table[x=dofs, y=error]{\one};
\addplot[red,line width=1.5pt,densely dashed]		table[x=dofs, y=error]{\three};
\addplot[blue,line width=1.5pt,densely dashed]	table[x=dofs, y=error]{\five};
\addplot[teal,line width=1.5pt,densely dashed]table[x=dofs, y=error]{\seven};
\addplot[myOrange,line width=1.5pt,densely dashed]	table[x=dofs, y=error]{\nine};
\addplot[myBrown,line width=1.5pt,densely dashed]	table[x=dofs, y=error]{\ten};
\addplot[black,solid,domain=10^(1.5):10^(7.8)] { 1.0*x^(-0.34) };
\node at (axis cs:3e4,3e-2) [anchor=south west] {$\mathcal{O}(N_\ell^{-0.34})$};
\legend{
{$\thetaP = 0.1$},
{$\thetaP = 0.3$},
{$\thetaP = 0.5$},
{$\thetaP = 0.7$},
{$\thetaP = 0.9$},
{$\thetaP = 1$}
}
\end{loglogaxis}
\end{tikzpicture}\\
\begin{tikzpicture}
\pgfplotstableread{data/lshaped/criterionC/varth1_thX0.7_thP0.1.dat}{\one}
\pgfplotstableread{data/lshaped/criterionC/varth1_thX0.7_thP0.3.dat}{\three}
\pgfplotstableread{data/lshaped/criterionC/varth1_thX0.7_thP0.5.dat}{\five}
\pgfplotstableread{data/lshaped/criterionC/varth1_thX0.7_thP0.7.dat}{\seven}
\pgfplotstableread{data/lshaped/criterionC/varth1_thX0.7_thP0.9.dat}{\nine}
\pgfplotstableread{data/lshaped/criterionA/varth1_thX0.7_thP1.dat}{\ten}
\begin{loglogaxis}
[
width = 8.1cm, height=6.4cm,								
title={{\bf Algorithm~\ref{algorithm}.\ref{marking:C}} ($\thetaX = 0.7$)},				
xlabel={degree of freedom, $N_\ell$}, 						
ylabel={overall error estimate, $\est_\ell$},				
ymajorgrids=true, xmajorgrids=true, grid style=dashed,		
xmin = (2.5)*10^(1),
xmax = (2.5)*10^(6),
ymin = (4)*10^(-3),
ymax = (1.6)*10^(-1),
legend style={legend pos=south west, legend cell align=left, fill=none, draw=none, font={\fontsize{9pt}{12pt}\selectfont}},
]
\addplot[myViolet,line width=1.5pt,densely dashed]		table[x=dofs, y=error]{\one};
\addplot[red,line width=1.5pt,densely dashed]			table[x=dofs, y=error]{\three};
\addplot[blue,line width=1.5pt,densely dashed]			table[x=dofs, y=error]{\five};
\addplot[teal,line width=1.5pt,densely dashed]			table[x=dofs, y=error]{\seven};
\addplot[myOrange,line width=1.5pt,densely dashed]		table[x=dofs, y=error]{\nine};
\addplot[myBrown,line width=1.5pt,densely dashed]		table[x=dofs, y=error]{\ten};
\addplot[black,solid,domain=10^(1.5):10^(7.8)] { 1.0*x^(-0.34) };
\node at (axis cs:3e4,3e-2) [anchor=south west] {$\mathcal{O}(N_\ell^{-0.34})$};
\legend{
{$\thetaP = 0.1$},
{$\thetaP = 0.3$},
{$\thetaP = 0.5$},
{$\thetaP = 0.7$},
{$\thetaP = 0.9$},
{$\thetaP = 1$}
}
\end{loglogaxis}
\end{tikzpicture}
\hfill
\begin{tikzpicture}
\pgfplotstableread{data/lshaped/criterionD/varth1_thX0.7_thP0.1.dat}{\one}
\pgfplotstableread{data/lshaped/criterionD/varth1_thX0.7_thP0.3.dat}{\three}
\pgfplotstableread{data/lshaped/criterionD/varth1_thX0.7_thP0.5.dat}{\five}
\pgfplotstableread{data/lshaped/criterionD/varth1_thX0.7_thP0.7.dat}{\seven}
\pgfplotstableread{data/lshaped/criterionD/varth1_thX0.7_thP0.9.dat}{\nine}
\pgfplotstableread{data/lshaped/criterionB/varth1_thX0.7_thP1.dat}{\ten}
\begin{loglogaxis}
[
width = 8.1cm, height=6.4cm,								
title={{\bf Algorithm~\ref{algorithm}.\ref{marking:D}} ($\thetaX = 0.7$)},	
xlabel={degree of freedom, $N_\ell$}, 						
ylabel={overall error estimate, $\est_\ell$},				
ymajorgrids=true, xmajorgrids=true, grid style=dashed,		
xmin = (2.5)*10^(1),
xmax = (2.5)*10^(6),
ymin = (4)*10^(-3),
ymax = (1.6)*10^(-1),
legend style={legend pos=south west, legend cell align=left, fill=none, draw=none, font={\fontsize{9pt}{12pt}\selectfont}}
]
\addplot[myViolet,line width=1.5pt,densely dashed]		table[x=dofs, y=error]{\one};
\addplot[red,line width=1.5pt,densely dashed]			table[x=dofs, y=error]{\three};
\addplot[blue,line width=1.5pt,densely dashed]			table[x=dofs, y=error]{\five};
\addplot[teal,line width=1.5pt,densely dashed]			table[x=dofs, y=error]{\seven};
\addplot[myOrange,line width=1.5pt,densely dashed]		table[x=dofs, y=error]{\nine};
\addplot[myBrown,line width=1.5pt,densely dashed]		table[x=dofs, y=error]{\ten};
\addplot[black,solid,domain=10^(1.5):10^(7.8)] { 1.0*x^(-0.34) };
\node at (axis cs:3e4,3e-2) [anchor=south west] {$\mathcal{O}(N_\ell^{-0.34})$};
\legend{
{$\thetaP = 0.1$},
{$\thetaP = 0.3$},
{$\thetaP = 0.5$},
{$\thetaP = 0.7$},
{$\thetaP = 0.9$},
{$\thetaP = 1$}
}
\end{loglogaxis}
\end{tikzpicture}
\caption{
Decay of the overall error estimates $\est_\ell$ computed at each iteration of
Algorithm~\ref{algorithm}.\ref{marking:A} with $\thetaX=0.8$ and 
Algorithms~\ref{algorithm}.\ref{marking:B}--\ref{algorithm}.\ref{marking:D} with $\thetaX=0.7$,
for $\thetaP \in \{ 0.1, 0.3, 0.5, 0.7, 0.9 , 1\}$.
}
\label{Exp1:data:1}
\end{figure}
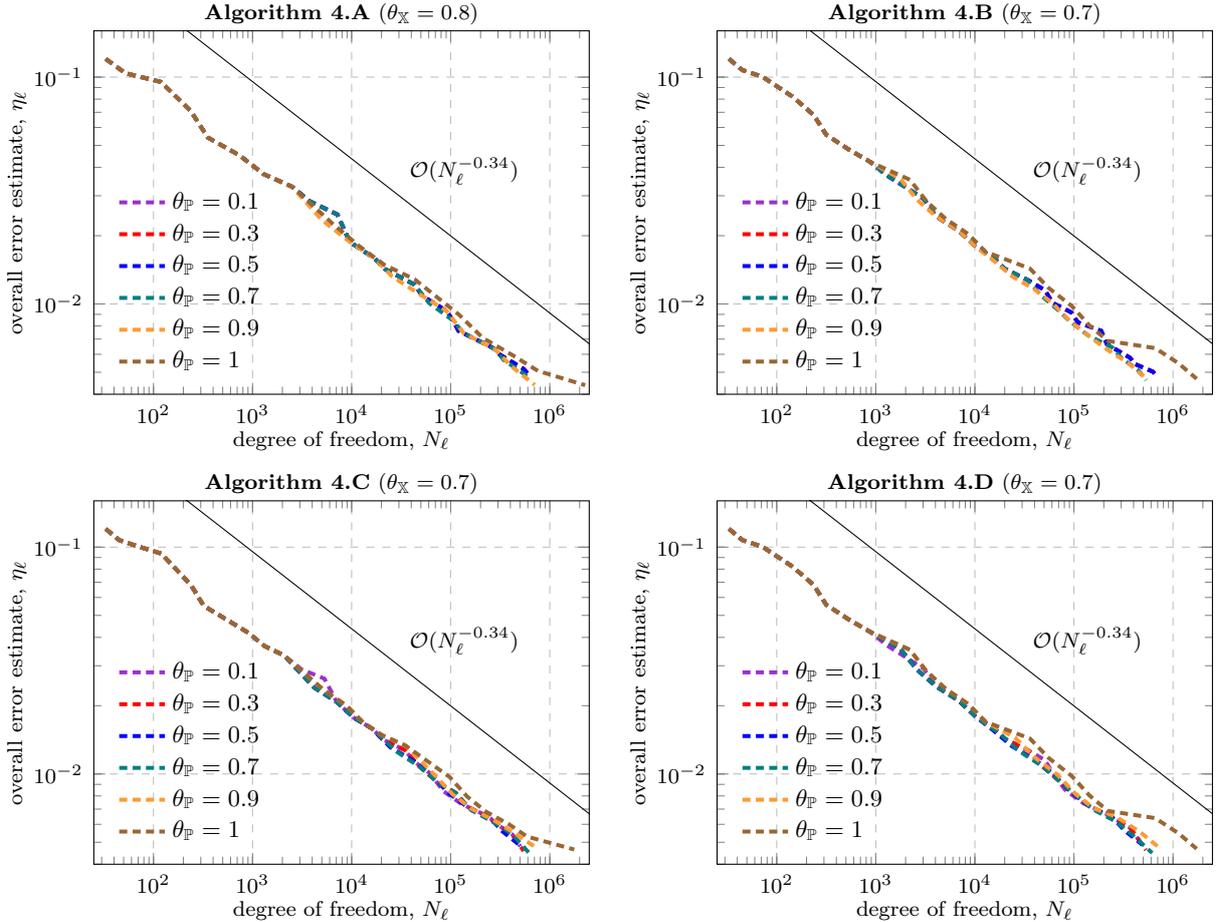

We compare the performance of Algorithms~\ref{algorithm}.\ref{marking:A}--\ref{algorithm}.\ref{marking:D}
with respect to a measure of the total amount of work needed to reach a prescribed tolerance~$\tol$.
Let $L=L(\tol) \in \N$ be the smallest integer such that $\est_L \leq \tol$, 
and let $N_\ell := \dim(\V_\ell) = \dim(\X_\ell)\dim(\P_\ell)$ be the total number of degrees 
of freedom at the $\ell$-th iteration. 
We define the computational \emph{cost} of Algorithm~\ref{algorithm}
as the cumulative number of degrees of freedom for all iterations of the adaptive loop, i.e.,
\begin{equation} \label{eq:cost}
\cost = \cost(L) := \sum_{\ell=0}^{L} N_\ell.
\end{equation}

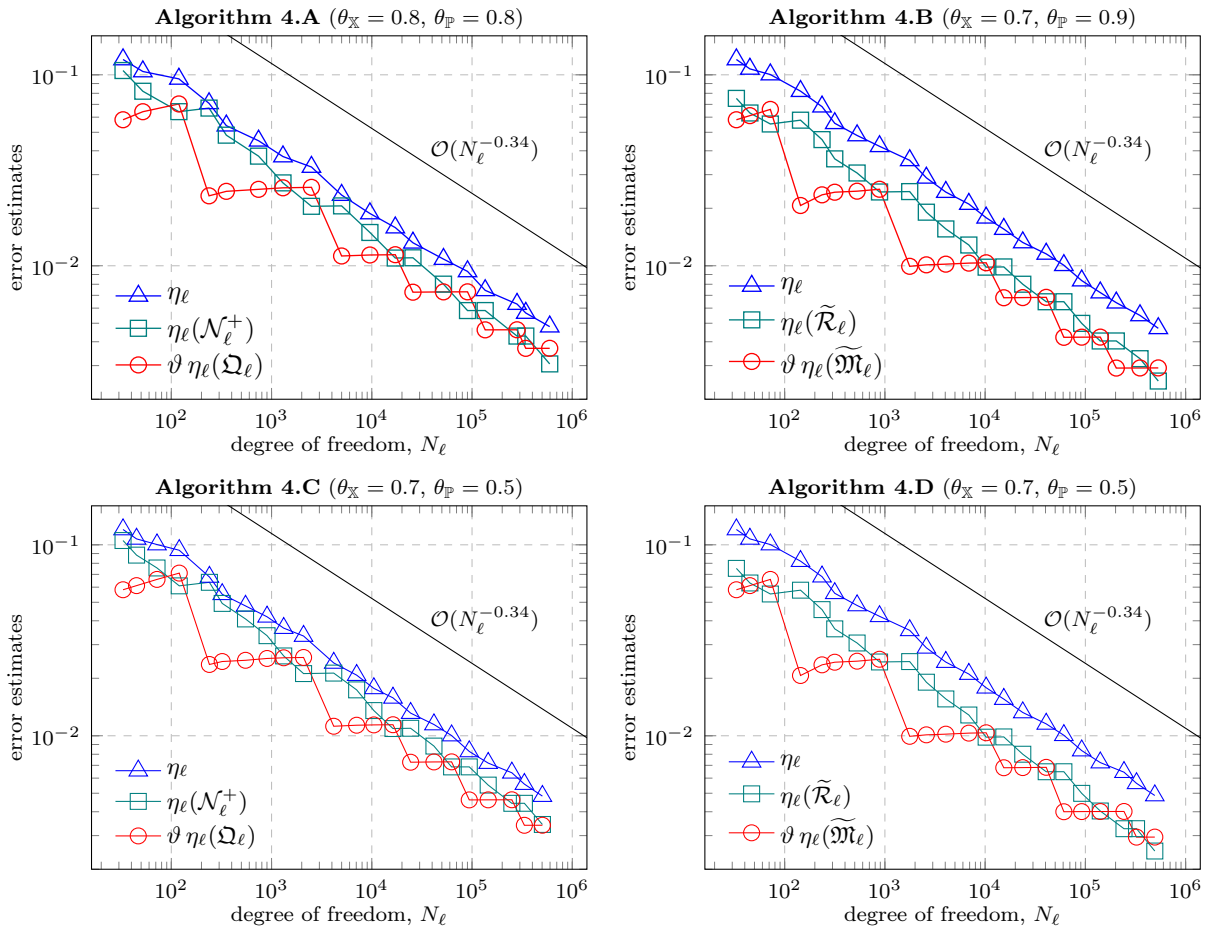
\begin{figure}[h!]
\begin{tikzpicture}
\pgfplotstableread{data/lshaped/criterionA/varth1_thX0.8_thP0.8.dat}{\one}
\begin{loglogaxis}
[
width = 8.1cm, height=6.4cm,								
title={{\bf Algorithm~\ref{algorithm}.\ref{marking:A}} ($\thetaX=0.8$, $\thetaP = 0.8$)},
xlabel={degree of freedom, $N_\ell$},						
ylabel={error estimates},									
ymajorgrids=true, xmajorgrids=true, grid style=dashed,		
xmin = (1.6)*10^(1),
xmax = (1.4)*10^(6),
ymin = (2.0)*10^(-3),
ymax = (1.6)*10^(-1),
legend style={legend pos=south west, legend cell align=left, fill=none, draw=none, font={\fontsize{10pt}{12pt}\selectfont}}
]
\addplot[blue,mark=triangle,mark size=4.0pt,line width=0.5]	table[x=dofs, y=error]{\one};
\addplot[teal,mark=square,mark size=3.0pt,line width=0.5]	table[x=dofs, y=yp_one]{\one};
\addplot[red,mark=o,mark size=3.0pt,line width=0.5]			table[x=dofs, y=xq_one]{\one};
\addplot[black,solid,domain=10^(1.5):10^(7.8)] { 1.2*x^(-0.34) };
\node at (axis cs:3e4,3e-2) [anchor=south west] {$\mathcal{O}(N_\ell^{-0.34})$};
\legend{
{$\est_\ell$},
{$\est_\ell(\NN_\ell^+)$},
{$\vartheta\,\est_\ell(\QQQ_\ell)$},
}
\end{loglogaxis}
\end{tikzpicture}
\hfill
\begin{tikzpicture}
\pgfplotstableread{data/lshaped/criterionB/varth1_thX0.7_thP0.9.dat}{\one}
\begin{loglogaxis}
[
width = 8.1cm, height=6.4cm,								
title={{\bf Algorithm~\ref{algorithm}.\ref{marking:B}} ($\thetaX = 0.7$, $\thetaP = 0.9$)},	
xlabel={degree of freedom, $N_\ell$},						
ylabel={error estimates},									
ymajorgrids=true, xmajorgrids=true, grid style=dashed,		
xmin = (1.6)*10^(1),
xmax = (1.4)*10^(6),
ymin = (2.0)*10^(-3),
ymax = (1.6)*10^(-1),
legend style={legend pos=south west, legend cell align=left, fill=none, draw=none, font={\fontsize{10pt}{12pt}\selectfont}}
]
\addplot[blue,mark=triangle,mark size=4.0pt,line width=0.5]	table[x=dofs, y=error]{\one};
\addplot[teal,mark=square,mark size=3.0pt,line width=0.5]	table[x=dofs, y=yp_two]{\one};
\addplot[red,mark=o,mark size=3.0pt,line width=0.5]			table[x=dofs, y=xq_two]{\one};
\addplot[black,solid,domain=10^(1.5):10^(7.8)] 			{ 1.2*x^(-0.34) };
\node at (axis cs:3e4,3e-2) [anchor=south west] {$\mathcal{O}(N_\ell^{-0.34})$};
\legend{
{$\est_\ell$},
{$\est_\ell(\Rtildel)$},
{$\vartheta\,\est_\ell(\Mtildel)$},
}
\end{loglogaxis}
\end{tikzpicture}\\
\begin{tikzpicture}
\pgfplotstableread{data/lshaped/criterionC/varth1_thX0.7_thP0.5.dat}{\one}
\begin{loglogaxis}
[
width = 8.1cm, height=6.4cm,								
title={{\bf Algorithm~\ref{algorithm}.\ref{marking:C}} ($\thetaX = 0.7$, $\thetaP = 0.5$)},	
xlabel={degree of freedom, $N_\ell$},						
ylabel={error estimates},									
ymajorgrids=true, xmajorgrids=true, grid style=dashed,		
xmin = (1.6)*10^(1),
xmax = (1.4)*10^(6),
ymin = (2.0)*10^(-3),
ymax = (1.6)*10^(-1),
legend style={legend pos=south west, legend cell align=left, fill=none, draw=none, font={\fontsize{9pt}{12pt}\selectfont}}
]
\addplot[blue,mark=triangle,mark size=4.0pt]	table[x=dofs, y=error]{\one};
\addplot[teal,mark=square,mark size=3.0pt]		table[x=dofs, y=yp_one]{\one};
\addplot[red,mark=o,mark size=3pt]				table[x=dofs, y=xq_one]{\one};
\addplot[black,solid,domain=10^(1.5):10^(7.8)] { 1.2*x^(-0.34) };
\node at (axis cs:3e4,3e-2) [anchor=south west] {$\mathcal{O}(N_\ell^{-0.34})$};
\legend{
{$\est_\ell$},
{$\est_\ell(\NN_\ell^+)$},
{$\vartheta\,\est_\ell(\QQQ_\ell)$},
}
\end{loglogaxis}
\end{tikzpicture}
\hfill
\begin{tikzpicture}
\pgfplotstableread{data/lshaped/criterionD/varth1_thX0.7_thP0.5.dat}{\one}
\begin{loglogaxis}
[
width = 8.1cm, height=6.4cm,								
title={{\bf Algorithm~\ref{algorithm}.\ref{marking:D}} ($\thetaX = 0.7$, $\thetaP = 0.5$)},					
xlabel={degree of freedom, $N_\ell$}, 						
ylabel={error estimates},							
ymajorgrids=true, xmajorgrids=true, grid style=dashed,		
xmin = (1.6)*10^(1),
xmax = (1.4)*10^(6),
ymin = (2.0)*10^(-3),
ymax = (1.6)*10^(-1),
legend style={legend pos=south west, legend cell align=left, fill=none, draw=none, font={\fontsize{9pt}{12pt}\selectfont}}
]
\addplot[blue,mark=triangle,mark size=4.0pt]	table[x=dofs, y=error]{\one};
\addplot[teal,mark=square,mark size=3.0pt]		table[x=dofs, y=yp_two]{\one};
\addplot[red,mark=o,mark size=3pt]			table[x=dofs, y=xq_two]{\one};
\addplot[black,solid,domain=10^(1.5):10^(7.8)] { 1.2*x^(-0.34) };
\node at (axis cs:3e4,3e-2) [anchor=south west] {$\mathcal{O}(N_\ell^{-0.34})$};
\legend{
{$\est_\ell$},
{$\est_\ell(\Rtildel)$},
{$\vartheta\,\est_\ell(\Mtildel)$},
}
\end{loglogaxis}
\end{tikzpicture}
\caption{
Decay of the error estimates computed at each iteration of
Algorithms~\ref{algorithm}.\ref{marking:A}--\ref{algorithm}.\ref{marking:D}
with the marking parameters $\thetaX,\,\thetaP \in \Theta$ that yield smallest cost (see Table~\ref{tab:best:cost}).
}
\label{Exp1:data:3}
\end{figure}

We set $\tol=5$e-$03$ and run Algorithms~\ref{algorithm}.\ref{marking:A}--\ref{algorithm}.\ref{marking:D}
with marking parameters $\thetaX,\,\thetaP \in \Theta := \{ 0.1, 0.2, \dots, 0.9 \}$
(we set $\vartheta=1$ in each Marking criterion~\ref{marking:A}--\ref{marking:D}).
The computational costs and the empirical convergence rates for each algorithm
with $81$ pairs $(\thetaX,\thetaP)\in\Theta\times\Theta$ of marking parameters are shown
in Tables~\ref{tab:fulldata:A}--\ref{tab:fulldata:D} in Appendix~\ref{sec:appendix}.
A~snapshot of these results is presented in Table~\ref{tab:best:cost}.
The results show that the overall smallest cost is achieved by Algorithm~\ref{algorithm}.\ref{marking:D}
for the values $\thetaX=0.7$ and $\thetaP=0.5$. 
These values of marking parameters are the ones for which also Algorithm~\ref{algorithm}.\ref{marking:C} 
yields the smallest cost among all pairs $(\thetaX,\thetaP) \in \Theta\times\Theta$.
This similarity does not hold for Algorithms~\ref{algorithm}.\ref{marking:A}--\ref{algorithm}.\ref{marking:B},
for which the smallest cost is achieved with 
$\thetaX = \thetaP = 0.8$ for Algorithm~\ref{algorithm}.\ref{marking:A} 
and with $\thetaX=0.7$ and $\thetaP=0.9$ for Algorithm~\ref{algorithm}.\ref{marking:B}.
Thus, we conclude that, for the above values of marking parameters,
the adaptive algorithms with refinements driven by dominant error reduction estimates
(Algorithms~\ref{algorithm}.\ref{marking:B} and~\ref{algorithm}.\ref{marking:D})
incur less computational costs than their counterparts
driven by dominant contributing error estimates
(Algorithms~\ref{algorithm}.\ref{marking:A} and~\ref{algorithm}.\ref{marking:C}).
On the other hand, the algorithms that employ the maximum criterion for parametric refinement
(Algorithms~\ref{algorithm}.\ref{marking:C} and~\ref{algorithm}.\ref{marking:D})
incur less computational costs than their counterparts that use D{\"o}rfler marking
(Algorithms~\ref{algorithm}.\ref{marking:A} and~\ref{algorithm}.\ref{marking:B}).
Overall, the smallest computational cost is incurred by the algorithm that combines
these two winning strategies---Algorithm~\ref{algorithm}.\ref{marking:D}.

Figure~\ref{Exp1:data:1} shows the decay of the overall error estimate $\est_\ell$ 
versus the number of degrees of freedom $N_\ell$ for different values of $\thetaP \in \Theta$
with $\thetaX=0.8$ in Algorithm~\ref{algorithm}.\ref{marking:A}
and $\thetaX=0.7$ in Algorithms~\ref{algorithm}.\ref{marking:B}--\ref{algorithm}.\ref{marking:D}. 
The aim of these plots is to show that the adaptive algorithm converges
regardless of the marking criterion and the value of $\thetaP$ used
(similar decay rates are obtained for other values of $\thetaX,\,\thetaP \in \Theta$;
see Appendix~\ref{sec:appendix}).
Observe that $\est_\ell$ decays also in the case $\thetaP = 1 \notin \Theta$
for all algorithms.
However, in this case, significantly more degrees of freedom are needed to reach the prescribed tolerance, 
compared to the cases of $\thetaP \in \Theta$. 
This is because, for $\thetaP = 1$, each parametric enrichment is performed
by augmenting the index set $\PPP_\ell$ with the whole detail index set $\QQQ_\ell$.

In Figure~\ref{Exp1:data:3}, we plot the decay of all error estimates computed by 
the four algorithms with the pairs of marking parameters yielding the corresponding smallest cost.
As expected, we see that the decay rates of $\est_\ell$ are similar in all four cases. 

To conclude, we test the effectiveness of our error estimation strategy by computing
a reference energy error as follows. 
We first compute an accurate solution 
$\uref \in \V_{\text{ref}}~:= \X_{\text{ref}}\otimes \P_{\text{ref}}$ using 
quadratic (P2) finite element approximations over a fine mesh 
$\TT_{\text{ref}}$ and employing a large index set $\PPP_{\text{ref}}$. 
Then, we define the effectivity indices
\begin{equation*}
\zeta_\ell := \frac{\est_\ell}{ \enorm{u_{\text{ref}} - u_\ell}{} } 
= \frac{\est_\ell}{ (\enorm{u_{\text{ref}}}{}^2 - \enorm{u_\ell}{}^2)^{1/2} }
\quad \text{for all $\ell=0,\dots,L$},
\end{equation*}
where the equality holds due to Galerkin orthogonality and the symmetry of the bilinear form $B(\cdot,\cdot)$.
In this experiment, we choose $\TT_{\mathrm{ref}}$ to be the uniform refinement of the 
mesh $\TT_L$ generated by Algorithm~\ref{algorithm}.\ref{marking:B} with $\thetaP=0.5$ 
(i.e., one of the final meshes with the largest number of elements) 
and $\PPP_{\mathrm{ref}}$ to be the final index set $\PPP_L$ produced by 
Algorithm~\ref{algorithm}.\ref{marking:D} with $\thetaP=0.8$ (i.e., one of the largest index sets generated). 

\begin{figure}[t!]
\begin{tikzpicture}
\pgfplotstableread{data/lshaped/criterionA/varth1_thX0.8_thP0.8.dat}{\one}
\pgfplotstableread{data/lshaped/criterionB/varth1_thX0.7_thP0.9.dat}{\two}
\begin{semilogxaxis}
[
width = 8.1cm, height=6cm,								
xlabel={degree of freedom, $N_\ell$}, 					
ylabel={effectivity index, $\zeta_\ell$},				
ymajorgrids=true, xmajorgrids=true, grid style=dashed,	
xmin=(2)*10^(1), xmax=(8)*10^(5),						
ymin = 0.69,	 ymax = 0.85,							
ytick={0.65,0.7,0.75,0.8, 0.83},						
legend style={legend pos=north west, legend cell align=left, fill=none, draw=none, font={\fontsize{8pt}{12pt}\selectfont}}
]
\addplot[blue,mark=o,mark size=2.5pt]		table[x=dofs, y=effind]{\one};
\addplot[red,,mark=square,mark size=2.5pt]	table[x=dofs, y=effind]{\two};
\legend{
{{\bf Algorithm~\ref{algorithm}.\ref{marking:A}} ($\thetaX=0.8$, $\thetaP=0.8$)},
{{\bf Algorithm~\ref{algorithm}.\ref{marking:B}} ($\thetaX=0.7$, $\thetaP=0.9$)},
}
\end{semilogxaxis}
\end{tikzpicture}
\hfill
\begin{tikzpicture}
\pgfplotstableread{data/lshaped/criterionC/varth1_thX0.7_thP0.5.dat}{\three}
\pgfplotstableread{data/lshaped/criterionD/varth1_thX0.7_thP0.5.dat}{\four}
\begin{semilogxaxis}
[
width = 8.1cm, height=6cm,								
xlabel={degree of freedom, $N_\ell$}, 					
ylabel={effectivity index, $\zeta_\ell$},		
ymajorgrids=true, xmajorgrids=true, grid style=dashed,	
xmin=(2)*10^(1), xmax=(8)*10^(5),						
ymin = 0.69,	 ymax = 0.85,							
ytick={0.65,0.7,0.75,0.8, 0.83},						
legend style={legend pos=north west, legend cell align=left, fill=none, draw=none, font={\fontsize{8pt}{12pt}\selectfont}}
]
\addplot[myOrange,mark=triangle,mark size=3.5pt]	table[x=dofs, y=effind]{\three};
\addplot[teal,mark=diamond,mark size=3.5pt]			table[x=dofs, y=effind]{\four};
\legend{
{{\bf Algorithm~\ref{algorithm}.\ref{marking:C}} ($\thetaX=0.7$, $\thetaP=0.5$)},
{{\bf Algorithm~\ref{algorithm}.\ref{marking:D}} ($\thetaX=0.7$, $\thetaP=0.5$)}
}
\end{semilogxaxis}
\end{tikzpicture}
\caption{
The effectivity indices $\zeta_\ell$ for the sGFEM solutions at each iteration of
Algorithms~\ref{algorithm}.\ref{marking:A}--\ref{algorithm}.\ref{marking:B} (left) and 
Algorithms~\ref{algorithm}.\ref{marking:C}--\ref{algorithm}.\ref{marking:D} (right)
with the marking parameters $\thetaX,\,\thetaP \in \Theta$ that yield smallest cost (see Table~\ref{tab:best:cost}).
}
\label{Exp1:data:4}
\end{figure}
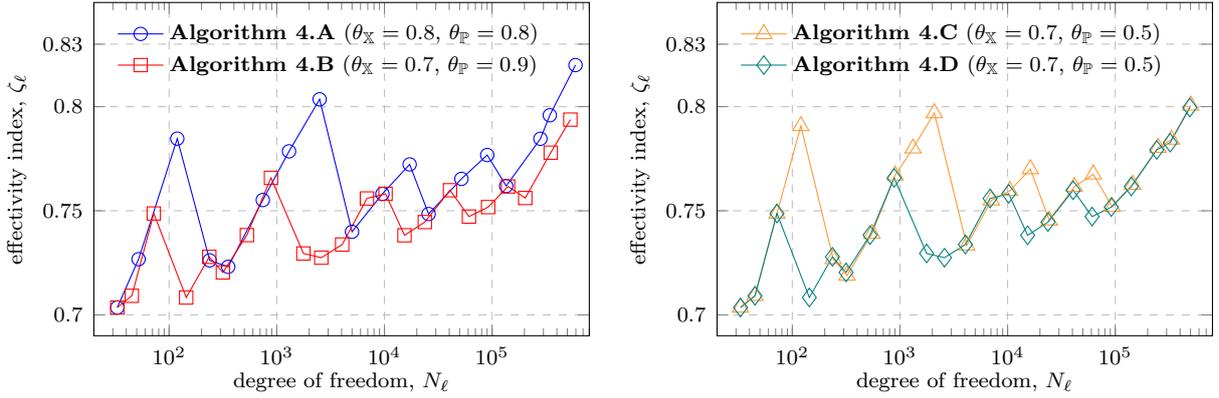

Figure~\ref{Exp1:data:4} shows the effectivity indices $\zeta_\ell$ obtained for 
Algorithms~\ref{algorithm}.\ref{marking:A}--\ref{algorithm}.\ref{marking:B} (left) and 
Algorithms~\ref{algorithm}.\ref{marking:C}--\ref{algorithm}.\ref{marking:D} (right) 
with the pairs of parameters $(\thetaX,\thetaP)$ for which the smallest cost is attained.
We observe that in all cases the error is slightly underestimated, 
as the effectivity indices vary in a range between $0.7$ and $0.82$ throughout all iterations.

\section{Proof of Theorem~\ref{thm:plain_convergence} (plain convergence)} \label{section:plain_convergence}

We start with stating three propositions which address convergence of either
the spatial component or the parametric component of the
error estimate given by~\eqref{eq:overall:err:estimate}.
To ease the readability, the proofs of propositions are postponed to Section~\ref{sec:auxiliary}.

The first proposition proves that each parametric error indicator converges
to some limiting error indicator.

\begin{proposition}\label{prop:conv:parametric:new}
For $\nu \in \QQQ_\ell$, let $\est_\ell(\nu) \ge 0$ be the parametric error indicator from~\eqref{eq:parametric-error-estimate}.
For $\nu \in \III \setminus \QQQ_\ell$, define $\est_\ell(\nu) := 0$.
Then, for each $\nu \in \III$, there exists $\est_\infty(\nu) \ge 0$ such that
\begin{equation}\label{eq:prop:conv:parametric:new}
 \sum_{\nu \in \III} \est_\infty(\nu)^2 < \infty
 \quad \text{and} \quad
 \sum_{\nu \in \III} |\est_\infty(\nu) - \est_\ell(\nu)|^2 \to 0
 \quad \text{as } \ell \to \infty.
\end{equation}
\end{proposition}

The second proposition states that the parametric enrichment satisfying a certain weak marking
criterion along a subsequence guarantees convergence of the whole sequence of parametric error estimates.

\begin{proposition}\label{prop:conv:parametric}
Let $g_\P:\R_{\ge0} \to \R_{\ge0}$ be a continuous function with $g_\P(0) = 0$.
Suppose that Algorithm~\ref{algorithm} yields
a subsequence $(\PPP_{\ell_k})_{k\in \N_0} \subset (\PPP_{\ell})_{\ell \in \N_0}$
satisfying the following property:
\begin{equation}\label{eq1:prop:conv:parametric}
\est_{\ell_k}(\mu) \le g_\P \big( \est_{\ell_k}(\MMM_{\ell_k}) \big)\quad
\text{for all } k \in \N_0
\text{ and all } \mu \in \QQQ_{\ell_k} \setminus \MMM_{\ell_k},
\end{equation}
i.e., the non-marked multi-indices are controlled by the marked ones.
Then, the sequence of parametric error estimates converges to zero, i.e., $\est_\ell(\QQQ_\ell) \to 0$ as $\ell \to \infty$.
\end{proposition}

The third proposition addresses convergence of spatial error estimates.
Unlike in Proposition~\ref{prop:conv:parametric} for parametric estimates,
the convergence here is only shown along the subsequence for which spatial refinement takes place.

\begin{proposition}\label{prop:conv:spatial}
Let $g_\X:\R_{\ge0} \to \R_{\ge0}$ be a continuous function with $g_\X(0) = 0$.
Suppose that Algorithm~\ref{algorithm} yields a subsequence
$(\TT_{\ell_k})_{k \in \N_0} \subset (\TT_\ell)_{\ell \in \N_0}$ satisfying the following property:
\begin{equation}\label{eq1:prop:conv:spatial}
 \est_{\ell_k}(z) \le g_\X \big( \est_{\ell_k}(\MM_{\ell_k}) \big)\quad
 \text{for all } k \in \N_0
 \text{ and all } z \in \NN_{\ell_k}^+ \setminus \MM_{\ell_k},
\end{equation} 
i.e., the non-marked vertices are controlled by the marked ones.
Then, the corresponding subsequence of spatial error estimates converges to zero,
i.e., $\est_{\ell_k}(\NN_{\ell_k}^+) \to 0$ as $k \to \infty$.
\end{proposition}

\begin{remark}\label{remark:marking}
The marking strategies employed in Criteria~\ref{marking:A}--\ref{marking:D}, i.e., the D\"orfler marking 
strategy and the maximum criterion, satisfy the properties~\eqref{eq1:prop:conv:parametric}--\eqref{eq1:prop:conv:spatial} assumed in Propositions~\ref{prop:conv:parametric}--\ref{prop:conv:spatial}.
For example, let us show that~\eqref{eq1:prop:conv:parametric} holds for parametric error indicators
(the same arguments will apply to spatial error indicators).
Suppose that the $\ell_k$-th step of the adaptive algorithm employs the maximum criterion, i.e.,
\begin{equation*}
 \MMM_{\ell_k} := \{ \mu \in \QQQ_{\ell_k} : \est_{\ell_k}(\mu) \ge (1-\thetaP) \, \max_{\nu \in \QQQ_{\ell_k}} \est_{\ell_k}(\nu) \}.
\end{equation*}
Then, for $\mu \in \QQQ_{\ell_k} \setminus \MMM_{\ell_k}$, there holds
\begin{equation*}
 \est_{\ell_k}(\mu) < (1-\thetaP) \, \max_{\nu \in \QQQ_{\ell_k}} \est_{\ell_k}(\nu)
 \le (1-\thetaP) \, \est_{\ell_k}(\MMM_{\ell_k}),
\end{equation*}
which is~\eqref{eq1:prop:conv:parametric} with $g_\P(s) := (1-\thetaP) s$.

Similarly, suppose that the $\ell_k$-th step of the algorithm employs D\"orfler marking,~i.e.,
\begin{equation*}
 \MMM_{\ell_k} \subseteq \QQQ_{\ell_k} 
 \text{ satisfies }\,
 \thetaP \, \est_{\ell_k}(\QQQ_{\ell_k}) \le \est_{\ell_k}(\MMM_{\ell_k}).
\end{equation*}
Then, for $\mu \in \QQQ_{\ell_k} \setminus \MMM_{\ell_k}$, one has
\begin{equation*}
 \est_{\ell_k}(\mu) \le \est_{\ell_k}(\QQQ_{\ell_k} \setminus \MMM_{\ell_k}) =
 \big( \est_{\ell_k}(\QQQ_{\ell_k})^2 - \est_{\ell_k}(\MMM_{\ell_k})^2 \big)^{1/2} \le
 (1 - \thetaP^{2})^{1/2} \, \thetaP^{-1} \, \est_{\ell_k}(\MMM_{\ell_k}),
\end{equation*}
which is~\eqref{eq1:prop:conv:parametric} with $g_\P(s) := (1 - \thetaP^{2})^{1/2} \, \thetaP^{-1} \, s$.
\end{remark}

With the aforegoing propositions, we can proceed to the proof of our first main result.

\begin{proof}[Proof of Theorem~\ref{thm:plain_convergence}]
We divide the proof into three steps.

{\bf Step~1.} Consider Algorithms~\ref{algorithm}.\ref{marking:A} and~\ref{algorithm}.\ref{marking:C}.
If case~(a) in the corresponding marking strategies occurs only finitely many times,
then there exists $\ell_0 \in \N$ such that case~(b) (i.e., parametric enrichment) occurs for all $\ell \ge \ell_0$.
Then, according to the criterion used to decide on the type of enrichment, one has
$0 \le \est_\ell(\NN_\ell^+) < \vartheta \, \est_\ell(\QQQ_\ell)$ for all $\ell \ge \ell_0$.
Since $\est_\ell(\QQQ_\ell) \to 0$ as $\ell \to \infty$ by Proposition~\ref{prop:conv:parametric},
we conclude that $\est_\ell \to 0$ as $\ell \to \infty$.

If case~(b) in Marking criteria~\ref{marking:A} and~\ref{marking:C} occurs finitely many times,
then there exists $\ell_0 \in \N$ such that only case~(a) (i.e., spatial refinement) occurs for all $\ell \ge \ell_0$.
Hence, $0 \le \vartheta \, \est_\ell(\QQQ_\ell) \le \est_\ell(\NN_\ell^+)$ for all $\ell \ge \ell_0$.
Since $\est_\ell(\NN_\ell^+) \to 0$ as $\ell \to \infty$ by Proposition~\ref{prop:conv:spatial},
we conclude that $\est_\ell \to 0$ as $\ell \to \infty$.

Finally, if both cases~(a) and~(b) happen infinitely often, we split the sequence $(\est_\ell)_{\ell \in \N_0}$ into two disjoint subsequences:
$(\est_{\ell_k^{\rm(a)}})_{k \in \N}$, where only case~(a) occurs, and $(\est_{\ell_k^{\rm(b)}})_{k \in \N}$, where only case~(b) occurs.
With the preceding argument, it follows that $\est_{\ell_k^{\rm(a)}}, \est_{\ell_k^{\rm(b)}} \to 0$ as $k \to \infty$.
This implies the convergence of the sequence $\est_\ell \to 0$ as $\ell \to \infty$.

{\bf Step~2.} Let us now consider Algorithm~\ref{algorithm}.\ref{marking:B}.
We argue as in Step~1.
If case~(a) in Marking criterion~\ref{marking:B} occurs only finitely many times,
then there exists $\ell_0 \in \N$ such that case~(b) (i.e., parametric enrichment) occurs for all $\ell \ge \ell_0$.
Then, according to the criterion used to decide on the type of enrichment, one has
\begin{equation*}
 0 \le \thetaX \, \est_\ell(\NN_\ell^+) 
 \le \est_\ell(\widetilde\MM_\ell)
 \le \est_\ell(\widetilde\RR_\ell)
 < \vartheta \, \est_\ell(\MMM_\ell) 
 \le \vartheta \, \est_\ell(\QQQ_\ell)\quad
 \hbox{for all $\ell \ge \ell_0$}.
\end{equation*}
Since $\est_\ell(\QQQ_\ell) \to 0$ as $\ell \to \infty$ by Proposition~\ref{prop:conv:parametric}, we conclude that $\est_\ell \to 0$ as $\ell \to \infty$.

If case~(b) in Marking criterion~\ref{marking:B} occurs finitely many times,
then there exists $\ell_0 \in \N$ such that only case~(a) (i.e., spatial refinement) occurs for all $\ell \ge \ell_0$ and hence
\begin{equation*}
 0 \le \thetaP \, \est_\ell(\QQQ_\ell)
 \le \est_\ell(\widetilde\MMM_\ell)
 < \vartheta^{-1} \, \est_\ell(\widetilde\RR_\ell)
 \le \vartheta^{-1} \, \est_\ell(\NN_\ell^+)\quad
 \hbox{for all $\ell \ge \ell_0$}.
\end{equation*}
Since $\est_\ell(\NN_\ell^+) \to 0$ as $\ell \to \infty$ by Proposition~\ref{prop:conv:spatial},
we conclude that $\est_\ell \to 0$ as $\ell \to \infty$.

If both cases (a) and (b) occur infinitely often, then we proceed as in Step~1 to show that $\est_\ell \to 0$ as $\ell \to \infty$.

{\bf Step~3.} Finally, consider Algorithm~\ref{algorithm}.\ref{marking:D}.
Arguing as for Algorithm~\ref{algorithm}.\ref{marking:B} in Step~2, we prove that
\begin{equation}\label{eq2:prop:conv:parametric:new}
 \est_\ell(\NN_\ell^+) \to 0 
 \quad \text{as well as} \quad
 \est_\ell(\widetilde\MMM_\ell) \to 0
 \quad \text{as } \ell \to \infty. 
\end{equation}
It remains to show that $\est_\ell(\QQQ_\ell) \to 0$ as $\ell \to \infty$.
By Proposition~\ref{prop:conv:parametric:new}, there exists a sequence $(\est_\infty(\nu))_{\nu \in \III}$ satisfying~\eqref{eq:prop:conv:parametric:new}.
In particular, $\sup_{\nu \in \III} \est_\infty(\nu) < \infty$.
Let $\eps >0$ and choose $\mu \in \III$ such that
\begin{equation*}
 \sup_{\nu \in \III} \est_\infty(\nu) \le \est_\infty(\mu) + \eps.
\end{equation*}
Together with~\eqref{eq:prop:conv:parametric:new} and~\eqref{eq2:prop:conv:parametric:new},
the triangle inequality yields that
\begin{equation*}
 0 \le \est_\infty(\mu) + \eps 
 \le |\est_\infty(\mu) - \est_\ell(\mu)| + \est_\ell(\mu) + \eps
 \le |\est_\infty(\mu) - \est_\ell(\mu)| + \est_\ell(\widetilde\MMM_\ell) + \eps
 \xrightarrow{\ell \to \infty} \eps.
\end{equation*}
Since $\eps > 0$ is arbitrary, we conclude that $\est_\infty(\nu) = 0$ for all $\nu \in \III$.
With~\eqref{eq:prop:conv:parametric:new},
this proves that $\est_\ell(\QQQ_\ell)^2 = \sum_{\nu \in \III} \est_\ell(\nu)^2 \to 0$ as $\ell \to \infty$.
\end{proof}

\section{Proof of Propositions~\ref{prop:conv:parametric:new},~\ref{prop:conv:parametric}, and~\ref{prop:conv:spatial}}
\label{sec:auxiliary}

In this section, we collect some auxiliary results and prove 
Propositions~\ref{prop:conv:parametric:new},~\ref{prop:conv:parametric}, and~\ref{prop:conv:spatial}.

\subsection{\textsl{A~priori} convergence of adaptive algorithms}
The following lemma is an early result from~\cite{bv84} which proves that adaptive algorithms
(without coarsening) always lead to convergence of the discrete solutions.

\begin{lemma}[\textsl{a~priori} convergence]\label{lemma:apriori}
Let $V$ be a Hilbert space. Let $a: V \times V \to \R$ be an elliptic and continuous bilinear form.
Let $F \in V^*$ be a linear and continuous functional.
For each $\ell \in \N_0$, let $V_\ell \subseteq V$ be a closed subspace such that $V_\ell \subseteq V_{\ell+1}$.
Furthermore, define the limiting space $V_\infty := \overline{\bigcup_{\ell = 0}^\infty V_\ell} \subseteq V$.
Then, for all $\ell \in \N_0 \cup \{\infty\}$, there exists a unique Galerkin solution $u_\ell \in V_\ell$ satisfying
\begin{equation}\label{eq1:lemma:apriori}
 a(u_\ell, v_\ell) = F(v_\ell) 
 \quad \text{for all } v_\ell \in V_\ell.
\end{equation}
Moreover, there holds
\begin{equation*}
 \norm{u_\infty - u_\ell}{V} \xrightarrow{\ell \to \infty} 0.
\end{equation*}
\end{lemma}

\begin{proof}
For each $\ell \in \N_0 \cup \{\infty\}$, the existence and uniqueness of the Galerkin solution $u_\ell \in V_\ell$
satisfying~\eqref{eq1:lemma:apriori} follow from the Lax--Milgram theorem.
Moreover, since $V_\ell \subseteq V_\infty$, $u_\ell$ is also a Galerkin approximation to $u_\infty$.
Therefore, the C\'ea lemma proves that
\begin{equation*}
 \norm{u_\infty - u_\ell}{V} 
 \lesssim \min_{v_\ell \in V_\ell} \norm{u_\infty - v_\ell}{V} 
 \xrightarrow{\ell \to \infty} 0
\end{equation*}
by definition of $V_\infty$.
\end{proof}

\subsection{Proof of Proposition~\ref{prop:conv:parametric:new}}
\label{proof:prop:conv:parametric:new}

For $\nu \in \QQQ_\ell$, recall the functions $e_\ell^\nu \in \X_\ell$ from~\eqref{eq:def:hat-e-ell-nu}.
Define
\begin{equation*}
 \widehat e_\ell' := \sum_{\nu \in \III} e_\ell^\nu P_\nu 
 \in 
 \X_\ell \otimes \widehat\P_\ell
 \stackrel{\eqref{eq:enriched:spaces}}{=}
 \widehat\V_\ell',
 \quad \text{where} \quad
 e_\ell^\nu := 0 \text{ for all } \nu \in \III \setminus \QQQ_\ell.
\end{equation*}
Note that $\est_\ell(\nu) = \enorm{e_\ell^\nu P_\nu}{0}$ for all $\nu \in \QQQ_\ell$
and define $\est_\ell(\nu) := \enorm{e_\ell^\nu P_\nu}{0} = 0$ for all
$\nu \in \III \setminus \QQQ_\ell$.
The next lemma shows that the sequence $\widehat e_\ell'$ converges to some limit $\widehat e_\infty'$ in~$\V$.

\begin{lemma}\label{lemma:conv:parametric:new}
There exists a sequence $(e_\infty^\nu)_{\nu \in \III} \subset \X$ such that
$\widehat e_{\infty}' := \sum_{\nu \in \III} e_\infty^\nu P_\nu \in \V$ satisfies 
\begin{equation}\label{eq:lemma:e_infty}
 \enorm{\widehat e_\infty'}{0}^2 = \sum_{\nu \in \III} \enorm{e_\infty^\nu P_\nu}{0}^2 < \infty
 \quad \text{and} \quad
 \enorm{\widehat e_{\infty}' - \widehat e_{\ell}'}{0}^2
 = \sum_{\nu \in \III} \enorm{(e_\infty^\nu - e_\ell^\nu) \, P_\nu}{0}^2
 \xrightarrow{\ell \to \infty} 0.
\end{equation}%
\end{lemma}

\begin{proof}
The tensor-product structure of $\widehat\V_\ell' = \X_\ell \otimes \widehat\P_\ell$ and
pairwise orthogonality of subspaces $X_\ell \otimes \hull\{P_\nu\}$ ($\nu \in \III$)
with respect to $B_0(\cdot,\cdot)$ imply that
\begin{equation*}
 B_0(\widehat e_\ell', v_\ell P_\nu)
 \stackrel{\eqref{eq1:lemma:orthogonal}}{=} B_0(e_\ell^\nu P_\nu, v_\ell P_\nu) 
 \stackrel{\eqref{eq:def:hat-e-ell-nu}}{=} F(v_\ell P_\nu) - B(u_\ell, v_\ell P_\nu) 
 \quad \text{for all } \nu \in \QQQ_\ell \text{ and } v_\ell \in \X_\ell.
\end{equation*}
Moreover, there holds
\begin{equation*}
 B_0(\widehat e_\ell', v_\ell P_\nu)
 \stackrel{\eqref{eq1:lemma:orthogonal}}{=} B_0(e_\ell^\nu P_\nu, v_\ell P_\nu) 
 = 0 
\stackrel{\eqref{eq:discrete_formulation}}{=} F(v_\ell P_\nu) - B(u_\ell, v_\ell P_\nu)
\text{ for all } \nu \in \PPP_\ell \text{ and } v_\ell \in \X_\ell.
\end{equation*}
Hence, $\widehat e_\ell' \in \widehat\V_\ell'$ is the unique solution of the variational problem
\begin{equation}\label{eq:widehat_e'}
 B_0(\widehat e_\ell',\widehat v_\ell') = F(\widehat v_\ell') - B(u_\ell,\widehat v_\ell')
 \quad \text{for all } \widehat v_\ell' \in \widehat\V_\ell'.
\end{equation}%
Lemma~\ref{lemma:apriori} proves that $\enorm{u_\infty - u_\ell}{} \to 0$ as $\ell \to \infty$ for some $u_\infty \in \V$.
Consider the unique solution $\widehat e_\ell'' \in \widehat\V_\ell'$ of the auxiliary problem
\begin{equation}\label{eq:widehat_e''}
 B_0(\widehat e_\ell'',\widehat v_\ell') = F(\widehat v_\ell') - B(u_\infty,\widehat v_\ell')
 \quad \text{for all } \widehat v_\ell' \in \widehat\V_\ell'.
\end{equation}%
Since $\widehat\V_\ell' \subseteq \widehat\V_{\ell+1}'$, Lemma~\ref{lemma:apriori} also proves
that $\enorm{\widehat e_\infty' - \widehat e_\ell''}{} \to 0$ as $\ell \to \infty$ for some $\widehat e_\infty' \in \V$. 
Exploiting~\eqref{eq:widehat_e'} and~\eqref{eq:widehat_e''} for
$\widehat v_\ell' = \widehat e_\ell'' - \widehat e_\ell' \in \widehat\V_\ell'$, we see that
\begin{equation*}
 \enorm{\widehat e_\ell'' - \widehat e_\ell'}{0}^2
 = B_0(\widehat e_\ell'' - \widehat e_\ell', \widehat e_\ell'' - \widehat e_\ell')
 = - B(u_\infty - u_\ell, \widehat e_\ell'' - \widehat e_\ell')
 \le \enorm{u_\infty - u_\ell}{} \enorm{\widehat e_\ell'' - \widehat e_\ell'}{}.
\end{equation*}
With the norm equivalence $\enorm{\cdot}{0} \simeq \enorm{\cdot}{}$, the triangle inequality thus proves that
\begin{equation*}
 \enorm{\widehat e_{\infty}' - \widehat e_{\ell}'}{0}
 \le \enorm{\widehat e_\infty' - \widehat e_\ell''}{0} + \enorm{\widehat e_\ell'' - \widehat e_\ell'}{0} 
 \lesssim \enorm{\widehat e_\infty' - \widehat e_\ell''}{0} + \enorm{u_\infty - u_\ell}{} 
 \xrightarrow{\ell \to \infty} 0.
\end{equation*}
Hence, the proof is concluded by noticing that the existence of $(e_\infty^\nu)_{\nu \in \III} \subset \X$
is a consequence of the representation in~\eqref{eq:representation} and that the equalities in~\eqref{eq:lemma:e_infty}
then immediately follow from~\eqref{eq2:lemma:orthogonal}.
\end{proof}

With the above result, we can proceed to the proof of Proposition~\ref{prop:conv:parametric:new}.

\begin{proof}[Proof of Proposition~\ref{prop:conv:parametric:new}]
Lemma~\ref{lemma:conv:parametric:new} provides a sequence $(e_\infty^\nu)_{\nu \in \III} \subset \X$
satisfying~\eqref{eq:lemma:e_infty}.
For each $\nu\in \III$, we define $\est_\infty(\nu) := \enorm{e_\infty^\nu P_\nu}{0}$.
From~\eqref{eq:lemma:e_infty} it follows that
\begin{equation*}
 \sum_{\nu \in \III} \est_\infty(\nu)^2 = \sum_{\nu \in \III} \enorm{e_\infty^\nu P_\nu}{0}^2 < \infty,
\end{equation*}
and using \eqref{eq:lemma:e_infty} together with the definition of $\est_\ell(\nu)$ in~\eqref{eq:parametric-error-estimate}
we find~that
\begin{equation*}
 \sum_{\nu \in \III} |\est_\infty(\nu) - \est_\ell(\nu)|^2
 = \sum_{\nu \in \III} \big( \enorm{e_\infty^\nu P_\nu}{0} - \enorm{e_\ell^\nu P_\nu}{0} \big)^2
 \le \sum_{\nu \in \III} \enorm{e_\infty^\nu P_\nu - e_\ell^\nu P_\nu}{0}^2
 \xrightarrow{\ell \to \infty} 0.
\end{equation*}
This yields~\eqref{eq:prop:conv:parametric:new} and concludes the proof.
\end{proof}

\subsection{Proof of Proposition~\ref{prop:conv:parametric}}
\label{proof:prop:conv:parametric}%

We first state an auxiliary result for square summable sequences.

\begin{lemma}\label{lemma:parametric}
Let $g : \R_{\ge0} \to \R_{\ge0}$ be a continuous function with $g(0) = 0$.
Let $(x_n)_{n\in\N} \subset \R_{\ge0}$ with $\sum_{n=1}^\infty x_n^2 < \infty$.
For $k \in \N_0$, let $(x_n^{(k)})_{n\in\N} \subset \R_{\ge0}$ with
$\sum_{n=1}^\infty (x_n - x_n^{(k)})^2 \to 0$ as $k \to \infty$.
In addition, let $(\PP_k)_{k \in \N_0}$ be a sequence of nested subsets of $\N$
(i.e., $\PP_k \subseteq \PP_{k+1}$ for all $k \in \N_0$) satisfying the following property:
\begin{equation} \label{eq:lemma:parametric}
x_m^{(k)} \le g \Bigg( \sum_{n \in \PP_{k+1} \setminus \PP_k} (x_n^{(k)})^2 \Bigg)\quad
\text{for all } k \in \N_0 \text{ and } m \in \N \setminus \PP_{k+1}.
\end{equation}
Then 
$\sum_{n \in \N \setminus \PP_k} x_n^2 \to 0$ as $k \to \infty$.
\end{lemma}

\begin{proof}
We divide the proof into 3 steps.

{\bf Step~1.} First, we show that $\min(\PP_{k+1} \setminus \PP_k) \to \infty$ as $k \to \infty$, where $\min(\emptyset) := \infty$.
This statement is trivial if there exists $K \in \N$ such that $\PP_k = \PP_{k+1}$ for all $k \ge K$.
Therefore, without loss of generality, we can consider a sequence of strictly nested sets,
i.e., {$\PP_k \subset \PP_{k+1}$} for all $k \in \N_0$.
We argue by contradiction and assume the existence of $C>0$ such that, for all $k_0 \in \N_0$, there exists
$k \ge k_0$ such that $M_k := \min(\PP_{k+1} \setminus \PP_k) \le C$.
In particular, we can construct a monotonic increasing sequence $(k_j)_{j \in \N_0} \subset \N_0$, i.e., $k_j \leq k_{j+1}$
for all $j \in \N_0$, and consider the corresponding bounded sequence $(M_{k_j})_{j \in \N_0}$.
Since this sequence is bounded, we can extract a convergent subsequence (not relabeled) and denote
its limit by $m := \lim_{j \to \infty} M_{k_j}$.
Since $(M_{k_j})_{j \in \N_0} \subset \N$, it follows that there exists $i \in \N_0$ such that $m = M_{k_j}$ for all $j \ge i$.
In particular, $m = M_{k_i}$ and $m = M_{k_{i+1}}$, so that $m \in \PP_{k_i+1} \cap \PP_{k_{i+1}+1} \, = \PP_{k_i+1}$.
On the other hand, since the sets are nested and $k_i+1 \leq k_{i+1}$,
we conclude that {$\PP_{k_i+1} \subseteq \PP_{k_{i+1}}$}.
This leads to a contradiction:
\begin{equation*}
m = M_{k_{i+1}}
= \min(\PP_{k_{i+1}+1} \setminus \PP_{k_{i+1}})
\in \PP_{k_{i+1}+1} \setminus \PP_{k_{i+1}}
\subseteq \PP_{k_{i+1}+1} \setminus \PP_{k_i+1}
\not\ni m.
\end{equation*}

{\bf Step~2.} Next, let us establish some auxiliary convergence statements.
Using the sum\-mabi\-li\-ty assumption on $(x_n)_{n\in\N}$ and the convergence assumption on $(x_n^{(k)})_{n\in\N}$ ($k \in \N_0$),
it follows from Step~1 that
\begin{equation*}
 \Bigg( \sum_{n \in \PP_{k+1} \setminus \PP_k} (x_n^{(k)})^2 \Bigg)^{1/2}
 \le \Bigg( \sum_{n = 1}^\infty (x_n - x_n^{(k)})^2 \Bigg)^{1/2}
 + \Bigg( \sum_{n = \min(\PP_{k+1} \setminus \PP_k)}^\infty x_n^2 \Bigg)^{1/2}
 \xrightarrow{k \to \infty} 0.
\end{equation*}
Therefore, considering the set
\begin{equation*}
\PP_\infty^c := \{ n \in \N : n \not\in \PP_k \text{ for all } k \in \N_0 \},
\end{equation*}
we deduce from~\eqref{eq:lemma:parametric} that
\begin{equation} \label{eq:old_step3}
 0 \le x_m^{(k)} \le g \Bigg( \sum_{n \in \PP_{k+1} \setminus \PP_k} (x_n^{(k)})^2 \Bigg)
 \xrightarrow{k \to \infty} 0
 \quad \text{for all } m \in \PP_\infty^c.
\end{equation}
To conclude this step, let us show that
\begin{equation} \label{eq:new_step4}
\min((\N \setminus \PP_{k+1}) \setminus \PP_\infty^c) \to \infty
\quad \text{as} \quad k \to \infty,
\quad \text{where } \min(\emptyset) := \infty.
\end{equation}
Let $m_k := \min((\N \setminus \PP_{k+1}) \setminus \PP_\infty^c)$ for all $k \in \N_0$.
Note that the sequence $(m_k)_{k\in\N_0}$ is monotonic increasing,
because the sets are nested.
Since $m_k \in (\N \setminus \PP_{k+1}) \setminus \PP_\infty^c$, there exists $j_0 \in \N_0$ with $j_0>k+1$ such that $m_k \in \PP_{j_0}$.
Therefore, since the sets are nested, we conclude that $m_k \in \PP_j$ for all $j \geq j_0$.
In particular, $m_j \geq m_k +1$ for all $j \geq j_0$.
Together with monotonicity of $(m_k)_{k\in\N_0}$, this implies that $\lim_{k \to \infty} m_k  = \infty$, which yields~\eqref{eq:new_step4}.

{\bf Step~3.} Finally, let us show that $\sum_{n \in \N \setminus \PP_k} x_n^2 \to 0$ as $k \to \infty$.
Let $N \in \N$ be an arbitrary free parameter
and consider the following sets:
\begin{align*}
\AA_k^1[N] &:= (\N \setminus \PP_k) \cap \{ n \in \N : n \geq N \}, \\
\AA_k^2[N] &:= (\N \setminus \PP_{k+1}) \cap \{ n \in \N : n < N \} \cap \PP_\infty^c, \\
\AA_k^3[N] &:= (\N \setminus \PP_{k+1}) \cap \{ n \in \N : n < N \} \setminus \PP_\infty^c, \\
\AA_k^4[N] &:= (\PP_{k+1} \setminus \PP_k) \cap \{ n \in \N : n < N \}.
\end{align*}
Note that this defines a disjoint partition of $\N \setminus \PP_k$, i.e.,
\begin{equation*}
  \N \setminus \PP_k = \AA_k^1[N] \cup \AA_k^2[N] \cup \AA_k^3[N] \cup \AA_k^4[N]\quad
  \hbox{and}\quad
  \AA_k^i[N] \cap \AA_k^j[N] =\emptyset\ \ \hbox{for $i \not= j$}.
\end{equation*}
For the sum over the set $\AA_k^2[N]$, we have
\begin{equation*}
\sum_{n \in  \AA_k^2[N]} x_n^2
\lesssim \sum_{n \in  \AA_k^2[N]} (x_n^{(k)})^2 + \sum_{n \in  \AA_k^2[N]} \big(x_n - x_n^{(k)}\big)^2
\leq \sum_{n \in  \AA_k^2[N]} (x_n^{(k)})^2 + \sum_{n \in  \N} \big(x_n - x_n^{(k)}\big)^2.
\end{equation*}
The second sum on the right-hand side of this estimate converges to $0$ as $k \to \infty$ by assumption,
whereas
the first sum is finite, and therefore also converges to $0$ as $k \to \infty$ because of~\eqref{eq:old_step3}.

For the sums over the sets $\AA_k^3[N]$ and $\AA_k^4[N]$,
we use the convergence result in~\eqref{eq:new_step4} and the result of Step~1, respectively.
Along with the summability assumption on $(x_n)_{n\in\N}$, this proves that
\begin{equation*}
\sum_{n \in  \AA_k^3[N]} x_n^2
\leq \sum_{n \in (\N \setminus \PP_{k+1}) \setminus \PP_\infty^c} x_n^2
\leq \sum_{n = \min ( (\N \setminus \PP_{k+1}) \setminus \PP_\infty^c )}^{\infty} x_n^2
\xrightarrow{k \to \infty} 0
\end{equation*}
and
\begin{equation*}
\sum_{n \in  \AA_k^4[N]} x_n^2
\le \sum_{n \in \PP_{k+1} \setminus \PP_k} x_n^2
\le \sum_{n = \min(\PP_{k+1} \setminus \PP_k)}^\infty x_n^2
\xrightarrow{k \to \infty} 0.
\end{equation*}
We have thus shown that
\begin{equation*}
\sum_{n \in  \AA_k^2[N]} x_n^2
+ \sum_{n \in  \AA_k^3[N]} x_n^2
+ \sum_{n \in  \AA_k^4[N]} x_n^2
\xrightarrow{k \to \infty} 0\quad \hbox{for all $N \in \N$}.
\end{equation*}
In particular, for all $N \in \N$, one has
\begin{equation*}
0
\le \liminf_{k \to \infty} \Bigg( \sum_{n \in \N \setminus \PP_k} x_n^2 \Bigg)
\le \limsup_{k \to \infty} \Bigg( \sum_{n \in \N \setminus \PP_k} x_n^2 \Bigg)
= \limsup_{k \to \infty} \Bigg(\sum_{n \in  \AA_k^1[N]} x_n^2 \Bigg)
 \le \sum_{n = N}^\infty x_n^2.
\end{equation*}
Thus, the limit inferior and the limit superior of $\sum_{n \in \N \setminus \PP_k} x_n^2$
are non-negative and bounded from above by a
tail of the convergent series.
Since $N$ is arbitrary, this leads to the desired convergence
$\sum_{n \in \N \setminus \PP_k} x_n^2 \to 0$ as $k \to \infty$.
This concludes the proof.
\end{proof}

With this lemma, we can proceed to the proof of Proposition~\ref{prop:conv:parametric}.

\begin{proof}[Proof of Proposition~\ref{prop:conv:parametric}]
Proposition~\ref{prop:conv:parametric:new} yields a sequence $(\est_{\infty}(\nu))_{\nu \in \III}$ such that
\begin{equation*}
\sum_{\nu \in \III} \est_\infty(\nu)^2 < \infty\quad
\text{and}\quad
\sum_{\nu \in \III} \big( \est_\infty(\nu) - \est_{\ell_k}(\nu) \big)^2
\xrightarrow{k \to \infty} 0.	
\end{equation*}
Let $g:\R_{\geq 0} \to \R_{\geq 0}$ be a continuous function
defined by $g(s) \,:=\, g_\P (\sqrt{s})$ for all $s \in \R_{\geq 0}$.
Setting $\est_{\ell_k}(\mu) = 0$ for $\mu \in \III \setminus \QQQ_{\ell_k}$,
we deduce from~\eqref{eq1:prop:conv:parametric} that
\begin{equation*}
\est_{\ell_k}(\mu) \le g_\P \big( \est_{\ell_k}(\MMM_{\ell_k}) \big)
= g \big( \est_{\ell_k}(\MMM_{\ell_k})^2 \big)
\quad \text{for all } k \in \N_0 \text{ and } \mu \in \III \setminus \MMM_{\ell_k}.
\end{equation*}
Note that the index set $\III$ is countable, since it can be understood as a countable union of countable sets,
and that
$\PPP_{\ell_n} \subseteq \PPP_{\ell_n+1} \subseteq \PPP_{\ell_{n+1}}$, since $\ell_n+1 \le \ell_{n+1}$.
Therefore, we can establish a one-to-one map between $\III$ and $\N$, which allows us to identify
each index set $\PPP_{\ell_k} \subset \III$ ($k \in \N_0$) with a set $\PP_k \subset \N$.
Then $\PP_k \subseteq \PP_{k+1}$ and
applying Lemma~\ref{lemma:parametric} to the sequences 
$(x_n)_{n \in \N} \,:=\, (\est_\infty(\nu))_{\nu \in \III}$,
$(x_n^{(k)})_{n \in \N} \,:=\, (\est_{\ell_k}(\nu))_{\nu \in \III}$,
we prove that
\begin{equation*}
 \sum_{\nu \in \III \setminus \PPP_{\ell_k}} \est_\infty(\nu)^2
 \xrightarrow{k \to \infty} 0.
\end{equation*}
Note that the sequence $(z_\ell)_{\ell \in \N_0} := \big(\sum_{\nu \in \III \setminus \PPP_{\ell}} \est_\infty(\nu)^2\big)_{\ell \in \N_0}$
is monotonic decreasing and bounded from below. Hence, it is convergent.
Moreover, it has a subsequence that converges to zero. We therefore conclude that
\begin{equation*}
 \sum_{\nu \in \QQQ_\ell} \est_\infty(\nu)^2
 \le \sum_{\nu \in \III \setminus \PPP_\ell} \est_\infty(\nu)^2
 \xrightarrow{\ell \to \infty} 0.
\end{equation*}
Overall, we derive that
\begin{equation*}
\est_\ell(\QQQ_\ell)^2 = \sum_{\nu \in \QQQ_\ell} \est_\ell(\nu)^2
\lesssim \sum_{\nu \in \QQQ_\ell} \est_\infty(\nu)^2
 + \sum_{\nu \in \III} \big(\est_\infty(\nu) - \est_\ell(\nu)\big)^2
 \xrightarrow{\ell \to \infty} 0.
\end{equation*}
This concludes the proof.
\end{proof}

\subsection{Proof of Proposition~\ref{prop:conv:spatial}}
\label{proof:prop:conv:spatial}%

The proof of Proposition~\ref{prop:conv:spatial} essentially follows the same lines as that
of Theorem~2.1 in~\cite{msv08}. Therefore, here we only sketch the proof by demonstrating
how the results of~\cite{msv08} for 
deterministic problems can be extended to the parametric setting in the present paper.

We start by observing that the variational problem~\eqref{eq:weakform}, its discretization,
and the proposed adaptive algorithm satisfy the general framework described in~\cite[Section~2]{msv08}:
\begin{itemize}
\item the variational formulation~\eqref{eq:weakform} clearly fits into the class of problems considered in~\cite[\S2.1]{msv08};
\item our Galerkin discretization~\eqref{eq:discrete_formulation} satisfies the assumptions in~\cite[eqs.~(2.6)--(2.8)]{msv08};
\item the spatial NVB refinement considered in the present paper satisfies the assumptions
on the mesh refinement in~\cite[eqs.~(2.5) and~(2.14)]{msv08};
\item the weak marking condition~\eqref{eq1:prop:conv:spatial} in Proposition~\ref{prop:conv:spatial} is the same as the
marking condition in~\cite[eq.~(2.13)]{msv08};
\item finally, we prove in Lemma~\ref{lem:prop:spatial} below
that the local discrete efficiency estimate holds in the parametric seeting (cf.~\cite[eq.~(2.9b)]{msv08}).
Note that the global reliability of the estimator (see~\eqref{eq:reliability} and~\cite[eq.~(2.9a)]{msv08})
is not exploited in this section (and hence, not needed for the proof of Theorem~\ref{thm:plain_convergence}).
The reliability is only needed to establish convergence of the true error, i.e., $\enorm{u-u_\ell}{} \to 0$ as $\ell \to \infty$ (see Corollary~\ref{cor:plain_convergence}).
\end{itemize}

We will use the following notation:
For $\omega \subset D$, we define
\begin{equation*}
 B_\omega(v,w) := \int_\G \int_\omega a_0 \nabla u\cdot\nabla v \, \dx \, \dpi(\y) + \sum_{m=1}^\infty \int_\G \int_\omega y_m a_m \nabla u\cdot\nabla v \, \dx \, \dpi(\y)
 \text{ for } v, w \in \V.
\end{equation*}
Note that $B_\omega(\cdot,\cdot)$ is symmetric, bilinear, and positive semi-definite.
We denote by $\enorm{v}{\omega} := B_\omega(v,v)^{1/2}$ the corresponding induced semi-norm.
Furthermore, in addition to the limiting space $\V_\infty$ introduced in Lemma~\ref{lemma:apriori},
we define the spatial limiting space $\X_\infty := \bigcup_{\ell = 0}^\infty \X_\ell$.

\begin{lemma} \label{lem:prop:spatial}
Let $z \in \NN_\ell^+$ and denote by
$\omega_\ell(z) := \bigcup \{T \in \TT_\ell : z \in T \}$ the associated vertex patch.
Then the following estimate holds:
\begin{subequations}
\label{eq1:msv}
\begin{equation}
\label{eq1b:msv}
\est_\ell(z)
\leq C \enorm{u - u_\ell}{\omega_\ell(z)}.
\end{equation}
Furthermore, let $u_\infty \in \V$ be the limit of $(u_\ell)_{\ell \in \N_0}$ guaranteed by Lemma~\ref{lemma:apriori}.
If $\widehat\varphi_{\ell,z} \in \X_\infty$, then there holds
\begin{equation}
\label{eq2b:msv}
\est_\ell(z)
\leq C \enorm{u_\infty - u_\ell}{\omega_\ell(z)}.
\end{equation}%
\end{subequations}
The constant $C>0$ in~\eqref{eq1b:msv} and~\eqref{eq2b:msv} depends only on $a_0$ and $\tau$.
\end{lemma}

\begin{proof}
We recall the definition of the spatial error indicators in~\eqref{eq:spatial-error-estimate}:
\begin{equation*}
 \est_\ell(z)^2 
 = \sum_{\nu \in \PPP_\ell}
 \frac{|F(\widehat\varphi_{\ell,z}P_\nu) - B(u_\ell,\widehat\varphi_{\ell,z}P_\nu)|^2}{\norm{a_0^{1/2}\nabla \widehat\varphi_{\ell,z}}{L^2(D)}^2}
 = \sum_{\nu \in \PPP_\ell} \enorm{\GG_{\ell,z,\nu} e_\ell}{0}^2
 \quad \text{for all } z \in \NN_\ell^+,
\end{equation*}%
where $\GG_{\ell,z,\nu} : \V \to {\rm span}\{\widehat\varphi_{\ell,z} P_\nu\}$ is the orthogonal projection
onto the one-dimensional space ${\rm span}\{\widehat\varphi_{\ell,z} P_\nu\}$ with respect to
$B_0(\cdot,\cdot)$, and $e_\ell \in \widehat\X \otimes \P_\ell$ solves
\begin{equation*}
 B_0(e_\ell, v_\ell) = F(v_\ell) - B(u_\ell,v_\ell)
 \quad \text{for all } v_\ell \in \widehat\X \otimes \P_\ell.
\end{equation*}
Note that the functions $\{ \widehat\varphi_{\ell,z} P_\nu : \nu \in \PPP_\ell \}$ are orthogonal
with respect to $B_0(\cdot,\cdot)$.
Hence, $\sum_{\nu \in \PPP_\ell} \GG_{\ell,z,\nu} : \V \to {\rm span} \{\widehat\varphi_{\ell,z} P_\nu : \nu \in \PPP_\ell \} \subset \widehat\X \otimes \P_\ell$
is an orthogonal projection with respect to $B_0(\cdot,\cdot)$ as well.
This yields that
\begin{equation*}
\begin{split}
 \est_\ell(z)^2
 & = \sum_{\nu \in \PPP_\ell} \enorm{\GG_{\ell,z,\nu} e_\ell}{0}^2
 = \enorm[\Big]{\sum_{\nu \in \PPP_\ell} \GG_{\ell,z,\nu} e_\ell}{0}^2
 = B_0\Big(e_\ell, \sum_{\nu \in \PPP_\ell} \GG_{\ell,z,\nu} e_\ell\Big)
 \\
 & = F\Big(\sum_{\nu \in \PPP_\ell} \GG_{\ell,z,\nu} e_\ell\Big) - B\Big(u_\ell, \sum_{\nu \in \PPP_\ell} \GG_{\ell,z,\nu} e_\ell\Big)
 = B\Big(u - u_\ell, \sum_{\nu \in \PPP_\ell} \GG_{\ell,z,\nu} e_\ell\Big).
\end{split}
\end{equation*}
Note that the spatial support of $\sum_{\nu \in \PPP_\ell} \GG_{\ell,z,\nu} e_\ell$ lies in $\omega := \supp(\widehat\varphi_{\ell,z})$.
Then, the Cauchy--Schwarz inequality shows that
\begin{equation*}
\begin{split}
\enorm[\Big]{\sum_{\nu \in \PPP_\ell} \GG_{\ell,z,\nu} e_\ell}{0}^2
& = 
B\Big(u - u_\ell, \sum_{\nu \in \PPP_\ell} \GG_{\ell,z,\nu} e_\ell\Big)
  = B_\omega\Big(u - u_\ell, \sum_{\nu \in \PPP_\ell} \GG_{\ell,z,\nu} e_\ell\Big)
 \\
& \le \enorm{u - u_\ell}{\omega} \, \enorm[\Big]{\sum_{\nu \in \PPP_\ell} \GG_{\ell,z,\nu} e_\ell}{\omega} 
 \le \enorm{u - u_\ell}{\omega} \, \enorm[\Big]{\sum_{\nu \in \PPP_\ell} \GG_{\ell,z,\nu} e_\ell}{} \\
& \stackrel{\eqref{eq:lambda}}{\simeq}\; \enorm{u - u_\ell}{\omega} \, \enorm[\Big]{\sum_{\nu \in \PPP_\ell} \GG_{\ell,z,\nu} e_\ell}{0}.
\end{split}
\end{equation*}
We have thus shown that
\begin{equation*}
   \est_\ell(z)^{2} =
   B\Big(u - u_\ell, \sum_{\nu \in \PPP_\ell} \GG_{\ell,z,\nu} e_\ell\Big) =\; 
   \enorm[\Big]{\sum_{\nu \in \PPP_\ell} \GG_{\ell,z,\nu} e_\ell}{0}^{2}
   \lesssim \enorm{u - u_\ell}{\omega}^{2}.
\end{equation*}
Since $\omega = \supp(\widehat\varphi_{\ell,z}) \subseteq \omega_\ell(z)$, this proves~\eqref{eq1b:msv}.

Finally, if $\widehat\varphi_{\ell,z} \in \X_\infty$, then $\sum_{\nu \in \PPP_\ell} \GG_{\ell,z,\nu} e_\ell \in \V_\infty$.
Therefore, the same arguments as above yield~\eqref{eq2b:msv}.
\end{proof}

Note that in the present setting, the estimates~\eqref{eq1b:msv} and~\eqref{eq2b:msv} from Lemma~\ref{lem:prop:spatial}
replace~\cite[eq.~(2.9b)]{msv08} and~\cite[eq.~(4.11)]{msv08}, respectively.
Having these estimates, we can now proceed to the proof
of Proposition~\ref{prop:conv:spatial}.

\begin{proof}[Proof of Proposition~\ref{prop:conv:spatial}]
Let $\TT_\infty := \bigcup_{k \ge 0} \bigcap_{\ell \ge k} \TT_\ell$ be the set of all elements which remain unrefined after finitely many steps of refinement.
In the spirit of~\cite[eqs.~(4.10)]{msv08}, for all $\ell \in \N_0$, we consider the decomposition
$ \TT_\ell = \TT_\ell^{\rm good} \cup \TT_\ell^{\rm bad} \cup  \TT_\ell^{\rm neither}$,
where
\begin{align*}
 \TT_\ell^{\rm good} &:= \{T \in \TT_\ell : \widehat\varphi_{\ell,z} \in \X_\infty \text{ for all } z \in \NN_\ell^+ \cap T \}, \\
 \TT_\ell^{\rm bad} &:= \{T \in \TT_\ell : T' \in \TT_\infty \text{ for all } T' \in \TT_\ell \text{ with } T \cap T' \neq \emptyset \}, \\
 \TT_\ell^{\rm neither} &:= \TT_\ell \setminus (\TT_\ell^{\rm good} \cup \TT_\ell^{\rm bad}).
\end{align*}
The elements in $\TT_\ell^{\rm good}$ are refined sufficiently many times in order to guarantee~\eqref{eq2b:msv}.
The set $\TT_\ell^{\rm bad}$ consists of all elements such that the whole element patch remains unrefined.
The remaining elements are collected in the set $\TT_\ell^{\rm neither}$.
We note that $\TT_\ell^{\rm good}$ is slightly larger than the corresponding set $\GG_\ell^0$ in~\cite[eq.~(4.10a)]{msv08},
while $\TT_\ell^{\rm bad}$ coincides with the corresponding set $\GG_\ell^+$ in~\cite[eq.~(4.10b)]{msv08}.
As a consequence, $\TT_\ell^{\rm neither}$ is smaller than the corresponding set $\GG_\ell^*$ in~\cite[eq.~(4.10c)]{msv08}.

By arguing as in the proof of Propostion~4.1 in~\cite{msv08}, we exploit
the uniform shape regularity of the mesh $\TT_\ell$ guaranteed by NVB
and use Lemmas~\ref{lem:prop:spatial} and~\ref{lemma:apriori} to prove~that
\begin{equation}\label{eq:msv:step1}
 \sum_{T \in \TT_\ell^{\rm good}} \sum_{z \in \NN_\ell^+ \cap T} \est_\ell(z)^2
 \stackrel{\eqref{eq2b:msv}}{\lesssim} \sum_{T \in \TT_\ell^{\rm good}} \sum_{z \in \NN_\ell^+ \cap T} \enorm{u_\infty - u_\ell}{\omega_\ell(z)}^2
 \lesssim \enorm{u_\infty - u_{\ell}}{}^2 
 \xrightarrow{\ell \to \infty} 0.
\end{equation}

Let $D_\ell^{\rm neither} := \bigcup\{T' \in \TT_\ell : T\cap T' \neq \emptyset \text{ for some } T \in \TT_\ell^{\rm neither} \}$.
Since $\TT_\ell^{\rm neither}$ is contained in the corresponding set $\GG_\ell^*$ in~\cite[eq.~(4.10c)]{msv08},
arguing as in Step~1 of the proof of Proposition~4.2 in~\cite{msv08},
we show that $|D_\ell^{\rm neither}| \to 0$ as $\ell \to \infty$.
Hence, Lemma~\ref{lem:prop:spatial}, uniform shape regularity, and
the fact that the local energy seminorm is absolutely continuous with respect to the Lebesgue measure,
i.e., $\enorm{v}{\omega} \to 0$ as $\vert\omega\vert \to 0$ for all $v \in \V$,
lead to
\begin{equation}\label{eq:msv:step2}
 \sum_{T \in \TT_\ell^{\rm neither}} \sum_{z \in \NN_\ell^+ \cap T} \est_\ell(z)^2
  \stackrel{\eqref{eq1b:msv}}{\lesssim} \sum_{T \in \TT_\ell^{\rm neither}} \sum_{z \in \NN_\ell^+ \cap T} \enorm{u - u_\ell}{\omega_\ell(z)}^2
 \lesssim \enorm{u - u_\ell}{D_\ell^{\rm neither}}^2
 \xrightarrow{\ell \to \infty} 0. 
\end{equation}
We note that~\eqref{eq:msv:step1} and~\eqref{eq:msv:step2} hold independently
of the marking property~\eqref{eq1:prop:conv:spatial},
but rely only on the nestedness of the finite-dimensional subspaces~$\X_\ell \subseteq \X_{\ell+1}$ and
$\V_\ell \subseteq \V_{\ell+1}$ for all $\ell \in \N_0$.

To conclude the proof, it remains to consider the set $\TT_\ell^{\rm bad}$.
Let $(\TT_{\ell_k})_{k \in \N_0}$ be the subsequence of $(\TT_{\ell})_{\ell \in \N_0}$ satisfying~\eqref{eq1:prop:conv:spatial}.
If $z \in \MM_{\ell_k}$ and $T \in \TT_{\ell_k}$ with $z \in T$, then
$T \in \TT_{\ell_k} \setminus \TT_{\ell_k}^{\rm bad} = \TT_{\ell_k}^{\rm good} \cup \TT_{\ell_k}^{\rm neither}$.
Therefore, it follows from~\eqref{eq:msv:step1}--\eqref{eq:msv:step2} that
\begin{equation*}
 \sum_{z \in \MM_{\ell_k}} \est_{\ell_k}(z)^2
 \le \sum_{T \in \TT_{\ell_k}^{\rm good}} \sum_{z \in \NN_{\ell_k}^+ \cap T} \est_{\ell_k}(z)^2
 + \sum_{T \in \TT_{\ell_k}^{\rm neither}} \sum_{z \in \NN_{\ell_k}^+ \cap T} \est_{\ell_k}(z)^2
 \xrightarrow{k \to \infty} 0.
\end{equation*}
This implies that
\begin{equation*}
0 \le \est_{\ell_k}(z)
\stackrel{\eqref{eq1:prop:conv:spatial}}{\leq} g_\X \big( \est_{\ell_k}(\MM_{\ell_k}) \big)
\xrightarrow{k \to \infty} 0
\quad \text{for all } z \in \NN_{\ell_k}^+ \setminus \MM_{\ell_k}.
\end{equation*}
Hence, recalling the definition of $\TT_{\bullet}^{\rm bad}$, we obtain (cf.,~\cite[eq.~(4.17)]{msv08}) 
\begin{equation} \label{eq:msv:step2:1}
\sum_{z \in \NN_{\ell_k}^+ \cap T} \est_{\ell_k}(z)^2 \xrightarrow{k \to \infty} 0
 \quad \text{for all } T \in \TT_{\ell_k}^{\rm bad}.
\end{equation}
Finally, arguing as in Steps~2--5 of the proof of Proposition~4.3 in~\cite{msv08},
we use~\eqref{eq:msv:step2:1} and
apply the Lebesgue dominated convergence theorem to derive that
\begin{equation}\label{eq:msv:step3}
 \sum_{T \in \TT_{\ell_k}^{\rm bad}} \sum_{z \in \NN_{\ell_k}^+ \cap T} \est_{\ell_k}(z)^2
 \xrightarrow{k \to \infty} 0.
\end{equation}
Combining now~\eqref{eq:msv:step1}--\eqref{eq:msv:step3}, we find that
\begin{equation*}
\begin{split}
 &\est_{\ell_k}(\NN_{\ell_k}^+)^2 
 = \sum_{z \in \NN_{\ell_k}^+} \est_{\ell_k}(z)^2
 \\& \quad
 \le \sum_{T \in \TT_{\ell_k}^{\rm good}} \sum_{z \in \NN_{\ell_k}^+ \cap T} \est_{\ell_k}(z)^2
 + \sum_{T \in \TT_{\ell_k}^{\rm neither}} \sum_{z \in \NN_{\ell_k}^+ \cap T} \est_{\ell_k}(z)^2
 + \sum_{T \in \TT_{\ell_k}^{\rm bad}} \sum_{z \in \NN_{\ell_k}^+ \cap T} \est_{\ell_k}(z)^2
 \xrightarrow{k \to \infty} 0.
\end{split}
\end{equation*}
This concludes the proof.
\end{proof}

\section{Proof of Theorem~\ref{thm:linear_convergence} (linear convergence)} \label{section:linear_convergence}

In this section, we prove that in 2D the saturation assumption yields contraction of the energy error
at each iteration of Algorithms~\ref{algorithm}.\ref{marking:A} and~\ref{algorithm}.\ref{marking:B}.
In the proof, we adapt the arguments of~\cite{doerfler,mns00}.
In particular, the following result holds for iterations where the spatial refinement is performed.

\begin{lemma}\label{lemma:satass:conv:spatial}
Let $\ell \in \N_0$.
Suppose that the saturation assumption~\eqref{eq:saturation}
holds for two Galerkin solutions $u_\ell$ and $\widehat u_\ell$ satisfying
\eqref{eq:discrete_formulation} and \eqref{eq:discrete_formulation:hat}, respectively.
Suppose that
\begin{equation}\label{eq1:lemma:satass:conv:spatial}
 \est_\ell(\QQQ_\ell) \le \Cmark \, \est_\ell(\NN_\ell^+)\quad \hbox{with\; $\Cmark > 0$}
\end{equation}
and let $\MM_\ell \subseteq \NN_\ell^+ \cap \NN_{\ell+1}$ be such that
\begin{equation}\label{eq2:lemma:satass:conv:spatial}
 \theta \, \est_\ell(\NN_\ell^+) \le \est_\ell(\MM_\ell)\quad \hbox{with\; $0 < \theta \le 1$}.
\end{equation}
Then, for the enhanced Galerkin solution $u_{\ell+1} \in \X_{\ell+1} \otimes \P_{\ell}$, there holds
\begin{equation*}
 \enorm{u - u_{\ell+1}}{}^2
 \le (1-q) \, \enorm{u - u_{\ell}}{}^2,
\end{equation*}
where $0 < q < 1$  depends only on $a_0$, $\Cmark$, $\qsat$, $\TT_0$, $\tau$, and $\theta$.
\end{lemma}

\begin{proof}
Using reliability~\eqref{eq:reliability}, the refinement criterion~\eqref{eq1:lemma:satass:conv:spatial},
and the marking criterion~\eqref{eq2:lemma:satass:conv:spatial}, we obtain
\begin{equation*}
 \frac{1-\qsat^2}{\Lambda\Cthm} \, \enorm{u - u_\ell}{}^2
 \stackrel{\eqref{eq:reliability}}{\le} 
 \est_\ell(\NN_\ell^+)^2 + \est_\ell(\QQQ_\ell)^2
 \stackrel{\eqref{eq1:lemma:satass:conv:spatial}}{\le} (1 + \Cmark^2) \, \est_\ell(\NN_\ell^+)^2
 \stackrel{\eqref{eq2:lemma:satass:conv:spatial}}{\le} (1 + \Cmark^2) \theta^{-2} \, \est_\ell(\MM_\ell)^2. 
\end{equation*}
Hence, using Corollary~\ref{cor:estimator} and the fact that
$\MM_\ell \subseteq \NN_\ell^+ \cap \NN_{\ell+1}$, we derive that
\begin{equation*}
\begin{split}
\enorm{u - u_{\ell+1}}{}^2 
& \stackrel{\phantom{\eqref{eq1:cor:estimator}}}{=} \enorm{u - u_\ell}{}^2 - \enorm{u_{\ell+1} - u_\ell}{}^2 \\
& \stackrel{\eqref{eq1:cor:estimator}}{\le} \enorm{u - u_\ell}{}^2 
 - \frac{\lambda}{K} \, \est_\ell (\NN_\ell^+ \cap \NN_{\ell+1} )^2 \\
& \stackrel{\phantom{\eqref{eq1:cor:estimator}}}{\leq} \enorm{u - u_\ell}{}^2 - \frac{\lambda}{K} \, \est_\ell(\MM_\ell)^2
 \le \Bigg(1 - \frac{\lambda}{\Lambda} \cdot \frac{\theta^2 (1-\qsat^2)}{\Cthm (1+\Cmark^2) K} \Bigg) \, \enorm{u - u_\ell}{}^2.
\end{split}
\end{equation*}
This concludes the proof.
\end{proof}

The next lemma concerns iterations where parametric enrichment is performed.
The proof is similar to that of Lemma~\ref{lemma:satass:conv:spatial}.

\begin{lemma}\label{lemma:satass:conv:parametric}
Let $\ell \in \N_0$.
Suppose that the saturation assumption~\eqref{eq:saturation}
holds for two Galerkin solutions $u_\ell$ and $\widehat u_\ell$ satisfying
\eqref{eq:discrete_formulation} and \eqref{eq:discrete_formulation:hat}, respectively.
Suppose that 
\begin{equation*}
\est_\ell(\NN_\ell^+) \le \Cmark \, \est_\ell(\QQQ_\ell)\quad \hbox{with\; $\Cmark > 0$}
\end{equation*}
and let $\MMM_\ell \subseteq \QQQ_\ell \cap \PPP_{\ell+1}$ be such that
\begin{equation*}
\theta \, \est_\ell(\QQQ_\ell) \le \est_\ell(\MMM_\ell)\quad \hbox{with\; $0 < \theta \le 1$}.
\end{equation*}
Then, for the enhanced Galerkin solution $u_{\ell+1} \in \X_{\ell} \otimes \P_{\ell+1}$, there holds
\begin{equation*}
 \enorm{u - u_{\ell+1}}{}^2
 \le (1-q) \, \enorm{u - u_{\ell}}{}^2,
\end{equation*}
where $0 < q < 1$ depends only on $a_0$, $\Cmark$, $\qsat$, $\TT_0$, $\tau$, and $\theta$.
\end{lemma}

With these results, we can prove Theorem~\ref{thm:linear_convergence}.

\begin{proof}[Proof of Theorem~\ref{thm:linear_convergence}]
We divide the proof into two steps.

{\bf Step~1.} Consider Algorithm~\ref{algorithm}.\ref{marking:A}.
In case~(a) of Marking criterion~\ref{marking:A}, we apply Lemma~\ref{lemma:satass:conv:spatial}
with $\Cmark = \vartheta^{-1}$ and $\theta = \thetaX$,
whereas in case (b) of this marking criterion, we use Lemma~\ref{lemma:satass:conv:parametric}
with $\Cmark = \vartheta$ and $\theta = \thetaP$.
In both cases, this proves contraction of the energy error
$\enorm{u - u_{\ell+1}}{} \le q_{\rm lin} \, \enorm{u - u_\ell}{}$ with $q_{\rm lin} \in (0,1)$.

{\bf Step~2.} Consider now Algorithm~\ref{algorithm}.\ref{marking:B}.
In case~(a) of Marking criterion~\ref{marking:B} one has
\begin{equation*}
 \thetaP \, \est_\ell(\QQQ_\ell) \le \est_\ell(\widetilde\MMM_\ell) 
 \le \vartheta^{-1} \est_\ell(\widetilde\RR_\ell)
 \le \vartheta^{-1} \est_\ell(\NN_\ell^+).
\end{equation*}
Hence, Lemma~\ref{lemma:satass:conv:spatial} applies to this case
with $\Cmark = \thetaP^{-1}\vartheta^{-1}$ and $\theta = \thetaX$.
Similarly, in case~(b) of Marking criterion~\ref{marking:B}, one has
\begin{equation*}
 \thetaX \, \est_\ell(\NN_\ell^+) \le \est_\ell(\widetilde\MM_\ell) 
 \le \est_\ell(\widetilde\RR_\ell) 
 < \vartheta \, \est_\ell(\widetilde\MMM_\ell)
 \le \vartheta \, \est_\ell(\QQQ_\ell).
\end{equation*}
Hence, in this case, Lemma~\ref{lemma:satass:conv:parametric} applies with $\Cmark = \thetaX^{-1}\vartheta$ and $\theta = \thetaP$.
Thus, in both cases, we obtain contraction of the energy error
$\enorm{u - u_{\ell+1}}{} \le q_{\rm lin} \, \enorm{u - u_\ell}{}$ with $q_{\rm lin} \in (0,1)$.
\end{proof}

\appendix
\input{appendix.tex}

\bibliographystyle{alpha}
\bibliography{literature}

\end{document}

%% file: appendix.tex
\section{Numerical results (extended version)} \label{sec:appendix}

In Tables~\ref{tab:fulldata:A}--\ref{tab:fulldata:D}, we collect the computational costs~\eqref{eq:cost} 
and empirical convergence rates for
Algorithms~\ref{algorithm}.\ref{marking:A}--\ref{algorithm}.\ref{marking:D} applied to
the parametric model problem from Section~\ref{sec:numer:results}.
The empirical convergence rates are computed as the slopes of 
the lines which are best fit, in the least squares sense, of the overall error estimates $\est_\ell$ 
computed by the algorithm with the corresponding pair of marking parameters 
$(\thetaX,\thetaP) \,{\in}\, \Theta\,{\times}\,\Theta$
with $\Theta = \{ 0.1, 0.2, \dots, 0.9 \}$.
We observe that all the rates are similar and vary in a range between $-0.36$ and $-0.32$.
Furthermore, in each table, numbers in boldface indicate the smallest cost in the corresponding row (i.e., for fixed $\thetaX$), 
whereas the starred boldface number shows the overall smallest cost in the table.

\begin{table}[h!]
\begin{center} 
\setlength\tabcolsep{2.7pt} 
\smallfontthree{
\renewcommand{\arraystretch}{1.50} 
\begin{tabular}{c |*{8}{c} c } 
\noalign{\hrule height 1.0pt}
\multicolumn{10}{c}{ {\bf Algorithm~\ref{algorithm}.A} } \\
\noalign{\hrule height 1.0pt}
\backslashbox{$\thetaX$}{$\thetaP$} 
& 0.1 & 0.2 & 0.3 & 0.4 & 0.5 & 0.6 & 0.7 & 0.8 & 0.9 \\
\hline\\[-13pt]
0.1 
& \begin{tabular}{@{}c@{}}36,634,764 			\cellvsp $-$0.3395\end{tabular}
& \begin{tabular}{@{}c@{}}36,634,764 			\cellvsp $-$0.3395\end{tabular}
& \begin{tabular}{@{}c@{}}36,634,764 			\cellvsp $-$0.3395\end{tabular}
& \begin{tabular}{@{}c@{}}36,634,764 			\cellvsp $-$0.3395\end{tabular}
& \begin{tabular}{@{}c@{}}36,634,764 			\cellvsp $-$0.3395\end{tabular}
& \begin{tabular}{@{}c@{}}37,126,693 			\cellvsp $-$0.3412\end{tabular}
& \begin{tabular}{@{}c@{}}38,135,658 			\cellvsp $-$0.3401\end{tabular}
& \begin{tabular}{@{}c@{}}38,948,918 			\cellvsp $-$0.3398\end{tabular}
& \begin{tabular}{@{}c@{}}{\bf 36,522,593} 		\cellvsp $-$0.3401\end{tabular} \\  
\rowcolor{myLightGray}
0.2 
& \begin{tabular}{@{}c@{}}10,652,382 			\cellvsp $-$0.3386\end{tabular}
& \begin{tabular}{@{}c@{}}10,652,382 			\cellvsp $-$0.3386\end{tabular}
& \begin{tabular}{@{}c@{}}10,652,382 			\cellvsp $-$0.3386\end{tabular}
& \begin{tabular}{@{}c@{}}10,652,382 			\cellvsp $-$0.3386\end{tabular}
& \begin{tabular}{@{}c@{}}10,652,382 			\cellvsp $-$0.3386\end{tabular}
& \begin{tabular}{@{}c@{}}10,472,434 			\cellvsp $-$0.3401\end{tabular}
& \begin{tabular}{@{}c@{}}10,611,056 			\cellvsp $-$0.3392\end{tabular}
& \begin{tabular}{@{}c@{}}10,842,902 			\cellvsp $-$0.3395\end{tabular}
& \begin{tabular}{@{}c@{}}{\bf 9,891,950} 		\cellvsp $-$0.3398\end{tabular} \\  
0.3
& \begin{tabular}{@{}c@{}}5,737,346 			\cellvsp $-$0.3380\end{tabular}
& \begin{tabular}{@{}c@{}}5,737,346 			\cellvsp $-$0.3380\end{tabular}
& \begin{tabular}{@{}c@{}}5,737,346 			\cellvsp $-$0.3380\end{tabular}
& \begin{tabular}{@{}c@{}}5,737,346 			\cellvsp $-$0.3380\end{tabular}
& \begin{tabular}{@{}c@{}}5,737,346 			\cellvsp $-$0.3380\end{tabular}
& \begin{tabular}{@{}c@{}}5,398,269 			\cellvsp $-$0.3392\end{tabular}
& \begin{tabular}{@{}c@{}}5,444,071 			\cellvsp $-$0.3386\end{tabular}
& \begin{tabular}{@{}c@{}}5,491,501			 	\cellvsp $-$0.3390\end{tabular}
& \begin{tabular}{@{}c@{}}{\bf 4,487,527} 		\cellvsp $-$0.3392\end{tabular} \\  
\rowcolor{myLightGray}
0.4 
& \begin{tabular}{@{}c@{}}4,066,841 			\cellvsp $-$0.3369\end{tabular}
& \begin{tabular}{@{}c@{}}4,066,841 			\cellvsp $-$0.3369\end{tabular}
& \begin{tabular}{@{}c@{}}4,066,841 			\cellvsp $-$0.3369\end{tabular}
& \begin{tabular}{@{}c@{}}4,066,841 			\cellvsp $-$0.3369\end{tabular}
& \begin{tabular}{@{}c@{}}4,066,841 			\cellvsp $-$0.3369\end{tabular}
& \begin{tabular}{@{}c@{}}3,657,156 			\cellvsp $-$0.3382\end{tabular}
& \begin{tabular}{@{}c@{}}3,703,037 			\cellvsp $-$0.3379\end{tabular}
& \begin{tabular}{@{}c@{}}3,738,567 			\cellvsp $-$0.3384\end{tabular}
& \begin{tabular}{@{}c@{}}{\bf 3,005,547} 		\cellvsp $-$0.3385\end{tabular} \\  
0.5
& \begin{tabular}{@{}c@{}}2,974,895 			\cellvsp $-$0.3360\end{tabular}
& \begin{tabular}{@{}c@{}}2,974,895 			\cellvsp $-$0.3360\end{tabular}
& \begin{tabular}{@{}c@{}}2,974,895 			\cellvsp $-$0.3360\end{tabular}
& \begin{tabular}{@{}c@{}}2,974,895 			\cellvsp $-$0.3360\end{tabular}
& \begin{tabular}{@{}c@{}}2,974,895 			\cellvsp $-$0.3360\end{tabular}
& \begin{tabular}{@{}c@{}}2,523,497 			\cellvsp $-$0.3374\end{tabular}
& \begin{tabular}{@{}c@{}}2,526,413 			\cellvsp $-$0.3373\end{tabular}
& \begin{tabular}{@{}c@{}}2,521,302 			\cellvsp $-$0.3381\end{tabular}
& \begin{tabular}{@{}c@{}}{\bf 2,193,757}		\cellvsp $-$0.3380\end{tabular} \\  
\rowcolor{myLightGray}
0.6 
& \begin{tabular}{@{}c@{}}2,838,789 			\cellvsp $-$0.3371\end{tabular}
& \begin{tabular}{@{}c@{}}2,838,789 			\cellvsp $-$0.3371\end{tabular}
& \begin{tabular}{@{}c@{}}2,838,789 			\cellvsp $-$0.3371\end{tabular}
& \begin{tabular}{@{}c@{}}2,838,789 			\cellvsp $-$0.3371\end{tabular}
& \begin{tabular}{@{}c@{}}2,838,789 			\cellvsp $-$0.3371\end{tabular}
& \begin{tabular}{@{}c@{}}2,323,857 			\cellvsp $-$0.3383\end{tabular}
& \begin{tabular}{@{}c@{}}2,351,810 			\cellvsp $-$0.3384\end{tabular}
& \begin{tabular}{@{}c@{}}2,331,065 			\cellvsp $-$0.3389\end{tabular}
& \begin{tabular}{@{}c@{}}{\bf 1,900,951} 		\cellvsp $-$0.3385\end{tabular} \\  
0.7 
& \begin{tabular}{@{}c@{}}2,658,382			 	\cellvsp $-$0.3373\end{tabular}
& \begin{tabular}{@{}c@{}}2,658,382 			\cellvsp $-$0.3373\end{tabular}
& \begin{tabular}{@{}c@{}}2,658,382 			\cellvsp $-$0.3373\end{tabular}
& \begin{tabular}{@{}c@{}}2,658,382				\cellvsp $-$0.3373\end{tabular}
& \begin{tabular}{@{}c@{}}2,658,382 			\cellvsp $-$0.3373\end{tabular}
& \begin{tabular}{@{}c@{}}2,094,382 			\cellvsp $-$0.3380\end{tabular}
& \begin{tabular}{@{}c@{}}2,046,871				\cellvsp $-$0.3390\end{tabular}
& \begin{tabular}{@{}c@{}}2,014,430				\cellvsp $-$0.3399\end{tabular}
& \begin{tabular}{@{}c@{}}{\bf 1,566,530} 		\cellvsp $-$0.3394\end{tabular} \\  
\rowcolor{myLightGray}
0.8 
& \begin{tabular}{@{}c@{}}2,454,929 			\cellvsp $-$0.3346\end{tabular}
& \begin{tabular}{@{}c@{}}2,454,929 			\cellvsp $-$0.3346\end{tabular}
& \begin{tabular}{@{}c@{}}2,454,929 			\cellvsp $-$0.3346\end{tabular}
& \begin{tabular}{@{}c@{}}2,454,929 			\cellvsp $-$0.3346\end{tabular}
& \begin{tabular}{@{}c@{}}2,454,929 			\cellvsp $-$0.3346\end{tabular}
& \begin{tabular}{@{}c@{}}2,403,912 			\cellvsp $-$0.3354\end{tabular}
& \begin{tabular}{@{}c@{}}1,628,563 			\cellvsp $-$0.3367\end{tabular}
& \begin{tabular}{@{}c@{}}$\mathbf{1,560,286^\star}$ \cellvsp $-$0.3363\end{tabular}
& \begin{tabular}{@{}c@{}}1,731,044 			\cellvsp $-$0.3373\end{tabular} \\  
0.9 
& \begin{tabular}{@{}c@{}}3,042,687 			\cellvsp $-$0.3278\end{tabular}
& \begin{tabular}{@{}c@{}}3,042,687 			\cellvsp $-$0.3278\end{tabular}
& \begin{tabular}{@{}c@{}}3,042,687 			\cellvsp $-$0.3278\end{tabular}
& \begin{tabular}{@{}c@{}}3,042,687 			\cellvsp $-$0.3278\end{tabular}
& \begin{tabular}{@{}c@{}}3,042,687 			\cellvsp $-$0.3278\end{tabular}
& \begin{tabular}{@{}c@{}}2,891,115 			\cellvsp $-$0.3308\end{tabular}
& \begin{tabular}{@{}c@{}}1,978,191 			\cellvsp $-$0.3289\end{tabular}
& \begin{tabular}{@{}c@{}}{\bf 1,776,192} 		\cellvsp $-$0.3321\end{tabular}
& \begin{tabular}{@{}c@{}}1,993,972 			\cellvsp $-$0.3334\end{tabular} \\  
\noalign{\hrule height 1.0pt}
\end{tabular}
\vspace{8pt}
\caption{
Computational cost (top of the cell) and empirical convergence rates (bottom of the cell)
for Algorithm~\ref{algorithm}.\ref{marking:A} applied to the parametric model problem
in Section~\ref{sec:numer:results}.
}
\label{tab:fulldata:A}
}
\end{center} 
\end{table}

\begin{table}[t!]
\begin{center} 
\setlength\tabcolsep{1.8pt} 
\smallfontthree{
\renewcommand{\arraystretch}{1.50}
\begin{tabular}{c |*{8}{c} c } 
\noalign{\hrule height 1.0pt}
\multicolumn{10}{c}{ {\bf Algorithm~\ref{algorithm}.B} } \\
\noalign{\hrule height 1.0pt}
\backslashbox{$\thetaX$}{$\thetaP$} 
& 0.1 & 0.2 & 0.3 & 0.4 & 0.5 & 0.6 & 0.7 & 0.8 & 0.9 \\
\hline\\[-13pt]
0.1 
& \begin{tabular}{@{}c@{}}{\bf 55,591,871}		\cellvsp $-$0.3543\end{tabular}
& \begin{tabular}{@{}c@{}}{\bf 55,591,871} 		\cellvsp $-$0.3543\end{tabular}
& \begin{tabular}{@{}c@{}}{\bf 55,591,871} 		\cellvsp $-$0.3543\end{tabular}
& \begin{tabular}{@{}c@{}}61,960,516 			\cellvsp $-$0.3478\end{tabular}
& \begin{tabular}{@{}c@{}}71,558,237 			\cellvsp $-$0.3372\end{tabular}
& \begin{tabular}{@{}c@{}}82,193,108 			\cellvsp $-$0.3326\end{tabular}
& \begin{tabular}{@{}c@{}}90,977,432 			\cellvsp $-$0.3288\end{tabular}
& \begin{tabular}{@{}c@{}}99,219,735			\cellvsp $-$0.3281\end{tabular}
& \begin{tabular}{@{}c@{}}$>1$e+$08$   			\cellvsp $-$0.3216\end{tabular} \\
\rowcolor{myLightGray}
0.2 
& \begin{tabular}{@{}c@{}}11,801,518 			\cellvsp $-$0.3543\end{tabular}
& \begin{tabular}{@{}c@{}}11,801,518 			\cellvsp $-$0.3543\end{tabular}
& \begin{tabular}{@{}c@{}}11,801,518 			\cellvsp $-$0.3543\end{tabular}
& \begin{tabular}{@{}c@{}}{\bf 11,606,375}		\cellvsp $-$0.3542\end{tabular}
& \begin{tabular}{@{}c@{}}12,864,430 			\cellvsp $-$0.3479\end{tabular}
& \begin{tabular}{@{}c@{}}13,991,407			\cellvsp $-$0.3424\end{tabular}
& \begin{tabular}{@{}c@{}}14,616,770 			\cellvsp $-$0.3377\end{tabular}
& \begin{tabular}{@{}c@{}}15,930,094 			\cellvsp $-$0.3376\end{tabular}
& \begin{tabular}{@{}c@{}}18,204,663 			\cellvsp $-$0.3362\end{tabular} \\
0.3
& \begin{tabular}{@{}c@{}}5,385,296 			\cellvsp $-$0.3499\end{tabular}
& \begin{tabular}{@{}c@{}}5,385,296 			\cellvsp $-$0.3499\end{tabular}
& \begin{tabular}{@{}c@{}}5,385,296				\cellvsp $-$0.3499\end{tabular}
& \begin{tabular}{@{}c@{}}5,385,296		 		\cellvsp $-$0.3499\end{tabular}
& \begin{tabular}{@{}c@{}}{\bf 5,340,256}	 	\cellvsp $-$0.3492\end{tabular}
& \begin{tabular}{@{}c@{}}5,757,081 			\cellvsp $-$0.3454\end{tabular}
& \begin{tabular}{@{}c@{}}6,042,307 			\cellvsp $-$0.3416\end{tabular}
& \begin{tabular}{@{}c@{}}5,796,230 			\cellvsp $-$0.3452\end{tabular}
& \begin{tabular}{@{}c@{}}6,330,829 			\cellvsp $-$0.3391\end{tabular} \\
\rowcolor{myLightGray}
0.4 
& \begin{tabular}{@{}c@{}}3,587,223 			\cellvsp $-$0.3457\end{tabular}
& \begin{tabular}{@{}c@{}}3,587,223 			\cellvsp $-$0.3457\end{tabular}
& \begin{tabular}{@{}c@{}}3,587,223 			\cellvsp $-$0.3457\end{tabular}
& \begin{tabular}{@{}c@{}}3,587,223 			\cellvsp $-$0.3457\end{tabular}
& \begin{tabular}{@{}c@{}}3,626,569 			\cellvsp $-$0.3461\end{tabular}
& \begin{tabular}{@{}c@{}}3,432,938 			\cellvsp $-$0.3467\end{tabular}
& \begin{tabular}{@{}c@{}}3,338,087 			\cellvsp $-$0.3442\end{tabular}
& \begin{tabular}{@{}c@{}}{\bf 3,086,323}	 	\cellvsp $-$0.3473\end{tabular}
& \begin{tabular}{@{}c@{}}3,165,582 			\cellvsp $-$0.3425\end{tabular} \\
0.5 
& \begin{tabular}{@{}c@{}}2,874,852 			\cellvsp $-$0.3429\end{tabular}
& \begin{tabular}{@{}c@{}}2,874,852 			\cellvsp $-$0.3429\end{tabular}
& \begin{tabular}{@{}c@{}}2,874,852 			\cellvsp $-$0.3429\end{tabular}
& \begin{tabular}{@{}c@{}}2,874,852 			\cellvsp $-$0.3429\end{tabular}
& \begin{tabular}{@{}c@{}}2,874,852 			\cellvsp $-$0.3429\end{tabular}
& \begin{tabular}{@{}c@{}}2,380,185 			\cellvsp $-$0.3464\end{tabular}
& \begin{tabular}{@{}c@{}}2,560,036 			\cellvsp $-$0.3451\end{tabular}
& \begin{tabular}{@{}c@{}}{\bf 2,081,426} 		\cellvsp $-$0.3465\end{tabular}
& \begin{tabular}{@{}c@{}}2,582,765 			\cellvsp $-$0.3430\end{tabular} \\
\rowcolor{myLightGray}
0.6 
& \begin{tabular}{@{}c@{}}2,883,427 			\cellvsp $-$0.3383\end{tabular}
& \begin{tabular}{@{}c@{}}2,883,427 			\cellvsp $-$0.3383\end{tabular}
& \begin{tabular}{@{}c@{}}2,883,427 			\cellvsp $-$0.3383\end{tabular}
& \begin{tabular}{@{}c@{}}2,883,427 			\cellvsp $-$0.3383\end{tabular}
& \begin{tabular}{@{}c@{}}2,883,427 			\cellvsp $-$0.3383\end{tabular}
& \begin{tabular}{@{}c@{}}2,259,538 			\cellvsp $-$0.3411\end{tabular}
& \begin{tabular}{@{}c@{}}2,307,901 			\cellvsp $-$0.3421\end{tabular}
& \begin{tabular}{@{}c@{}}{\bf 1,764,686}	 	\cellvsp $-$0.3422\end{tabular}
& \begin{tabular}{@{}c@{}}2,078,219 			\cellvsp $-$0.3415\end{tabular} \\
0.7 
& \begin{tabular}{@{}c@{}}3,157,697 			\cellvsp $-$0.3292\end{tabular}
& \begin{tabular}{@{}c@{}}3,157,697 			\cellvsp $-$0.3292\end{tabular}
& \begin{tabular}{@{}c@{}}3,157,697 			\cellvsp $-$0.3292\end{tabular}
& \begin{tabular}{@{}c@{}}3,157,697 			\cellvsp $-$0.3292\end{tabular}
& \begin{tabular}{@{}c@{}}3,157,697 			\cellvsp $-$0.3292\end{tabular}
& \begin{tabular}{@{}c@{}}2,146,095 			\cellvsp $-$0.3350\end{tabular}
& \begin{tabular}{@{}c@{}}1,973,460 			\cellvsp $-$0.3383\end{tabular}
& \begin{tabular}{@{}c@{}}1,966,801 			\cellvsp $-$0.3389\end{tabular}
& \begin{tabular}{@{}c@{}}$\mathbf{1,488,993^\star}$ \cellvsp $-$0.3398\end{tabular} \\
\rowcolor{myLightGray}
0.8 
& \begin{tabular}{@{}c@{}}3,381,315 			\cellvsp $-$0.3237\end{tabular}
& \begin{tabular}{@{}c@{}}3,381,315 			\cellvsp $-$0.3237\end{tabular}
& \begin{tabular}{@{}c@{}}3,381,315 			\cellvsp $-$0.3237\end{tabular}
& \begin{tabular}{@{}c@{}}3,381,315 			\cellvsp $-$0.3237\end{tabular}
& \begin{tabular}{@{}c@{}}3,381,315		 		\cellvsp $-$0.3237\end{tabular}
& \begin{tabular}{@{}c@{}}2,613,691 			\cellvsp $-$0.3320\end{tabular}
& \begin{tabular}{@{}c@{}}1,641,372 			\cellvsp $-$0.3355\end{tabular}
& \begin{tabular}{@{}c@{}}{\bf 1,549,138} 		\cellvsp $-$0.3369\end{tabular}
& \begin{tabular}{@{}c@{}}1,720,006 			\cellvsp $-$0.3378\end{tabular} \\
0.9 
& \begin{tabular}{@{}c@{}}4,886,790 			\cellvsp $-$0.3153\end{tabular}
& \begin{tabular}{@{}c@{}}4,886,790 			\cellvsp $-$0.3153\end{tabular}
& \begin{tabular}{@{}c@{}}4,886,790 			\cellvsp $-$0.3153\end{tabular}
& \begin{tabular}{@{}c@{}}4,886,790 			\cellvsp $-$0.3153\end{tabular}
& \begin{tabular}{@{}c@{}}4,886,790 			\cellvsp $-$0.3153\end{tabular}
& \begin{tabular}{@{}c@{}}3,708,374 			\cellvsp $-$0.3205\end{tabular}
& \begin{tabular}{@{}c@{}}2,288,775 			\cellvsp $-$0.3246\end{tabular}
& \begin{tabular}{@{}c@{}}2,071,551 			\cellvsp $-$0.3282\end{tabular}
& \begin{tabular}{@{}c@{}}{\bf 1,993,972} 		\cellvsp $-$0.3334\end{tabular} \\
\noalign{\hrule height 1.0pt}
\end{tabular}
\vspace{8pt}
\caption{
Computational cost (top of the cell) and empirical convergence rates (bottom of the cell)
for Algorithm~\ref{algorithm}.\ref{marking:B} applied to the parametric model problem
in Section~\ref{sec:numer:results}.
}
\label{tab:fulldata:B}
}
\end{center} 
\end{table}

\begin{table}[t!]
\begin{center} 
\setlength\tabcolsep{2.8pt} 
\smallfontthree{
\renewcommand{\arraystretch}{1.50} 
\begin{tabular}{c |*{8}{c} c } 
\noalign{\hrule height 1.0pt}
\multicolumn{10}{c}{ {\bf Algorithm~\ref{algorithm}.C} } \\
\noalign{\hrule height 1.0pt}
\backslashbox{$\thetaX$}{$\thetaP$} 
& 0.1 & 0.2 & 0.3 & 0.4 & 0.5 & 0.6 & 0.7 & 0.8 & 0.9 \\
\hline\\[-13pt]
0.1 
& \begin{tabular}{@{}c@{}}37,126,693 			\cellvsp $-$0.3412\end{tabular}
& \begin{tabular}{@{}c@{}}38,436,652 			\cellvsp $-$0.3392\end{tabular}
& \begin{tabular}{@{}c@{}}{\bf 31,766,942}		\cellvsp $-$0.3391\end{tabular}
& \begin{tabular}{@{}c@{}}38,948,918 			\cellvsp $-$0.3398\end{tabular}
& \begin{tabular}{@{}c@{}}40,891,821 			\cellvsp $-$0.3400\end{tabular}
& \begin{tabular}{@{}c@{}}35,855,809 			\cellvsp $-$0.3397\end{tabular}
& \begin{tabular}{@{}c@{}}30,252,882 			\cellvsp $-$0.3397\end{tabular}
& \begin{tabular}{@{}c@{}}44,306,077 			\cellvsp $-$0.3389\end{tabular}
& \begin{tabular}{@{}c@{}}47,582,801 			\cellvsp $-$0.3342\end{tabular} \\
\rowcolor{myLightGray}
0.2 
& \begin{tabular}{@{}c@{}}10,472,434 			\cellvsp $-$0.3401\end{tabular}
& \begin{tabular}{@{}c@{}}10,293,846 			\cellvsp $-$0.3388\end{tabular}
& \begin{tabular}{@{}c@{}}8,743,434 			\cellvsp $-$0.3388\end{tabular}
& \begin{tabular}{@{}c@{}}10,842,902 			\cellvsp $-$0.3395\end{tabular}
& \begin{tabular}{@{}c@{}}10,790,957 			\cellvsp $-$0.3397\end{tabular}
& \begin{tabular}{@{}c@{}}9,833,369 			\cellvsp $-$0.3395\end{tabular}
& \begin{tabular}{@{}c@{}}{\bf 8,317,634} 		\cellvsp $-$0.3395\end{tabular}
& \begin{tabular}{@{}c@{}}12,082,564			\cellvsp $-$0.3389\end{tabular}
& \begin{tabular}{@{}c@{}}12,942,155 			\cellvsp $-$0.3338\end{tabular}\\  
0.3
& \begin{tabular}{@{}c@{}}5,398,269 			\cellvsp $-$0.3392\end{tabular}
& \begin{tabular}{@{}c@{}}5,386,660 			\cellvsp $-$0.3382\end{tabular}
& \begin{tabular}{@{}c@{}}4,609,593 			\cellvsp $-$0.3384\end{tabular}
& \begin{tabular}{@{}c@{}}5,491,501 			\cellvsp $-$0.3390\end{tabular}
& \begin{tabular}{@{}c@{}}5,194,711 			\cellvsp $-$0.3390\end{tabular}
& \begin{tabular}{@{}c@{}}4,573,863 			\cellvsp $-$0.3391\end{tabular}
& \begin{tabular}{@{}c@{}}{\bf 4,270,672} 		\cellvsp $-$0.3390\end{tabular}
& \begin{tabular}{@{}c@{}}6,113,283 			\cellvsp $-$0.3382\end{tabular}
& \begin{tabular}{@{}c@{}}5,957,047 			\cellvsp $-$0.3334\end{tabular}\\  
\rowcolor{myLightGray}
0.4 
& \begin{tabular}{@{}c@{}}3,657,156 			\cellvsp $-$0.3382\end{tabular}
& \begin{tabular}{@{}c@{}}3,573,880 			\cellvsp $-$0.3372\end{tabular}
& \begin{tabular}{@{}c@{}}3,169,527 			\cellvsp $-$0.3373\end{tabular}
& \begin{tabular}{@{}c@{}}3,738,567 			\cellvsp $-$0.3384\end{tabular}
& \begin{tabular}{@{}c@{}}3,352,712 			\cellvsp $-$0.3384\end{tabular}
& \begin{tabular}{@{}c@{}}3,024,178 			\cellvsp $-$0.3380\end{tabular}
& \begin{tabular}{@{}c@{}}{\bf 2,634,872} 		\cellvsp $-$0.3378\end{tabular}
& \begin{tabular}{@{}c@{}}4,146,897 			\cellvsp $-$0.3370\end{tabular}
& \begin{tabular}{@{}c@{}}3,567,033 			\cellvsp $-$0.3325\end{tabular}\\  
0.5 
& \begin{tabular}{@{}c@{}}2,523,497 			\cellvsp $-$0.3374\end{tabular}
& \begin{tabular}{@{}c@{}}2,380,561 			\cellvsp $-$0.3367\end{tabular}
& \begin{tabular}{@{}c@{}}2,459,900 			\cellvsp $-$0.3368\end{tabular}
& \begin{tabular}{@{}c@{}}2,521,302 			\cellvsp $-$0.3381\end{tabular}
& \begin{tabular}{@{}c@{}}2,102,539 			\cellvsp $-$0.3379\end{tabular}
& \begin{tabular}{@{}c@{}}2,260,210	 			\cellvsp $-$0.3374\end{tabular}
& \begin{tabular}{@{}c@{}}{\bf 1,847,454} 		\cellvsp $-$0.3371\end{tabular}
& \begin{tabular}{@{}c@{}}2,721,150 			\cellvsp $-$0.3369\end{tabular}
& \begin{tabular}{@{}c@{}}2,572,804 			\cellvsp $-$0.3320\end{tabular}\\  
\rowcolor{myLightGray}
0.6 
& \begin{tabular}{@{}c@{}}2,323,857 			\cellvsp $-$0.3384\end{tabular}
& \begin{tabular}{@{}c@{}}2,168,310 			\cellvsp $-$0.3377\end{tabular}
& \begin{tabular}{@{}c@{}}2,271,816 			\cellvsp $-$0.3382\end{tabular}
& \begin{tabular}{@{}c@{}}2,331,065 			\cellvsp $-$0.3389\end{tabular}
& \begin{tabular}{@{}c@{}}1,828,574 			\cellvsp $-$0.3384\end{tabular}
& \begin{tabular}{@{}c@{}}2,004,779 			\cellvsp $-$0.3390\end{tabular}
& \begin{tabular}{@{}c@{}}{\bf 1,533,861} 		\cellvsp $-$0.3386\end{tabular}
& \begin{tabular}{@{}c@{}}2,528,526 			\cellvsp $-$0.3378\end{tabular}
& \begin{tabular}{@{}c@{}}2,294,306 			\cellvsp $-$0.3325\end{tabular}\\  
0.7 
& \begin{tabular}{@{}c@{}}2,094,382 			\cellvsp $-$0.3380\end{tabular}
& \begin{tabular}{@{}c@{}}1,891,752 			\cellvsp $-$0.3375\end{tabular}
& \begin{tabular}{@{}c@{}}1,970,087 			\cellvsp $-$0.3376\end{tabular}
& \begin{tabular}{@{}c@{}}2,014,430 			\cellvsp $-$0.3399\end{tabular}
& \begin{tabular}{@{}c@{}}$\mathbf{1,496,851^\star}$ \cellvsp $-$0.3393\end{tabular}
& \begin{tabular}{@{}c@{}}1,710,029 			\cellvsp $-$0.3383\end{tabular}
& \begin{tabular}{@{}c@{}}1,793,937 			\cellvsp $-$0.3383\end{tabular}
& \begin{tabular}{@{}c@{}}2,185,402 			\cellvsp $-$0.3377\end{tabular}
& \begin{tabular}{@{}c@{}}1,837,025 			\cellvsp $-$0.3325\end{tabular}\\  
\rowcolor{myLightGray}
0.8 
& \begin{tabular}{@{}c@{}}2,403,912				\cellvsp $-$0.3354\end{tabular}
& \begin{tabular}{@{}c@{}}2,162,469 			\cellvsp $-$0.3347\end{tabular}
& \begin{tabular}{@{}c@{}}2,240,383 			\cellvsp $-$0.3362\end{tabular}
& \begin{tabular}{@{}c@{}}{\bf 1,560,286} 		\cellvsp $-$0.3363\end{tabular}
& \begin{tabular}{@{}c@{}}1,645,652 			\cellvsp $-$0.3370\end{tabular}
& \begin{tabular}{@{}c@{}}1,940,368 			\cellvsp $-$0.3368\end{tabular}
& \begin{tabular}{@{}c@{}}2,048,616 			\cellvsp $-$0.3370\end{tabular}
& \begin{tabular}{@{}c@{}}1,620,466 			\cellvsp $-$0.3353\end{tabular}
& \begin{tabular}{@{}c@{}}2,089,746				\cellvsp $-$0.3297\end{tabular}\\  
0.9 
& \begin{tabular}{@{}c@{}}2,891,115 			\cellvsp $-$0.3308\end{tabular}
& \begin{tabular}{@{}c@{}}2,621,348 			\cellvsp $-$0.3295\end{tabular}
& \begin{tabular}{@{}c@{}}2,830,679 			\cellvsp $-$0.3283\end{tabular}
& \begin{tabular}{@{}c@{}}{\bf 1,776,192} 		\cellvsp $-$0.3321\end{tabular}
& \begin{tabular}{@{}c@{}}1,885,067 			\cellvsp $-$0.3329\end{tabular}
& \begin{tabular}{@{}c@{}}2,470,591 			\cellvsp $-$0.3282\end{tabular}
& \begin{tabular}{@{}c@{}}2,619,845 			\cellvsp $-$0.3285\end{tabular}
& \begin{tabular}{@{}c@{}}1,880,106 			\cellvsp $-$0.3295\end{tabular}
& \begin{tabular}{@{}c@{}}2,611,297 			\cellvsp $-$0.3231\end{tabular}\\
\noalign{\hrule height 1.0pt}
\end{tabular}
\vspace{8pt}
\caption{
Computational cost (top of the cell) and empirical convergence rates (bottom of the cell)
for Algorithm~\ref{algorithm}.\ref{marking:C} applied to the parametric model problem
in Section~\ref{sec:numer:results}.
}
\label{tab:fulldata:C}
}
\end{center} 
\end{table}

\begin{table}[t!]
\begin{center} 
\setlength\tabcolsep{2.5pt} 
\smallfontthree{
\renewcommand{\arraystretch}{1.50}
\begin{tabular}{c |*{8}{c} c } 
\noalign{\hrule height 1.0pt}
\multicolumn{10}{c}{ {\bf Algorithm~\ref{algorithm}.D} } \\
\noalign{\hrule height 1.0pt}
\backslashbox{$\thetaX$}{$\thetaP$} 
& 0.1 & 0.2 & 0.3 & 0.4 & 0.5 & 0.6 & 0.7 & 0.8 & 0.9 \\
\hline\\[-13pt]
0.1 
& \begin{tabular}{@{}c@{}}{\bf 65,375,862} \cellvsp $-$0.3482\end{tabular}
& \begin{tabular}{@{}c@{}}78,893,946 \cellvsp $-$0.3367\end{tabular}
& \begin{tabular}{@{}c@{}}82,752,064 \cellvsp $-$0.3375\end{tabular}
& \begin{tabular}{@{}c@{}}86,536,330 \cellvsp $-$0.3383\end{tabular}
& \begin{tabular}{@{}c@{}}85,276,880 \cellvsp $-$0.3389\end{tabular}
& \begin{tabular}{@{}c@{}}94,093,402 \cellvsp $-$0.3378\end{tabular}
& \begin{tabular}{@{}c@{}}$>1$e+$08$ \cellvsp $-$0.3339\end{tabular} 
& \begin{tabular}{@{}c@{}}$>1$e+$08$ \cellvsp $-$0.3285\end{tabular} 
& \begin{tabular}{@{}c@{}}$>1$e+$08$ \cellvsp $-$0.3235\end{tabular}\\ 
\rowcolor{myLightGray}
0.2 
& \begin{tabular}{@{}c@{}}{\bf 11,852,403} \cellvsp $-$0.3516\end{tabular}
& \begin{tabular}{@{}c@{}}12,956,995 \cellvsp $-$0.3447\end{tabular}
& \begin{tabular}{@{}c@{}}13,658,830 \cellvsp $-$0.3427\end{tabular}
& \begin{tabular}{@{}c@{}}14,138,022 \cellvsp $-$0.3475\end{tabular}
& \begin{tabular}{@{}c@{}}14,888,877 \cellvsp $-$0.3441\end{tabular}
& \begin{tabular}{@{}c@{}}16,118,678 \cellvsp $-$0.3443\end{tabular}
& \begin{tabular}{@{}c@{}}17,014,103 \cellvsp $-$0.3412\end{tabular}
& \begin{tabular}{@{}c@{}}17,774,020 \cellvsp $-$0.3370\end{tabular}
& \begin{tabular}{@{}c@{}}22,821,433 \cellvsp $-$0.3294\end{tabular} \\  
0.3
& \begin{tabular}{@{}c@{}}5,393,465 \cellvsp $-$0.3482\end{tabular}
& \begin{tabular}{@{}c@{}}{\bf 5,187,233} \cellvsp $-$0.3503\end{tabular}
& \begin{tabular}{@{}c@{}}5,976,264 \cellvsp $-$0.3465\end{tabular}
& \begin{tabular}{@{}c@{}}5,607,496 \cellvsp $-$0.3462\end{tabular}
& \begin{tabular}{@{}c@{}}6,018,687 \cellvsp $-$0.3443\end{tabular}
& \begin{tabular}{@{}c@{}}6,200,892 \cellvsp $-$0.3415\end{tabular}
& \begin{tabular}{@{}c@{}}6,737,413 \cellvsp $-$0.3383\end{tabular}
& \begin{tabular}{@{}c@{}}6,628,919 \cellvsp $-$0.3354\end{tabular}
& \begin{tabular}{@{}c@{}}8,667,208 \cellvsp $-$0.3279\end{tabular} \\  
\rowcolor{myLightGray}
0.4 
& \begin{tabular}{@{}c@{}}3,359,537 \cellvsp $-$0.3466\end{tabular}
& \begin{tabular}{@{}c@{}}2,993,784 \cellvsp $-$0.3499\end{tabular}
& \begin{tabular}{@{}c@{}}{\bf 2,968,892} \cellvsp $-$0.3484\end{tabular}
& \begin{tabular}{@{}c@{}}3,086,323 \cellvsp $-$0.3473\end{tabular}
& \begin{tabular}{@{}c@{}}3,280,115 \cellvsp $-$0.3460\end{tabular}
& \begin{tabular}{@{}c@{}}3,526,098 \cellvsp $-$0.3446\end{tabular}
& \begin{tabular}{@{}c@{}}3,229,531 \cellvsp $-$0.3419\end{tabular}
& \begin{tabular}{@{}c@{}}3,942,973 \cellvsp $-$0.3367\end{tabular}
& \begin{tabular}{@{}c@{}}5,087,723 \cellvsp $-$0.3298\end{tabular} \\  
0.5 
& \begin{tabular}{@{}c@{}}2,380,185 \cellvsp $-$0.3464\end{tabular}
& \begin{tabular}{@{}c@{}}2,317,914 \cellvsp $-$0.3467\end{tabular}
& \begin{tabular}{@{}c@{}}2,570,641 \cellvsp $-$0.3466\end{tabular}
& \begin{tabular}{@{}c@{}}{\bf 2,081,426} \cellvsp $-$0.3465\end{tabular}
& \begin{tabular}{@{}c@{}}2,294,857 \cellvsp $-$0.3456\end{tabular}
& \begin{tabular}{@{}c@{}}2,461,136 \cellvsp $-$0.3442\end{tabular}
& \begin{tabular}{@{}c@{}}2,727,436 \cellvsp $-$0.3422\end{tabular}
& \begin{tabular}{@{}c@{}}2,513,794 \cellvsp $-$0.3384\end{tabular}
& \begin{tabular}{@{}c@{}}3,221,702 \cellvsp $-$0.3316\end{tabular} \\  
\rowcolor{myLightGray}
0.6 
& \begin{tabular}{@{}c@{}}2,259,538 \cellvsp $-$0.3411\end{tabular}
& \begin{tabular}{@{}c@{}}2,163,842 \cellvsp $-$0.3402\end{tabular}
& \begin{tabular}{@{}c@{}}1,719,454 \cellvsp $-$0.3413\end{tabular}
& \begin{tabular}{@{}c@{}}1,764,686 \cellvsp $-$0.3422\end{tabular}
& \begin{tabular}{@{}c@{}}1,897,407 \cellvsp $-$0.3423\end{tabular}
& \begin{tabular}{@{}c@{}}2,067,075 \cellvsp $-$0.3412\end{tabular}
& \begin{tabular}{@{}c@{}}{\bf 1,672,508} \cellvsp $-$0.3395\end{tabular}
& \begin{tabular}{@{}c@{}}1,935,515 \cellvsp $-$0.3370\end{tabular}
& \begin{tabular}{@{}c@{}}2,563,621 \cellvsp $-$0.3309\end{tabular} \\  
0.7 
& \begin{tabular}{@{}c@{}}2,146,095 \cellvsp $-$0.3350\end{tabular}
& \begin{tabular}{@{}c@{}}1,952,007 \cellvsp $-$0.3347\end{tabular}
& \begin{tabular}{@{}c@{}}2,000,424 \cellvsp $-$0.3363\end{tabular}
& \begin{tabular}{@{}c@{}}1,966,801 \cellvsp $-$0.3389\end{tabular}
& \begin{tabular}{@{}c@{}}$\mathbf{1,460,210^\star}$ \cellvsp $-$0.3383\end{tabular}
& \begin{tabular}{@{}c@{}}1,604,638 \cellvsp $-$0.3389\end{tabular}
& \begin{tabular}{@{}c@{}}1,740,662 \cellvsp $-$0.3386\end{tabular}
& \begin{tabular}{@{}c@{}}2,050,900 \cellvsp $-$0.3370\end{tabular}
& \begin{tabular}{@{}c@{}}1,855,200 \cellvsp $-$0.3305\end{tabular} \\  
\rowcolor{myLightGray}
0.8 
& \begin{tabular}{@{}c@{}}2,613,691 \cellvsp $-$0.3320\end{tabular}
& \begin{tabular}{@{}c@{}}2,613,691 \cellvsp $-$0.3304\end{tabular}
& \begin{tabular}{@{}c@{}}2,429,679 \cellvsp $-$0.3331\end{tabular}
& \begin{tabular}{@{}c@{}}{\bf 1,549,138} \cellvsp $-$0.3368\end{tabular}
& \begin{tabular}{@{}c@{}}1,634,584 \cellvsp $-$0.3375\end{tabular}
& \begin{tabular}{@{}c@{}}1,806,369 \cellvsp $-$0.3381\end{tabular}
& \begin{tabular}{@{}c@{}}1,977,211 \cellvsp $-$0.3380\end{tabular}
& \begin{tabular}{@{}c@{}}1,621,561 \cellvsp $-$0.3349\end{tabular}
& \begin{tabular}{@{}c@{}}2,093,208 \cellvsp $-$0.3304\end{tabular} \\  
0.9 
& \begin{tabular}{@{}c@{}}3,708,374 \cellvsp $-$0.3205\end{tabular}
& \begin{tabular}{@{}c@{}}3,151,928 \cellvsp $-$0.3189\end{tabular}
& \begin{tabular}{@{}c@{}}3,183,738 \cellvsp $-$0.3249\end{tabular}
& \begin{tabular}{@{}c@{}}2,071,551 \cellvsp $-$0.3288\end{tabular}
& \begin{tabular}{@{}c@{}}1,885,067 \cellvsp $-$0.3329\end{tabular}
& \begin{tabular}{@{}c@{}}2,470,591 \cellvsp $-$0.3287\end{tabular}
& \begin{tabular}{@{}c@{}}2,386,770 \cellvsp $-$0.3311\end{tabular}
& \begin{tabular}{@{}c@{}}{\bf 1,880,106} \cellvsp $-$0.3295\end{tabular}
& \begin{tabular}{@{}c@{}}2,439,044 \cellvsp $-$0.3254\end{tabular} \\  
\noalign{\hrule height 1.0pt}
\end{tabular}
\vspace{8pt}
\caption{
Computational cost (top of the cell) and empirical convergence rates (bottom of the cell)
for Algorithm~\ref{algorithm}.\ref{marking:D} applied to the parametric model problem
in Section~\ref{sec:numer:results}.
}
\label{tab:fulldata:D}
}
\end{center} 
\end{table}